\documentclass[11pt]{article}
\usepackage{graphicx} 

\usepackage[top=2.5 cm, bottom=2.5 cm, left=2cm, right=2cm] {geometry} 

\usepackage{amscd,amsmath,amstext,amsfonts,amsbsy,
amssymb,amsthm,mathrsfs,float,latexsym}
\usepackage{multicol,graphicx,array,multirow,color,lineno}
\usepackage{sidecap,url,fancybox,fancyhdr}
\usepackage[colorlinks=true,linkcolor=red,citecolor=cyan,pdfencoding=auto,psdextra]{hyperref}
\usepackage[utf8]{inputenc}
\usepackage[english]{babel}
\usepackage{enumitem}

\usepackage[dvipsnames]{xcolor}

\usepackage{mlmodern}
\usepackage[T1]{fontenc}

\allowdisplaybreaks

\DeclareMathOperator*{\esssup}{ess\,sup}
\newcommand{\pr}{\text{pr}}
\newtheorem{theorem}{Theorem}[section]
\newtheorem{lemma}{Lemma}[section]
\newtheorem{prop}{Proposition}[section]

\newtheorem{definition}{Definition}[section]
\newtheorem{remark}{Remark}[section]

\numberwithin{equation}{section}
\renewcommand{\epsilon}{\varepsilon}
\newcommand{\eps}{\varepsilon}

\newcommand{\ue}{u_\eps}
\newcommand{\ve}{v_\eps}

\usepackage{authblk}

\title{Slow Manifolds for PDE with Fast Reactions\\ and Small Cross Diffusion}

\author[1]{Laurent Desvillettes}
\author[2]{Christian Kuehn}
\author[2]{Jan-Eric Sulzbach}
\author[3]{Bao Quoc Tang}
\author[3]{Bao-Ngoc Tran}
\affil[1]{\small  Universit\'e Paris Cit\'e, Sorbonne Universit\'e, CNRS and IUF, IMJ-PRG, F-75005
Paris, France.\break
\href{mailto:desvillettes@imj-prg.fr}{desvillettes@imj-prg.fr} }
\affil[2]{\small Department of Mathematics, Technical University of Munich, Germany\break   
\href{mailto:ckuehn@ma.tum.de}{ckuehn@ma.tum.de}, \href{mailto:janeric.sulzbach@tum.de}{janeric.sulzbach@tum.de}}
\affil[3]{\small Department of Mathematics and Scientific Computing, University of Graz, Austria\break  
\href{mailto:quoc.tang@uni-graz.at}{quoc.tang@uni-graz.at}, \href{mailto:bao-ngoc.tran@uni-graz.at}{bao-ngoc.tran@uni-graz.at}}

\date{}

\begin{document}

\maketitle

\begin{abstract}
    Multiple time scales problems are investigated by combining geometrical and analytical approaches. More precisely, for fast-slow reaction-diffusion systems, we first prove the existence of slow manifolds for the abstract problem under the assumption that cross diffusion is small. This is done by extending the Fenichel theory to an infinite dimensional setting, where a main idea is to introduce a suitable space splitting corresponding to a small parameter, which controls additional fast contributions of the slow variable. These results require a strong convergence in $L^\infty(0,T;H^2(\Omega))$, which is the subtle analytical issue in fast reaction problems in comparison to previous works. By considering a nonlinear fast reversible reaction in one dimension, we successfully prove this convergence and therefore obtain the slow manifold for a PDE with fast reactions. Moreover, the obtained convergence also shows the influence of the cross diffusion term and illuminates the role of the initial layer. Explicit approximations of the slow manifold are also carried out in the case of linear systems.
\end{abstract}

\tableofcontents
\section{Introduction}

The study of systems with multiple time scales is a crucial topic in dynamical systems. 
This has led to the geometric singular perturbation theory for fast-slow systems first developed in finite dimensions,
see e.g. \cite{fenichel1979geometric,wiggins1994normally,MR1374108jones,MR1718893Kaper,kuehn2015multiple}.
This approach addresses problems with distinct time scales using a geometric perspective. It leverages invariant manifolds in phase space to analyze the global structure and to construct orbits with specific desired properties.
For the fast-slow PDE systems discussed in this paper we extend the existing slow manifold theory, where the foundations of this approach were established in \cite{MR55515tihonov} and \cite{fenichel1979geometric}.
The idea of the slow manifold theory is to reduce the complexity of the fast-slow system while maintaining the geometric structure and properties of the system and its solution. 
Later these methods were extended to infinite dimensions, first in \cite{bates1998existence} and more recently in a series of papers \cite{hummel2022slow,kuehn2023fast,kuehn2024infinite}.

\medskip
For PDEs, one may encounter systems with multiple time scales in the context of chemical reaction-diffusion systems where fast reactions are present. The literature concerning fast reaction limits for PDEs has also flourished with many notable results, starting from the eighties \cite{evans1980convergence,martin1980mathematical} to recent numerous results \cite{iida2006diffusion,iida2017vanishing,bothe2010quasi,bothe2012cross,perthame2022fast} and references therein. In this line of works, many interesting connections between reaction-diffusion systems and limit systems have been revealed, ranging from nonlinear diffusion equations \cite{evans1980convergence,bothe2010quasi}, free boundary problems \cite{martin1980mathematical,murakawa2011fast}, cross-diffusion systems \cite{bothe2012cross,daus2020cross,brocchieri2021evolution}, dynamic boundary conditions \cite{henneke2016fast}, Michaelis-Menten kinetics \cite{tang2024rigorous}, or systems involving Young measures \cite{perthame2022fast}. Most of these works are concerned with the convergence analysis of solutions, meaning that solutions of the fast reaction systems are shown to tend to that of the limit systems by using tools from functional analysis and PDEs. In this paper, we aim to combine the geometric theory and the functional analytic framework by studying slow manifolds for fast reaction limits with small cross diffusion.

\medskip
To describe our problem, let us consider for instance the following reversible reaction
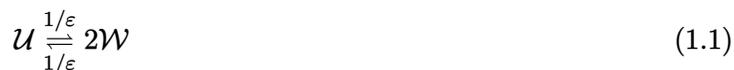
\begin{equation}\label{reactions}
    \mathcal{U} \overset{1/\eps}{\underset{1/\eps}{\rightleftharpoons}} 2\mathcal W
\end{equation}
for two chemical substances $\mathcal U$ and $\mathcal W$, where $1/\eps$ for $\eps>0$ represents and forward and backward reaction rate constants. Assume that the reaction takes place in a bounded vessel $\Omega\subset\mathbb R^n$ with smooth boundary $\partial\Omega$. By taking into account spatial diffusion, possible sources and homogeneous Neumann boundary conditions, the corresponding chemical reaction-diffusion system is given by (after suitable rescaling)
\begin{equation}\label{sys1}
    \begin{cases}
        \partial_t u - d_1\Delta u = \eps^{-1}(-u + w^2) + \tilde{\phi}(u,w), &x\in\Omega, \, t>0,\\
        \partial_t w - d_2\Delta w = -\eps^{-1}(-u + w^2) + \tilde{\psi}(u,w), &x\in\Omega, \, t>0,\\
        \nabla u \cdot \nu = \nabla w\cdot \nu = 0, &x\in\partial\Omega, \, t>0,\\
        u(x,0) = u_0(x), w(x,0) = w_0(x), &x\in\Omega,
    \end{cases}
\end{equation}
where $d_1, d_2>0$ are diffusion coefficients, $\nu$ is normal outward unit vector on $\partial\Omega$, and $\tilde \phi$, $\tilde \psi$ are given sources in the system. We are interested in the fast reaction limit, i.e. when the reaction rate constant $1/\eps$ tends to $\infty$. By introducing the variable $v = u + w$, and the function $f(u,v) = (v - u)^2$, $\phi(u,v) = \tilde \phi(u, v-u)$, $\psi(u,v) = \tilde \phi(u,v-u) + \tilde \psi(u,v-u)$, the system \eqref{sys1} can be rewritten as
\begin{equation}\label{sys2}
\begin{cases}
    \partial_t u - d_1\Delta u = \eps^{-1}(-u + f(u,v)) + \phi(u,v), &x\in\Omega, \, t>0,\\
    \partial_t v - d_2\Delta v + (d_2-d_1)\Delta u = \psi(u,v), &x\in\Omega, \, t>0,\\
    \nabla u\cdot \nu = \nabla v \cdot \nu = 0, &x\in\partial\Omega, \, t>0,\\
    u(x,0) = u_0(x), v(x,0) = u_0(x) + w_0(x), &x\in\Omega.
\end{cases}
\end{equation}
Denote by $(\ue,\ve)$ the solution to \eqref{sys2} corresponding to $\eps>0$.  In the form \eqref{sys2}, the term $(d_2-d_1)\Delta u$ in the second equation plays the role of a \textit{cross diffusion} term. In the limit $\eps\to 0$, it is expected that $\ue \to u$, $\ve \to v$ and $-\ue + f(\ue, \ve) \to 0$ for some suitable $u, v$. This leads to $-u + f(u,v) = 0$. From this one can solve $u = \varphi(v)$. Then the limit system
can be rewritten as the semi-linear reaction-diffusion system with nonlinear diffusion
\begin{equation}\label{sys4}
\begin{cases}
    \partial_t v - \Delta [d_2v - (d_2 - d_1)\varphi(v)] = \psi(\varphi(v), v), &x\in\Omega,\, t>0,\\
    \nabla v \cdot \nu  = 0, &x\in\partial\Omega,\, t>0,\\
    v(x,0) = u_0(x) + w_0(x), &x\in\Omega.
\end{cases}
\end{equation}
The convergence analysis of \eqref{sys1} towards \eqref{sys4} has been shown in \cite{bothe2010quasi}. In this work, we are concerned with the geometric perspective of fast reaction limits. 
In fact, we study a more general setting, which reads as
\begin{equation} \label{general system}
    \left\{
    \begin{aligned}
        &\partial_t u = (d+\delta(\varepsilon))A u +\frac{1}{\varepsilon}g(u,v) +\phi(u,v),\\
        &\partial_t v= d A v+ \delta(\varepsilon) A u +\psi(u,v),\\
        &u(0)=u_0,\quad v(0)=v_0,
    \end{aligned}
    \right.
\end{equation}
where $A$ is a linear (differential) operator on a Banach space $X$ and $g,\phi$ and $\psi$ are nonlinear functions. By assuming certain conditions on the operator $A$, the nonlinearities, the convergence of trajectories and the limit $\lim_{\eps\to 0}\delta(\eps) = 0$, we will show that there exists a family of slow manifolds $S_\varepsilon$ for \eqref{general system}, which approach the critical manifold $S_0 = \{(u,v)\in X\times X: g(u,v) = 0\}$ as $\eps \to 0$. 
The main idea is to introduce a second small parameter $\zeta$, which controls additional \textit{fast} contributions of the $v$-dynamics.
For the construction of the slow manifolds, we need a condition that requires the spectral gaps of the operator to be of a certain size in relation to the Lipschitz constants of the nonlinearities.
Here, the influence of the cross-diffusion term comes into play as it can change this spectral gap condition as observed in \cite{kostianko2021kwak}. 

\medskip
Coming back to the fast reaction limit for \eqref{sys1}, one of the crucial points for the existence of slow manifolds is the convergence of $(\ue, \ve)$ to $(u,v)$ in the strong topology of $L^\infty(0,T;H^2(\Omega))$. Such a strong convergence is difficult to obtain for fast reaction problems (see e.g. \cite{iida2006reaction} for the case where this convergence is proved provided some additional restrictive assumptions). In this paper, we prove this strong convergence in $L^\infty(0,T;H^2(\Omega))$ in one dimension and small cross-diffusion without artificially imposing any further estimates on the approximate solutions $(\ue,\ve)$. To do this, first, we utilize the system's structure to obtain some useful uniform-in-$\eps$ estimates. These estimates are then used in an energy-type argument for the difference $\ue - u$ and $\ve - v$ to show first the strong convergence in $L^\infty(0,T;L^2(\Omega))$. Note that here, we see the influence of the initial layer, i.e. the distance from initial data to the critical manifold, and the cross diffusion term. Then this convergence is exploited in a bootstrap manner to get the convergence in $L^\infty(0,T;H^1(\Omega))$ and eventually in $L^\infty(0,T;H^2(\Omega))$, where the last part is quite technical as we have to deal with a second-order energy function. Particularly, the latter is essential in obtaining the Lipschitz continuity of the nonlinearities, which is strictly required in the abstract setting. It is emphasized that we also obtain the convergence rate as $\eps\to 0$, which can be very useful for practical purposes.  

 \medskip
It is remarked that, in general, it is difficult to explicitly compute slow manifolds. Nevertheless, we demonstrate for linear reversible reaction with fast reactions
\begin{equation*}
    \mathcal U \overset{1/\eps}{\underset{1/\eps}{\leftrightharpoons}} \mathcal V,
\end{equation*}
that it is possible to explicitly compute slow manifolds. More precisely, this is done by utilizing the Fourier decomposition using the basis consisting of the eigenfunctions of the differential operator.
Moreover, we can use this approach to construct approximate slow manifolds that are sufficiently close to the original ones. 

\medskip
\textbf{Organization of the paper.} In the next section \ref{section:2} we present in detail the assumptions and main results of this paper, namely the existence of slow manifolds for the abstract problem \eqref{general system} and the convergence as well as slow manifolds for fast reaction system \eqref{sys1}. The proofs of these two main results are presented in Section \ref{section:3} and \ref{section:4} respectively. Finally, in section \ref{section:5}, we explicitly approximate the slow manifolds in case of linear reversible system.

\section{Problem Setting and Main results}\label{section:2}

\subsection{Abstract Existence Theory}\label{Sec:Assumptions}

We first investigate the underlying \emph{geometric structure} of a general fast reaction system with small cross diffusion term in an abstract Banach space setting \eqref{general system}. Taking the formal limit in~\eqref{general system} as $\varepsilon\to 0$ we obtain the limiting (or reduced) system
\begin{equation}\label{limit general system}
\left\{
    \begin{aligned}
        0&=g(u,v),\\
        \partial_t v&= d A v+ \delta(0) A u +\psi(u,v),\\
        u(0)&=u_0,\quad v(0)=v_0.
    \end{aligned}
    \right.
\end{equation}
This system is a differential-algebraic equation, where the algebraic constraint is given by the critical manifold $S_0$ defined by 
\begin{align}\label{critical manifold}
    S_0:=\{ (u,v)\in X_1\times X_1: g(u,v)=0\} \subset X\times X,
\end{align}
where the Banach space $X_1$ will be defined below.
Depending on the regularity of $g$ the critical manifold can also be embedded into $X_1\times X_1$.
We note that, in general, $S_0$ does not have to define a manifold but we shall always assume here that it is actually a (Banach) manifold. The goal is to prove a generalization of Fenichel's theorem from finite-dimensional dynamical systems, where it is a cornerstone of the geometric singular perturbation theory, to the infinite-dimensional setting of system \eqref{general system} illustrating its application to cross-diffusion systems. The key notion in Fenichel's theory is the so called slow manifold, which can be viewed as an $\varepsilon$ perturbation of the critical manifold. 
\begin{definition}\label{Def.slow.manifold}
    We call $S_\varepsilon \subset X_1\times X_1$ an attracting slow manifold if there exists an $\varepsilon_0>0$ such that the following properties for all $\varepsilon\in (0,\varepsilon_0]$
    \begin{itemize}
    \item It is a locally invariant and an (at least locally) exponentially attracting manifold.
    \item The slow manifold has a (Hausdorff semi-)distance of $\mathcal{O}(\varepsilon)$ to the critical manifold $S_0$ as $\varepsilon\rightarrow 0$ in the space $X_\alpha\times X_1$.
    \item If $S_0$ is Lipschitz (or even $C^k$-smooth for some $k\in \mathbb{N}$), then $S_\varepsilon$ is also Lipschitz (or $C^k$-smooth respectively).
    \item The semi-flow on the slow manifold $S_\varepsilon$ converges in $X_1\times X_1$ to the semi-flow on the critical manifold $S_{0}$. 
\end{itemize}
\end{definition}

\medskip
The existence of slow manifolds for the abstract problem \eqref{general system} requires certain assumptions, which we will list below, followed by some discussion on the assumptions.

\medskip
\noindent\underline{\textit{Assumptions on the system}:}
\begin{enumerate}[label=(S\theenumi),ref=S\theenumi]
    \item \label{S1}  $\lim_{\varepsilon\to 0} \delta(\varepsilon)=\delta(0)=0$, i.e. the cross diffusion term is small.
    \item  \label{S2} $X_1= D(A)$ is dense in $X$.
    \item  \label{S3} The initial data satisfies $(u_{0},v_0)\in X_1\times X_1\cap S_0$, i.e. it lies on the critical manifold.
\end{enumerate}

\medskip
\noindent\underline{\textit{Assumptions on the operator $A$}:}
\begin{enumerate}[label=(A\theenumi),ref=A\theenumi]
    \item \label{A1} $A$ is a closed linear operator, where $A:X\supset D(A)\to X$ generates an exponentially stable analytic $C_0$-semigroup $(\textnormal{e}^{t A})_{t\geq 0}\subset \mathcal{B}(X)$, where $\mathcal{B}(X)$ denotes the space of bounded operators on $X$. The interpolation-extrapolation scales generated by $(X,A)$ are denoted by $(X_{\alpha},A^\alpha)_{\alpha\in [-1,\infty)}$.
     \item\label{A2} There are constants $C_{A},K_A>0$, $\omega_{A} <0 $ such that
    \begin{align*}
        \|\textnormal{e}^{t{A}}\|_{\mathcal{B}(X_1)}&\leq C_{A} \textnormal{e}^{\omega_{A} t},\\
        \|t A \textnormal{e}^{t A}\|_{\mathcal{B}(X_1)}&\leq K_{A} \textnormal{e}^{\omega_{A} t}
    \end{align*}
    hold for all $t\in (0,\infty)$. 
\end{enumerate}
\begin{remark}
    Moreover, applying the perturbation result for semigroups it holds that
    \begin{align*}
         \|\textnormal{e}^{t(A-\varepsilon^{-1} \textnormal{Id})}\|_{\mathcal{B}(X_1)}&\leq C_{A} \textnormal{e}^{(\omega_{A}-\varepsilon^{-1}) t},
    \end{align*}
    where for $\varepsilon>0$ small enough the expression $\omega_A-\varepsilon^{-1}<0$ and thus the perturbation generates a contractive semigroup.
\end{remark}

An operator $A$ satisfying these assumptions is for example the standard Laplacian $\Delta$ with zero boundary condition on the Banach space $X=L^2(\Omega)$ with $X_1=D(\Delta)= H^2(\Omega)\cap H^1_0(\Omega)$, where $\Omega$ is a bounded domain. We refer to \cite{engel2000one} for more details and further examples.

\medskip
\noindent\underline{\textit{Assumptions on the nonlinearities}:}
\begin{enumerate}[label=(N\theenumi),ref=N\theenumi]
\item\label{N1} $g(0,0)=0$, $\phi(0,0)=0$ and $\psi(0,0)=0$, and $g(u,v)=-u+f(u,v)$ for a mapping $f: \mathbb R^2 \to \mathbb R$;
    \item\label{N2} the nonlinearities $f,\phi,\psi:X_1\times X_1 \to X_1$ are continuously Fr\'echet differentiable 
    and there are constants $L_\phi,L_\psi>0$ and $0<L_f<1$ such that
\begin{align*}
       \|\textnormal{D} f(x,y)\|_{\mathcal{B}(X_{1}\times X_{1},X_1)}&\leq L_f,\qquad  \|\textnormal{D} \phi(x,y)\|_{\mathcal{B}(X_{1}\times X_{1},Y_1)}\leq L_\phi, \\
       \|\textnormal{D} \psi(x,y)\|_{\mathcal{B}(X_{1}\times X_{1},X_1)}&\leq L_\psi
\end{align*}
  for all $x,y\in X_1$. 
    \item 
    Let $\mathcal{C}\subset X\times X$ be a non-empty set such that for all $(x,y)\in \mathcal{C}$ we have $g(x,y)= 0$ and in addition we assume that the operator $\textnormal{D}_x g(x,y)\equiv \textnormal{D}  g(\cdot,y)x: X_1\to X_1$ has its spectrum in the left half plane with an upper bound $\lambda_0<0$ on the real part and in particular is a Banach space isomorphism, i.e. a bounded invertible linear operator.   
    
\end{enumerate}
\begin{remark}
Then for each $(x_0,y_0) \in\mathcal{C}$ the implicit function theorem holds. 
That is, there exist open neighborhoods $V_{y_0}$ of $y_0$ and $U_{x_0}$ of $x_0$ and a Fr\'echet differentiable function $h^{y_0} :V_{y_0}\to U_{x_0}$ such that $g(h^{y_0}(y_0),y_0)=0$. Then $h^{y_0}$ satisfies $h^{y_0}(y)= x$ and $g(h^{y_0}(y),y)=0$ for all $(x,y)\in U_{x_0}\times V_{y_0}$.
Moreover $h^{y_0}$ has the same regularity as $g$ and is Lipschitz continuous with constant $L_{h^0}$.
Indeed, using the special form of $g$ we have
\begin{align*}
    \| h^0(v_1)-h^0(v_2)\|_{X_1}&= \|f(h^0(v_1),v_1)-f(h^0(v_2),v_2)\|_{X_1}\\
    &\leq L_f \big( \| h^0(v_1)-h^0(v_2)\|_{X_1} +\|v_1-v_2\|_{X_1}\big)
\end{align*}
and thus
\begin{align*}
    \| h^0(v_1)-h^0(v_2)\|_{X_1} \leq \frac{L_f}{1-L_f}\|v_1-v_2\|_{X_1},
\end{align*}
where $L_{h^0}=\frac{L_f}{1-L_f}$.
Here, we observe that the condition $L_f<1$ is indeed crucial.
For more details on the implicit function theorem in Banach spaces we refer to \cite{pata2019fixed}.
\end{remark}
\begin{enumerate}[label=(N\theenumi),ref=N\theenumi]
    \setcounter{enumi}{3}
    \item\label{N4} In order to keep the computations as clear as possible and to avoid dealing with chart transforms, we restrict the nonlinearity $g$ to the case of one neighborhood $(U_0,V_0)$ covering $\mathcal{C}$.
\end{enumerate}

The last assumption avoids the need for a local patching construction gluing different local charts together. On a compact manifold this can always be achieved but just makes the notation of the abstract setting more technical. See \cite{wiggins1994normally} for details.

\begin{remark}\hfill

With all the assumptions above, the existence of solutions to the above systems follows from the semigroup theory for semi-linear parabolic problems.
We refer to \cite{hummel2022slow,kuehn2023fast} and \cite{lunardi2012analytic} for more details. 
\end{remark}

\medskip
\noindent\underline{\textit{Assumptions on splitting the phase space $X$}:} The next assumptions are crucial to show the existence of a slow manifold, as they allow to identify the fast and slow components of the variables.
To this end we introduce a splitting of the phase space $X$ as
\begin{align*}
    X= X^\zeta_F\oplus X^\zeta_S
\end{align*}
into a fast part $X^\zeta_F$ and a slow part $X^\zeta_S$, where  $\zeta>0$ is a small parameter, such that 
\begin{enumerate}[label=(X\theenumi), ref=X\theenumi]
    \item\label{X1} The spaces $X^\zeta_F$ and $X^\zeta_S$ are closed in $X$ and the projections $\pr_{X^\zeta_F}$ and $\pr_{X^\zeta_S}$ commute with $A$ on $X_1$. The spaces $X^\zeta_F\cap X_1$ and $X^\zeta_S\cap X_1$ are closed subspaces of $X_1$ and are endowed with the norm of $(X_1,\|\cdot\|_{X_1})$.
    \item\label{X2} The realization of $A$ in $X^\zeta_F$, i.e. 
    \begin{align*}
        A_{X^\zeta_F}:D(X^\zeta_F) \subset X^\zeta_F\to X^\zeta_F,\quad v\mapsto A v
    \end{align*}
    with $D(X^\zeta_F):= \{v\in X^\zeta_F\,:\,A v\in X^\zeta_F\}$ has $0$ in its resolvent set. The realization of $A$ in $X^\zeta_S$ generates a $C_0$-group $(\textnormal{e}^{t A_{X^\zeta_S}})_{t\in \mathbb{R}}\subset \mathcal{B}((X^\zeta_S,\|\cdot\|_{X}) )$ which for $t> 0$ satisfies $\textnormal{e}^{t A_{X^\zeta_S}}=\textnormal{e}^{t A}$ on $X^\zeta_S$.
    \item\label{X3}
    For $\zeta>0$ satisfying $\eps\zeta^{-1} \le c < 1$ there exist constants $N_F^\zeta, N_S^\zeta >0$ satisfying
    $$0\leq N_F^\zeta< N_S^\zeta<-\zeta^{-1}(\varepsilon(d+\delta)\omega_A-1)$$
    such that for all $t> 0$ and $x_F\in X^\zeta_F,\, x_S\in X^\zeta_S$ we have that the splitting of the fast and slow subspaces can be classified via the dichotomy
    \begin{align*}
        \|\textnormal{e}^{td A}x_F\|_{X_1}&\leq C_A \textnormal{e}^{(N_F^\zeta +\zeta^{-1}(\varepsilon(d+\delta)\omega_A-1))t}\|x_F\|_{X_1},\\
        \|\textnormal{e}^{-t dA}x_S\|_{X_1}&\leq M_A \textnormal{e}^{-(N_S^\zeta +\zeta^{-1}(\varepsilon(d+\delta)\omega_A-1))t}\|x_S\|_{X_1},
    \end{align*}
    where $C_A,M_A>0$.      
\end{enumerate}
Assumption \eqref{X3} implies that the following inequality holds true
    \begin{align}\label{parameter inequality}
        (1-\varepsilon \zeta^{-1}) (\varepsilon(d+\delta) \omega_A-1) -\varepsilon \frac{N_S^\zeta+N_F^\zeta}{2}<0.
    \end{align}
    The parameters $N_S^\zeta,\,N_F^\zeta$ in \eqref{X3} are related to the spectrum of the operator $A$ via their difference $N_S^\zeta-N_F^\zeta$.
    This measures how far one can separate the decay properties of the fast and the slow part in the slow variable.
\begin{remark}
    Formulated differently assumption \eqref{X3} implies that the fast subspace $X^\zeta_F$ contains the parts of $X_1$ that decay under the semigroup $(\textnormal{e}^{td A})_{t\geq 0}$ almost as fast as functions under the semi-group $(\textnormal{e}^{t((d+\delta)A-\epsilon^{-1}\textnormal{Id})})_{t\geq 0}$ generated by the modified operator $(d+\delta)A-\epsilon^{-1}\textnormal{Id}$.
    The space $X^\zeta_S$ on the other hand contains the parts of $X_1$ that do not decay or which decay only slowly under the semigroup $(\textnormal{e}^{td A})_{t\geq 0}$ compared to $X_1$ under the action of $(\textnormal{e}^{t((d+\delta)A-\epsilon^{-1}\textnormal{Id})} )_{t\geq 0}$.
\end{remark}
    \medskip
    \noindent\underline{\textit{Assumptions on spectral gap and Lipschitz constant}:}
\begin{enumerate}[label=(G), ref=G]
    \item\label{G} The following spectral gap condition holds
    \begin{align*}
   &\frac{C_ A (L_f+\varepsilon L_\phi)}{| (1-\varepsilon \zeta^{-1})(\varepsilon (d+\delta)\omega_A-1) - \varepsilon\frac{1}{2}(N_S^\zeta+N_F^\zeta)|} +  2 \frac{ \delta (C_A+d^{-1} C_A^2)(  L_f + \varepsilon L_\phi) +2C_A\varepsilon L_\psi  }{\varepsilon|N_S^\zeta-N_F^\zeta|}  <1.
\end{align*}
\end{enumerate}
The appearance of the spectral gap condition is not surprising as by the definition of the slow manifold they are connected to the theory of inertial manifolds, and one condition for the existence of an inertial manifold in a dissipative system is the existence of a spectral gap. We refer to \cite{zelik2014inertial} and the references therein.

\medskip
\noindent\underline{\textit{Assumptions on the convergence of trajectories}:}
\begin{enumerate}[label=(C),ref=C]
    \item\label{C} There exists a time $T>0$ such that 
    \begin{align}\label{convergence of trajectories general}
    \left\|\begin{pmatrix}
        u^\varepsilon(t)-h^0(v^0(t))\\ v^\varepsilon(t)-v^0(t)
    \end{pmatrix}\right\|_{X_1\times X_1} \leq C(\varepsilon,T)\big( \|u_0-h^0(v_0)\|_{X_1}+\|v_0\|_{X_1}\big)
\end{align}
holds for all $t\in [0,T]$, and where $C(\varepsilon,T)\to 0$ as $\varepsilon \to 0$.
\end{enumerate}

We are now ready to state the first main result of our paper.
\begin{theorem}\label{thm:abstract slow manifold}
    Let the assumptions of \eqref{S1}--\eqref{S3}, \eqref{A1}--\eqref{A2}, \eqref{N1}--\eqref{N4}, \eqref{X1}--\eqref{X3}, \eqref{G} and \eqref{C} hold. Then, there exists a family of slow manifolds $\{S_\eps\}_{\eps\in (0,\eps_0]} \subset X_1\times X_1$ of system \eqref{general system} in the sense of Definition \ref{Def.slow.manifold}.  
\end{theorem}

The reduced dynamics on the slow manifold are then given by
\begin{align}\label{eq:reduced dynamics}
\begin{split}
    \partial_t v_S&= dA v_S+\delta(\varepsilon)A h^{\varepsilon,\zeta}_u(v_S) +\pr_{X_S^\zeta}\psi\big(h^{\varepsilon,\zeta}_u(v_S),h^{\varepsilon,\zeta}_{X_F^\zeta}(v_S),v_S\big),\\
    v_S(0)&=  \pr_{X_S^\zeta} v_0,
    \end{split}
\end{align}
where $v_S=  \pr_{X_S^\zeta} v$ and $h^{\varepsilon,\zeta}= (h^{\varepsilon,\zeta}_u,h^{\varepsilon,\zeta}_{X_F^\zeta}) $ are the first two components of the slow manifold $S_{\varepsilon}$.
\subsection{Slow Manifolds for Fast Reversible Reactions}

We show how the theory developed in Theorem \ref{thm:abstract slow manifold} can be applied to a fast reversible reaction limit problem related to \eqref{reactions}. Consider the case $\Omega = I$, an open bounded interval in $\mathbb{R}$ with the boundary $\Gamma = \partial I$. For any $0<T\le \infty$, we denote  $I_T:=I \times (0,T)$. We rewrite \eqref{sys2} as
\begin{align} \label{Problem.Original} 
\left\{
\begin{array}{llllllll}
\partial_t u^\varepsilon = (d+\delta)  \partial_{xx}^2 u^\varepsilon  + \phi(u^\varepsilon,v^\varepsilon)  + \dfrac{1}{\varepsilon}  (-u^\varepsilon+ \kappa \tilde f(u^\varepsilon,v^\varepsilon)) & \text{in } I_\infty, \vspace{0.15cm}  \\
\partial_t v^\varepsilon = d \partial_{xx}^2 v^\varepsilon + \psi(u^\varepsilon,v^\varepsilon) + \delta \partial_{xx}^2 u^\varepsilon & \text{in } I_\infty,  
\end{array}
\right.  
\end{align}
subject to the Neumann boundary conditions 
\begin{align}
    \partial_{\nu} u^\varepsilon = \partial_{\nu} v^\varepsilon = 0 \quad  \text{on } \Gamma \times (0,\infty)
    \label{Condition.Boundary}
\end{align}
and the initial condition 
\begin{align}
    (u^\varepsilon(0) ,v^\varepsilon(0)) = (u_{\mathrm{in}},v_{\mathrm{in}})  \quad  \text{on } I.
    \label{Condition.Initial}
\end{align}
Here, $\kappa>0$ is a small parameter and   
\begin{align*}
\tilde f(x,y)=(y - x)^2, \quad 0 \le x\le y <\infty, 
\end{align*}
the couple $(u_{\mathrm{in}},v_{\mathrm{in}})$ stands for the initial data, which will be assumed to satisfy some assumptions. The nonlinearities $\phi,\psi$ are given by
\begin{align*}
\phi(x,y)&=(a-bx-cy)x, \\
\psi(x,y)&=(a-bx-cy)y,
\end{align*} 
for $x,y\in [0,\infty)$ in which $a,b,c$ are positive constants, which are of Lotka-Volterra competitive type. The parameter $\varepsilon>0$ corresponds to a sufficiently small time-scale separation, i.e. $\varepsilon \ll 1$. The coefficients $d$, and $d+\delta$ are the diffusion rates, where $d>0$ does not depend on $\varepsilon$, and   $\delta=\delta(\varepsilon)>0$ is assumed to be such that 
\begin{align}
\lim_{\varepsilon\to 0} \delta = \lim_{\varepsilon\to 0} \delta(\varepsilon) = 0.
\label{Assumption.Delta}
\end{align}
As $\varepsilon\to 0$, we expect at the formal level that 
\begin{align}
\left\{ \begin{array}{llll}
(u^\varepsilon,v^\varepsilon) \to (u,v) & \text{as } \varepsilon \to 0,    \vspace*{0.15cm} \\
  -u^\varepsilon + \kappa (v^\varepsilon - u^\varepsilon)^2 \to 0 & \text{as } \varepsilon \to 0,
\end{array}  
\right.
\label{Limit.CriticalManifold}
\end{align}
as well as $(d+\delta) \partial_{xx}^2 u^\varepsilon \to d \partial_{xx}^2 u$. Then, the couple $(u,v)$ satisfies the system 
\begin{align} \label{Problem.Limiting} 
\left\{
\begin{array}{llllllll}
-u+ \kappa (v-u)^2 = 0& \text{in } I_\infty, \vspace{0.15cm}  \\
\partial_t v - d \partial_{xx}^2 v = \psi (u,v)  & \text{in } I_\infty,
\end{array}
\right.  
\end{align}
subject to the Neumann boundary and initial conditions
\begin{align}
    \partial_{\nu} v = 0 \quad  \text{on } \Gamma \times (0,\infty), \quad \text{and} \quad
    v(0) = v_{\mathrm{in}}  \quad  \text{on } I.
    \label{Condition.Initial.Limiting}
\end{align} 
The algebraic equation of \eqref{Problem.Limiting} forms a critical manifold $S_0$ in $H^2(I)\times H^2(I)$ for a suitable range of $\kappa$. We will show the existence of a family of slow manifolds $S_\varepsilon$ for the fast reaction system \eqref{Problem.Original}-\eqref{Condition.Initial}, where $S_\varepsilon$ is regular with a distance of $\mathcal O(\varepsilon)$ to $S_0$ and the semi-flow on $S_\varepsilon$ converges to the semi-flow on $S_0$ with the rate of order $\mathcal O(\varepsilon)$. These results are based heavily on some strong convergence of trajectories in $L^{\infty}(0,T;H^2(I))$. The second main result in this paper is the following theorem.
\begin{theorem}\label{thm:fast_reaction}
    Assume that  $(u_{\mathrm{in}},v_{\mathrm{in}}) \in B_M(0,0)\subset H^2(I)\times H^2(I)$ with the radius $M>0$ satisfying   
    \begin{align}
       12C_*K_M < \frac{1}{\kappa},
    \end{align}
where $C_*$ and $K_M$ are defined in \eqref{SharpSEconst} and \eqref{K:Def}. Then, there exists $C_T$ independent of $\eps, \eps_{\mathrm{in}}, \delta$ such that 
\begin{align*}
\|(u^\varepsilon,v^\varepsilon)-(u,v)\|_{L^\infty(0,T;L^2(I))^2} 
 + \|(u^\varepsilon,v^\varepsilon)-(u,v)\|_{L^2(0,T;H^1(I))^2} 
\le C_{T} (   \varepsilon + \varepsilon_{\mathrm{in}}  +  \delta  )
\end{align*}
where 
\begin{align}
     \varepsilon_{\mathrm{in}} := \|-u_{\mathrm{in}} + \kappa (v_{\mathrm{in}}-u_{\mathrm{in}})^2\|_{H^2(I)}
    \label{Assumption.EpsilonIN}
\end{align} 
represents the initial layer. Assume furthermore that 
\begin{equation}
    \lim_{\varepsilon \to 0} \left( \frac{\varepsilon_{\mathrm{in}}^2}{\varepsilon}  +  \frac{\delta^2}{\varepsilon} \right) < \infty
\quad \text{and} \quad 
\lim_{\varepsilon\to 0}\frac{\delta}{\varepsilon \zeta^{-1/2}}=0,
\label{ParaAss}
\end{equation}
where $\zeta$ is the splitting parameter in Section 2. 
Then, for any $T>0$, 
    \begin{align}\label{ee}
    \|(u^\varepsilon,v^\varepsilon)-(u,v)\|_{L^\infty(0,T;H^2(I))^2}  \le C_{T} (\varepsilon +  \varepsilon_{\mathrm{in}} + \delta ),
\end{align} 
and consequently, there exists a family of slow manifolds $\{S_\eps\}_{\eps\in (0,\eps_0]}$,  for \eqref{Problem.Original} in $H^2(I)\times H^2(I)$ in the sense of Definition \ref{Def.slow.manifold}, for some $\varepsilon_0>0$.
\end{theorem}
\begin{remark}
    It is remarked that the convergence \eqref{ee} can be proved using $\delta = O(\sqrt{\eps})$. While to have the last limit in \eqref{ParaAss}, we need $\delta$ being in the order $\varepsilon^{1+\gamma}$ for some $\gamma>0$, and $\zeta$ being in the order $\varepsilon^{1+\gamma'}$ with $\gamma'>\gamma$, 
    is needed to show the existence of slow manifolds, which result from the assumption concerning spaces, the splitting and spectral gap in the abstract result.
\end{remark}

\section{Existence of Slow Manifolds}\label{section:3}
This section is devoted to prove Theorem \ref{thm:abstract slow manifold}. We first rewrite the general fast reaction system \eqref{general system} with splitting in the slow variable, as defined in Section \ref{Sec:Assumptions}, given by
\begin{equation}\label{general system split}
    \left\{
    \begin{aligned}
        \partial_t u^\varepsilon &= (d+\delta)A u^\varepsilon -\frac{1}{\varepsilon}u^\varepsilon +\frac{1}{\varepsilon}f(u,v_F^\varepsilon,v_S^\varepsilon)+\phi(u,v_F^\varepsilon,v_S^\varepsilon),\\
        \partial_t v_F^\varepsilon &= d A v_F^\varepsilon+ \delta A \pr_{X_F^\zeta} u^\varepsilon+\pr_{X_F^\zeta}\psi(u,v_F^\varepsilon,v_S^\varepsilon) ,\\
         \partial_t v_S^\varepsilon &=d A v_S^\varepsilon+ \delta A \pr_{X_S^\zeta} u^\varepsilon+\pr_{X_S^\zeta}\psi(u,v_F^\varepsilon,v_S^\varepsilon) ,\\
          u^\varepsilon(0)&=u_0,\quad     v^\varepsilon_F(0)=\pr_{X_F^\zeta} v_0,\quad v^\varepsilon_S(0)=\pr_{X_S^\zeta} v_0.
    \end{aligned}
    \right.
\end{equation}
We then construct a locally invariant manifold to system \eqref{general system split} using a fixed point argument. Afterwards, we show that this manifold satisfies the conditions of Definition \ref{Def.slow.manifold} and therefore is indeed the desired slow manifold.

\subsection{Lyapunov-Perron Operator}

Now, the aim is to construct a family of slow manifolds $S_{\varepsilon,\zeta}$ as 
\begin{align*}
    S_{\varepsilon,\zeta}:=\{(h^{\varepsilon,\zeta}(v_0),v_0):\, v_0\in X_1\cap X_S^\zeta\},
\end{align*}
where the function $h^{\varepsilon,\zeta}$ maps
\begin{align*}
    h^{\varepsilon,\zeta}: (X_S^\zeta\cap X_1)\to  X_1
    \times(X_F^\zeta\cap X_1).
\end{align*}
This is done by proving that the Lyapunov-Perron operator has a unique fixed point, where the function $h^{\varepsilon,\zeta}$ is then obtained by evaluating this fixed point at time $t=0$.

A key part in the Lyapunov-Perron method is understanding the dynamics of the fast and slow variables of the linear part of the system at hand.
Hence, we define a linear differential operator for \eqref{general system split} that takes into account the linear operator $A$ and the leading linear part of the fast-reaction term, by
\begin{align}
    \mathcal{A}^\varepsilon=\begin{pmatrix}
        (d+\delta)A -\frac{1}{\varepsilon}Id&0&0\\ \delta A& d A&0\\ \delta A&0& d A
    \end{pmatrix}.
\end{align}
This linear operator generates the semigroup 
\begin{align}
    \textnormal{e}^{\mathcal{A}^\varepsilon t}= \begin{pmatrix}
          \textnormal{e}^{[(d+\delta)A -\frac{1}{\varepsilon}Id]t}&0&0\\ \delta A \big(\delta A-\frac{1}{\varepsilon}\big)^{-1}\bigg( \textnormal{e}^{d A t}- \textnormal{e}^{[(d+\delta)A -\frac{1}{\varepsilon}Id]t} \bigg)& \textnormal{e}^{d A t}&0\\ \delta A \big(\delta A-\frac{1}{\varepsilon}\big)^{-1}\bigg( \textnormal{e}^{d A t}- \textnormal{e}^{[(d+\delta) A -\frac{1}{\varepsilon}Id]t}\bigg)& &0&\textnormal{e}^{d A t}
    \end{pmatrix}.
\end{align}
Using a perturbation result for semigroups \cite[Cor. III.1.7]{engel2000one} we can rewrite the above expression as follows
\begin{align}
     \textnormal{e}^{\mathcal{A}^\varepsilon t}=\begin{pmatrix}
        \textnormal{e}^{[(d+\delta)A -\frac{1}{\varepsilon}Id]t}&0&0\\ \delta A \textnormal{e}^{d A t} \int_0^t \textnormal{e}^{(\delta A-\frac{1}{\varepsilon}Id) s} \, \textnormal{d}s & \textnormal{e}^{d A t}&0\\ \delta A \big(\delta A-\frac{1}{\varepsilon}\big)^{-1}\bigg( \textnormal{e}^{d A t}- \textnormal{e}^{[(d+\delta)A -\frac{1}{\varepsilon}Id]t}\bigg)& &0&\textnormal{e}^{d A t}
     \end{pmatrix}.
\end{align}
With this we can define the Lyapunov-Perron operator as $\mathcal{L}_{\varepsilon,\zeta}:C_\eta \to C_\eta$ and for each component we have
\begin{align*}
    u &\mapsto \bigg[ t\mapsto \bigg(\int_{-\infty}^t \textnormal{e}^{[(d+\delta)A-\frac{1}{\varepsilon}\textnormal{Id}](t-s)}\big( \varepsilon^{-1}f(u(s),v_F(s),v_S(s))+ \phi(u(s),v_F(s),v_S(s))\big)\, \textnormal{d}s \bigg) \bigg],\\
    v_F&\mapsto \bigg [ t\mapsto  \bigg(\int_{-\infty}^t  \delta A \textnormal{e}^{d A (t-s)}\bigg( \int_0^{(t-s)} \textnormal{e}^{(\delta A-\frac{1}{\varepsilon}Id) r} \, \textnormal{d}r  \bigg)  \pr_{X^\zeta_F}\varepsilon^{-1}f(u(s),v_F(s),v_S(s))\, \textnormal{d}s\\
    &\qquad+ \int_{-\infty}^t  \delta A \textnormal{e}^{d A (t-s)}\bigg( \int_0^{(t-s)} \textnormal{e}^{(\delta A-\frac{1}{\varepsilon}Id) r} \, \textnormal{d}r  \bigg)  \pr_{X^\zeta_F} \phi(u(s),v_F(s),v_S(s))\, \textnormal{d}s\\ 
      &\qquad + \int_{-\infty}^t  \textnormal{e}^{d A (t-s)} \psi(u(s),v_F(s),v_S(s))\, \textnormal{d}s \bigg)\bigg],\\
     v_S &\mapsto \bigg[t\mapsto \bigg(  
     \delta A \big(\delta A-\frac{1}{\varepsilon}\big)^{-1}\bigg( \textnormal{e}^{d A t}- \textnormal{e}^{[(d+\delta)A -\frac{1}{\varepsilon}Id]t}\bigg) \pr_{X_S^\zeta} u_0+ \textnormal{e}^{d A t}v_0\\
     &\qquad + \int_0^t \delta A \big(\delta A-\frac{1}{\varepsilon}\big)^{-1}\bigg( \textnormal{e}^{d A (t-s)}- \textnormal{e}^{[(d+\delta)A -\frac{1}{\varepsilon}Id](t-s)}\bigg) \pr_{X_S^\zeta}  \varepsilon^{-1}f(u(s),v_F(s),v_S(s)) \, \textnormal{d}s\\
     &\qquad+\int_0^t \delta A \big(\delta A-\frac{1}{\varepsilon}\big)^{-1}\bigg( \textnormal{e}^{d A (t-s)}- \textnormal{e}^{[(d+\delta)A -\frac{1}{\varepsilon}Id](t-s)}\bigg) \pr_{X_S^\zeta}\phi(u(s),v_F(s),v_S(s))\, \textnormal{d}s\\
     &\qquad+\int_{0}^t  \textnormal{e}^{d A (t-s)} \pr_{X_S^\zeta} \psi(u(s),v_F(s),v_S(s))\, \textnormal{d}s\bigg)\bigg],
\end{align*}
 where $v_0\in X_1\cap X^\zeta_S$ and $u_0\in X_1$.
 Moreover, $C_\eta$ is defined as
 \[C_\eta:= C((-\infty,0],\textnormal{e}^{\eta t}; X_1
    \times(X_F^\zeta\cap X_1)\times (X_1\cap X^\zeta_S)), \]
   which is the space of $(u,v_F,v_S)\in C((-\infty,0]; X_1
    \times(X_F^\zeta\cap X_1)\times ( X^\zeta_S \cap X_1) $ such that
    \begin{align*}
        \|(u,v_F,v_S)\|_{C_\eta}:=\sup_{t\leq 0} \textnormal{e}^{-\eta t}(\|u\|_{X_1}+\|v_F\|_{X_1}+\|v_S\|_{X_1})<\infty.
    \end{align*}
The parameter $\eta<0$ is chosen as $\eta= \zeta^{-1}(\varepsilon (d+\delta)\omega_A-1) +\frac{1}{2}(N_F^\zeta+N_S^\zeta)$.

\begin{remark}
   The idea behind the Lyapunov-Perron method is to focus on the behavior of solution curves. 
   The solutions on the invariant slow manifold are precisely characterized by the fact that they stay bounded under backward evolution in the fast variables.
   For more details we refer to \cite[Ch. 6]{duan2014effective} and \cite[Ch. 1]{eldering2013normally}.
\end{remark}

\begin{remark}
    A key difficulty in the subsequent proofs is that the linear differential operator is not self-adjoint. 
    This fact leads to the smallness condition on $\delta$, where in the limit as $\varepsilon\to 0$ the operator becomes self-adjoint again.
    Moreover, we note that the parameter $\eta$ might not be chosen optimal, when comparing this result with other results on inertial manifolds in non-self-adjoint systems in \cite{kostianko2021kwak}.
    Yet, with our choice of the parameter we can build on previous work and obtain explicit convergence rates in the later steps.
\end{remark}

\begin{prop} \label{existence slow manifold}
    Let $v_0\in X_1\cap X^\zeta_S$ and $u_0 \in X_1$. Then the operator $\mathcal{L}_{\varepsilon,\zeta}$ has a unique fixed point in $C_\eta$.
\end{prop}

\begin{proof}
To show that the operator has a fixed point we have to show that it maps $C_\eta$ into $C_\eta$ and moreover that the operator is a contraction. 
As both computations have the same structure, we focus on showing the contraction property.
To this end let $(u,v_F,v_S),(\tilde u, \tilde v_F, \tilde v_S)\in C_\eta$.
Then, we estimate the difference for each component of the operator. 
For the first one we obtain
\begin{align*}
    &\|\pr_u ( \mathcal{L}_{\varepsilon,\zeta}(u,v_F,v_S)- \mathcal{L}_{\varepsilon,\zeta}(\tilde u, \tilde v_F, \tilde v_S)\|_{C_\eta}=\\
    &=\sup_{t\leq 0} \textnormal{e}^{-\eta t} \bigg\|\int_{-\infty}^t \textnormal{e}^{[(d+\delta) A -\frac{1}{\varepsilon}\textnormal{Id}](t-s)}\varepsilon^{-1}(f(u(s),v_F(s),v_S(s))-f(\tilde u(s),\tilde v_F(s), \tilde v_S(s))) \, \textnormal{d}s\\
    &\qquad\qquad+\int_{-\infty}^t \textnormal{e}^{[(d+\delta) A -\frac{1}{\varepsilon}\textnormal{Id}](t-s)}\varepsilon^{-1}(\phi(u(s),v_F(s),v_S(s))-\phi(\tilde u(s),\tilde v_F(s), \tilde v_S(s)))\, \textnormal{d}s\bigg\|_{X_1}\\
    &\leq \sup_{t\leq 0}\int_{-\infty}^t \textnormal{e}^{-\eta t} \bigg\|\textnormal{e}^{[(d+\delta) A -\frac{1}{\varepsilon}\textnormal{Id}](t-s)}\varepsilon^{-1}(f(u(s),v_F(s),v_S(s))-f(\tilde u(s),\tilde v_F(s), \tilde v_S(s)))\bigg\|_{X_1}\,\textnormal{d}s\\
    &\quad + \sup_{t\leq 0}\int_{-\infty}^t \textnormal{e}^{-\eta t} \bigg\|\textnormal{e}^{[(d+\delta) A -\frac{1}{\varepsilon}\textnormal{Id}](t-s)}(\phi(u(s),v_F(s),v_S(s))-\phi(\tilde u(s),\tilde v_F(s), \tilde v_S(s)))\bigg\|_{X_1}\,\textnormal{d}s\\
&\leq C_ A  \int_{-\infty}^t \textnormal{e}^{\varepsilon^{-1}(\varepsilon(d+\delta)\omega_ A -1-\varepsilon\eta)(t-s)}(\varepsilon^{-1}L_f +L\phi) \,\textnormal{d}s \|(u,v_F,v_S)-(\tilde u, \tilde v_F, \tilde v_S)\|_{C_\eta}\\
&\leq \frac{C_ A  (L_f+\varepsilon L_\phi) }{|\varepsilon(d+\delta)\omega_ A -1-\varepsilon\eta|}\|(u,v_F,v_S)-(\tilde u, \tilde v_F, \tilde v_S)\|_{C_\eta}\\
&\leq \frac{C_ A (L_f+\varepsilon L_\phi)}{|(1-\varepsilon \zeta^{-1})(\varepsilon (d+\delta)\omega_A-1) - \varepsilon\frac{1}{2}(N_S^\zeta+N_F^\zeta)|}\|(u,v_F,v_S)-(\tilde u, \tilde v_F, \tilde v_S)\|_{C_\eta},
\end{align*}
where we have used the assumption that $ (1-\varepsilon \zeta^{-1})(\varepsilon (d+\delta)\omega_A-1) -\varepsilon \frac{N_S^\zeta+N_F^\zeta}{2}<0$ to ensure that the exponent has the right sign.

For the second component we estimate
\begin{align*}
     &\|\pr_{X_F^\zeta} ( \mathcal{L}_{\varepsilon,\zeta}(u,v_F,v_S)- \mathcal{L}_{\varepsilon,\zeta}(\tilde u, \tilde v_F, \tilde v_S)\|_{C_\eta}=\\
     &=\sup_{t\leq 0} \textnormal{e}^{-\eta t} \bigg\| \int_{-\infty}^t   \delta A \textnormal{e}^{d A (t-s)} \int_0^{(t-s)} \textnormal{e}^{(\delta A-\frac{1}{\varepsilon}Id) r} \, \textnormal{d}r \pr_{X^\zeta_F}\varepsilon^{-1}(f(u(s),v_F(s),v_S(s))-f(\tilde u(s),\tilde v_F(s), \tilde v_S(s)))\, \textnormal{d}s \\
     &\qquad \qquad+\int_{-\infty}^t   \delta A \textnormal{e}^{d A (t-s)} \int_0^{(t-s)} \textnormal{e}^{(\delta A-\frac{1}{\varepsilon}Id) r} \, \textnormal{d}r \pr_{X^\zeta_F}(\phi(u(s),v_F(s),v_S(s))-\phi(\tilde u(s),\tilde v_F(s), \tilde v_S(s)))\, \textnormal{d}s \\
     &\qquad \qquad +\int_{-\infty}^t \textnormal{e}^{d A (t-s)}\pr_{X_F^\zeta}( \psi(u(s),v_F(s),v_S(s)) -\psi(\tilde u(s),\tilde v_F(s), \tilde v_S(s)))\, \textnormal{d}s\bigg\|_{X_1}.
\end{align*}
As $r\geq 0$ we can estimate the $X_1$-norm of the inner integral by
\begin{align*}
    \bigg \| \int_0^{(t-s)} \textnormal{e}^{(\delta A-\frac{1}{\varepsilon}Id) r} \, \textnormal{d}r\bigg\|_{\mathcal{B}(X_1)}\leq     (t-s)C_A\sup_{r\in [0,(t-s)]} \textnormal{e}^{(\delta\omega_A\frac{1}{\varepsilon})r}\leq (t-s)C_A.
\end{align*}
Thus, we obtain
\begin{align*}
 &\|\pr_{X_F^\zeta} ( \mathcal{L}_{\varepsilon,\zeta}(u,v_F,v_S)- \mathcal{L}_{\varepsilon,\zeta}(\tilde u, \tilde v_F, \tilde v_S)\|_{C_\eta}\leq \\
    &\leq C_A \sup_{t\leq 0} \textnormal{e}^{-\eta t} \int_{-\infty}^t \bigg \|   \delta (t-s) A \textnormal{e}^{d A (t-s)}\pr_{X^\zeta_F} \varepsilon^{-1}(f(u(s),v_F(s),v_S(s))-f(\tilde u(s),\tilde v_F(s), \tilde v_S(s)))\bigg\|_{X_1} \, \textnormal{d}s \\
    &\quad + C_A \sup_{t\leq 0} \textnormal{e}^{-\eta t} \int_{-\infty}^t \bigg \|   \delta (t-s) A \textnormal{e}^{d A (t-s)}\pr_{X^\zeta_F} (\phi(u(s),v_F(s),v_S(s))-\phi(\tilde u(s),\tilde v_F(s), \tilde v_S(s)))\bigg\|_{X_1} \, \textnormal{d}s \\
     &\quad + \sup_{t\leq 0} \int_{-\infty}^t \textnormal{e}^{-\eta t} \bigg\| \textnormal{e}^{d A (t-s)}\pr_{X_F^\zeta}( \psi(u(s),v_F(s),v_S(s)) -\psi(\tilde u(s),\tilde v_F(s), \tilde v_S(s)))\bigg\|_{X_1} \, \textnormal{d}s\\
    &\leq C_A^2 (\varepsilon^{-1}L_f +L\phi) \int_{-\infty}^t \delta d^{-1} \textnormal{e}^{(N_F^\zeta +\zeta^{-1}(\varepsilon (d+\delta) \omega_A-1)-\eta)(t-s)} \, \textnormal{d}s \|(u,v_F,v_S)-(\tilde u, \tilde v_F, \tilde v_S)\|_{C_\eta}\\
    &\quad + C_A L_\psi \int_{-\infty}^t \textnormal{e}^{(N_F^\zeta +\zeta^{-1}(\varepsilon (d+\delta) \omega_A-1)-\eta)(t-s)} \, \textnormal{d}s     \|(u,v_F,v_S)-(\tilde u, \tilde v_F, \tilde v_S)\|_{C_\eta}\\
    &\leq \frac{C_A^2 \delta d^{-1}(L_f+\varepsilon L_\phi) +C_A \varepsilon L_\psi}{\varepsilon|N_F^\zeta +\zeta^{-1}(\varepsilon d \omega_\Delta-1)-\eta|}\|(u,v_F,v_S)-(\tilde u, \tilde v_F, \tilde v_S)\|_{C_\eta}\\
    &\leq 2\frac{C_A^2 \delta d^{-1}(L_f+\varepsilon L_\phi) +C_A \varepsilon L_\psi}{\varepsilon|N_F^\zeta- N_S^\zeta |}\|(u,v_F,v_S)-(\tilde u, \tilde v_F, \tilde v_S)\|_{C_\eta}.
\end{align*}
For the third component we note that 
\[ \textnormal{e}^{[(d+\delta)A -\frac{1}{\varepsilon}Id]t} \pr_{X^\zeta_S} x=0 ~~\textnormal{for } x\in X_1\]
as the semigroup only acts on components in the fast variable space.
Hence,
\begin{align*}
    &\|\pr_{X_S^\zeta} ( \mathcal{L}_{\varepsilon,\zeta}(u,v_F,v_S)- \mathcal{L}_{\varepsilon,\zeta}(\tilde u, \tilde v_F, \tilde v_S)\|_{C_\eta}=\\
    &=\sup_{t\leq 0} \textnormal{e}^{-\eta t} \bigg\| \int_{0}^t \delta A \big(\delta A-\frac{1}{\varepsilon}\big)^{-1} \textnormal{e}^{d A (t-s)} \pr_{X^\zeta_S}\varepsilon^{-1}(f(u(s),v_F(s),v_S(s))-f(\tilde u(s),\tilde v_F(s), \tilde v_S(s)))\, \textnormal{d}s \\
    &\qquad \qquad +  \int_{0}^t \delta A \big(\delta A-\frac{1}{\varepsilon}\big)^{-1} \textnormal{e}^{d A (t-s)} \pr_{X^\zeta_S}(\phi(u(s),v_F(s),v_S(s))-\phi(\tilde u(s),\tilde v_F(s), \tilde v_S(s)))\, \textnormal{d}s \\
     &\qquad \qquad +\int_{-\infty}^t \textnormal{e}^{d A (t-s)}\pr_{X_S^\zeta}( \psi(u(s),v_F(s),v_S(s)) -\psi(\tilde u(s),\tilde v_F(s), \tilde v_S(s)))\, \textnormal{d}s\bigg\|_{X_1}.
\end{align*}
In addition, the operator $ A \big(\delta A-\frac{1}{\varepsilon}\big)^{-1}$ has the following bound in the slow variable space
$$\| A \big(\delta A-\frac{1}{\varepsilon}\big)^{-1} \pr_{X_S^\zeta} x\|_{X_1}\leq \|\pr_{X_S^\zeta} x\|_{X_1}.$$
Thus,
\begin{align*}
     &\|\pr_{X_S^\zeta} ( \mathcal{L}_{\varepsilon,\zeta}(u,v_F,v_S)- \mathcal{L}_{\varepsilon,\zeta}(\tilde u, \tilde v_F, \tilde v_S)\|_{C_\eta}\leq \\
     &\leq \sup_{t\leq 0} \textnormal{e}^{-\eta t} \int_0^t \bigg\| \delta  \textnormal{e}^{d A (t-s)} \pr_{X^\zeta_S}\varepsilon^{-1}(f(u(s),v_F(s),v_S(s))-f(\tilde u(s),\tilde v_F(s), \tilde v_S(s)))\bigg\|_{X_1} \, \textnormal{d}s\\
     &\quad +\sup_{t\leq 0} \textnormal{e}^{-\eta t} \int_0^t \bigg\| \delta  \textnormal{e}^{d A (t-s)} \pr_{X^\zeta_S}(\phi(u(s),v_F(s),v_S(s))-\phi(\tilde u(s),\tilde v_F(s), \tilde v_S(s)))\bigg\|_{X_1} \, \textnormal{d}s\\
     &\quad + \sup_{t\leq 0} \int_{-\infty}^t \textnormal{e}^{-\eta t} \bigg\| \textnormal{e}^{d A (t-s)}\pr_{X_S^\zeta}( \psi(u(s),v_F(s),v_S(s)) -\psi(\tilde u(s),\tilde v_F(s), \tilde v_S(s)))\bigg\|_{X_1} \, \textnormal{d}s\\
     &\leq C_A (\delta \varepsilon^{-1} L_f +\delta L_\phi +L_\psi) \int_0^t \varepsilon^{-1}  \textnormal{e}^{(N_S^\zeta +\zeta^{-1}(\varepsilon d\omega_A-1)-\eta) (t-s)}\, \textnormal{d}s \|(u,v_F,v_S)-(\tilde u, \tilde v_F, \tilde v_S)\|_{C_\eta}\\
     &\leq \frac{C_A  (\delta  L_f +\delta \varepsilon L_\phi +\varepsilon L_\psi)  }{\varepsilon|N_S^\zeta +\zeta^{-1}(\varepsilon d\omega_\Delta-1)-\eta|}\|(u,v_F,v_S)-(\tilde u, \tilde v_F, \tilde v_S)\|_{C_\eta}\\
     &\leq 2 \frac{C_A  (\delta  L_f +\delta \varepsilon L_\phi +\varepsilon L_\psi)  }{\varepsilon|N_S^\zeta-N_F^\zeta|}\|(u,v_F,v_S)-(\tilde u, \tilde v_F, \tilde v_S)\|_{C_\eta}.
\end{align*}
Then, using the spectral gap assumption that
\begin{align*}
  &\frac{C_ A (L_f+\varepsilon L_\phi)}{| (1-\varepsilon \zeta^{-1})(\varepsilon (d+\delta)\omega_A-1) - \varepsilon\frac{1}{2}(N_S^\zeta+N_F^\zeta)|} +  2 \frac{ \delta (C_A+d^{-1} C_A^2)(  L_f + \varepsilon L_\phi) +2C_A\varepsilon L_\psi  }{\varepsilon|N_S^\zeta-N_F^\zeta|} =:\tilde L <1
\end{align*}
we indeed obtain that the operator is a self-mapping and a contraction.
Therefore, by the Banach fixed-point theorem it has a unique fixed point.  
\end{proof}
This proves the first statement in the definition of a slow manifold.

\begin{remark}
    Observe that the Lyapunov-Perron operator depends on both $u_0$ and $v_0$ for $t<0$. Only for $t=0$ the operator solely depends on the initial data in the slow component $v_0$.
\end{remark}


\begin{prop}\label{Lipschitz cont}
Let $(u^{v_0},v_F^{v_0},v_S^{v_0})$ be the unique fixed point of $\mathcal{L}_{v_0,\varepsilon,\zeta}$ evaluated at $t=0$.
Then, the mapping
\begin{align*}
    h^{\varepsilon,\zeta}: (X^\zeta_S\cap X_1)\to X_1\times (X^\zeta_F\cap X_1),\, v_0\mapsto (u^{v_0}(0), v_F^{v_0}(0))
\end{align*}
is Lipschitz continuous with Lipschitz constant $L$.
\end{prop}
\begin{proof}
    Let $v_0, \tilde v_0 \in X_1\cap X_S^\zeta$ be the initial data and let $(u,v_F,v_S), (\tilde u,\tilde v_F,\tilde v_S)\in C_\eta$ be the corresponding fixed point of the Lyapunov-Perron operator evaluated at $t=0$.
    Then, by the previous result it follows that
    \begin{align*}
        &\sup_{t\leq 0} \textnormal{e}^{-\eta t} \|u(t)-\tilde u(t)\|_{X_1} \leq \frac{C_ A (L_f+\varepsilon L_\phi)}{|\varepsilon (1-\varepsilon \zeta^{-1})(\varepsilon (d+\delta)\omega_A-1) + \varepsilon\frac{1}{2}(N_S^\zeta+N_F^\zeta)|}       \|(u,v_F,v_S)-(\tilde u, \tilde v_F, \tilde v_S)\|_{C_\eta}\\
       & \sup_{t\leq 0} \textnormal{e}^{-\eta t} \|v_F(t)-\tilde v_F(t)\|_{X_1}\leq 2\frac{C_A^2 \delta d^{-1}(L_f+\varepsilon L_\phi) +C_A \varepsilon L_\psi}{\varepsilon|N_F^\zeta- N_S^\zeta |}\|(u,v_F,v_S)-(\tilde u, \tilde v_F, \tilde v_S)\|_{C_\eta}\\
       &\sup_{t\leq 0} \textnormal{e}^{-\eta t} \|v_S(t)-\tilde v_S(t)\|_{X_1}\leq M_A\|v_0-\tilde v_0\|_{X_1} +2 \frac{C_A  (\delta  L_f +\delta \varepsilon L_\phi +\varepsilon L_\psi)  }{\varepsilon|N_S^\zeta-N_F^\zeta|}\|(u,v_F,v_S)-(\tilde u, \tilde v_F, \tilde v_S)\|_{C_\eta}.
    \end{align*}
    By the bound of $\tilde L<1$ we can rewrite the inequalities as
    \begin{align*}
        \|(u,v_F,v_S)-(\tilde u, \tilde v_F, \tilde v_S)\|_{C_\eta} \leq\underbrace{M_A (1-\tilde L)^{-1}}_{=:L} \|v_0-\tilde v_0\|_{X_1}
    \end{align*}
    which proves the Lipschitz continuity of the Lyapunov-Perron operator.
\end{proof}
Therefore, the slow manifold $S_{\varepsilon,\zeta}$ given by
\begin{align*}
  S_{\varepsilon,\zeta}=\{(h^{\varepsilon,\zeta}(v_0),v_0):\, v_0\in X_1\cap X_S^\zeta\},
\end{align*}
is a Lipschitz continuous manifold.

\begin{remark}
    We want to point out that the above method also works for the case when the coefficients in front of the operator $A$ on the diagonal are the same, i.e.
    the system has the form
    \begin{align*}
        \partial_t u &= d A u^\varepsilon -\frac{1}{\varepsilon}u +\frac{1}{\varepsilon}f(u,v_F,v_S)+ \phi(u,v_F,v_S),\\
        \partial_t v_F &= d A v_F+ \delta A \pr_{X_F^\zeta} u +\pr_{X_F^\zeta}\psi(u,v_F,v_S),\\
         \partial_t v_S &=d A v_S + \delta A \pr_{X_S^\zeta} u +\pr_{X_S^\zeta}\psi(u,v_F,v_S),\\
          u(0)&=u_0,\quad     v_F(0)=\pr_{X_F^\zeta} v_0,\quad v_S(0)=\pr_{X_S^\zeta} v_0.
    \end{align*}
    However, as the analysis becomes more intricate in this setting and we will not present any details of the computations.
\end{remark}

\subsection{Distance to the Critical Manifold}

In this section, we give an upper bound on the distance between the critical manifold $S_0$ and the slow manifold $S_{\varepsilon,\zeta}$ by comparing the distance of the two generating graphs.

\begin{prop}\label{prop:distance to crit manifold}
   There exists a constant $C>0$ such that for all $\varepsilon,\, \zeta$ small enough satisfying $\varepsilon\zeta^{-1}\leq c<1$ and for all $v_0\in X_1\cap X_S^\zeta$ it holds that
   \begin{align*}
     \left\| \begin{pmatrix}
         h^{\varepsilon,\zeta}_{u}(v_0)-h^0(v_0)\\h^{\varepsilon,\zeta}_{X_F^\zeta}(v_0)
     \end{pmatrix}\right\|_{X_{\alpha}\times X_1}\leq C\bigg(\frac{\delta+\varepsilon}{\epsilon(N_S^\zeta-N_F^\zeta)}+ \varepsilon^{1-\alpha}\bigg) \|v_0\|_{X_1},
   \end{align*}
   where $0\leq \alpha<1$.
   This shows the second part in the definition of the slow manifold.
\end{prop}
Note, that as the proof is rather technical we place it at the end of this paper in Section \ref{sec.proof of prop distance}.

\begin{remark}
    Let us put this result into better context. The distance between the critical manifold $S_0$ and the slow manifold $S_{\varepsilon,\zeta}$ depends on three parameters; the diffusion coefficient of the cross diffusion term $\delta$, the time-scale separation $\varepsilon$ and the spectral gap of the differential operator $A$ given by $N_S^\zeta-N_F^\zeta$.
    Now, coming back to the fast-reaction system from the introduction, the Laplace operator $\Delta$ on $L^2([0,2\pi])$ with Dirichlet boundary conditions has a spectral gap $N_S^\zeta-N_F^\zeta$ of order $\mathcal{O}(\zeta^{-1/2})$. 
    This means in order to obtain convergence the diffusion coefficient $\delta=\delta(\varepsilon)$ has to satisfy $\lim_{\varepsilon\to 0} \frac{\delta(\varepsilon)}{\varepsilon^{1/2}}=0$.
\end{remark}
\begin{remark}
    The reason we cannot compute the distance between the critical manifold and the slow manifold in $X_1$ is due to the fact that the perturbed operator $\textnormal{Id}-\varepsilon A$ is not bounded as a mapping between $X_1\to X_1$.
    Hence, we have to work in the less regular space $X_\alpha$, for $\alpha\in [0,1)$.
    Moreover, we observe that the higher the regularity is for the estimate of the first component given in the space $X_\alpha$, where $\alpha\in [0,1)$, the slower the convergence rate is, i.e. the rate of convergence is $\varepsilon^{1-\alpha}$.
\end{remark}

\subsection{Differentiability of the Slow Manifold}

We will show the differentiability of the slow manifold by showing the differentiability of its graph, i.e. the mapping 
\begin{align*}
    \big(X_S^\zeta\cap X_1,\|\cdot\|_{X_1}\big)&\to   \big( X_1,\|\cdot\|_{X_1}\big)\times   \big(X_S^\zeta\cap X_1,\|\cdot\|_{X_1}\big),\\
    v_0&\mapsto \big( h_u^{\varepsilon,\zeta}(v_0),h_{X_F^\zeta}^{\varepsilon,\zeta}(v_0)\big)
\end{align*}
is differentiable.
\begin{lemma}\label{lemma.differentiability of slow manifold}
    Under the assumptions of Section \ref{Sec:Assumptions} the slow manifold $S_{\varepsilon,\zeta}$ is differentiable.
\end{lemma}
Again, as the proof of the statement is rather technical we moved it to the end of this paper in Section \ref{sec.proof of lemma}.\\
Thus, we have shown that the slow manifold is indeed a differentiable and Lipschitz continuous manifold and therefore has the same regularity as the critical manifold.
This completes the third step in the definition of the slow manifold.

\subsection{Attraction of Trajectories}
The goal of this section is to prove that the slow manifold is exponentially attractive, i.e. all solutions are attracted towards the slow manifold in exponential time.
To this end, let the initial data $v_0$ satisfy $v_0\in X_S^\zeta\cap X_1$, then $(h_u^{\varepsilon\zeta}(v_0),h_{X_F^\zeta}^{\varepsilon\zeta}(v_0),v_0)\in S_{\varepsilon,\zeta}$.\\
Now, let $z:=(u,v_F,v_S)$ be a solution to system \eqref{general system split} with initial values $(h_u^{\varepsilon\zeta}(v_0),h_{X_F^\zeta}^{\varepsilon\zeta}(v_0),v_0)$.
Moreover, let $z^\varepsilon:= (u^\varepsilon,v_F^\varepsilon,v_S^\varepsilon)$ be the solution to \eqref{general system split} with initial values $(u_0,v_{0,F},v_{0,S})$.
With this we make the following observations.
It holds that
\begin{align*}
  \partial_t z(t) = \mathcal{A}^\varepsilon z(t)+ F^\varepsilon(z(t)),
\end{align*}
where we recall the notation for $\mathcal{A}^\varepsilon $ from Proposition \ref{existence slow manifold} and denote the nonlinear terms $$F^\varepsilon(z(t))=\begin{pmatrix}
    \varepsilon^{-1} f(u(t),v_F(t),v_S(t))+ \phi(u(t),v_F(t),v_S(t))\\ \pr_{X_F^\zeta}\psi(u(t),v_F(t),v_S(t)) \\ \pr_{X_S^\zeta}\psi(u(t),v_F(t),v_S(t))
\end{pmatrix} .$$
As $S_{\varepsilon,\zeta}$ is an invariant and differentiable manifold it holds that 
$$z(t)=h^{\varepsilon,\zeta}(v_S(t))= \big(h^{\varepsilon,\zeta}_u(v_S(t)),h_{X_F}^{\varepsilon,\zeta}(v_S(t)),v_S(t)\big)^T.$$ 
Therefore, 
\begin{align*}
    \partial_t z(t)= \partial_t h^{\varepsilon,\zeta} (v_S(t))&= \textnormal D h^{\varepsilon,\zeta}(v_S(t)) [\partial_t v_S(t)]\\
    &= \textnormal D h^{\varepsilon,\zeta}(v_S(t)) [d \Delta v_S(t) +\delta\Delta \pr_{X_S^\zeta} u(t)+\pr_{X_S^\zeta}\psi(u(t),v_F(t),v_S(t))].
\end{align*}
Then, at time $t=0$ and using that $(u(t),v_F(t))= (h^{\varepsilon,\zeta}_u(v_S(t)),h^{\varepsilon,\zeta}_{X_F^\zeta}(v_S(t)))$ we have
\begin{align}\label{trajectory comp 1}
\begin{split}
     h^{\varepsilon,\zeta}(v_0) =&\big(\mathcal{A}\varepsilon\big)^{-1} \textnormal D h^{\varepsilon,\zeta}_u(v_0) \big[d \Delta v_0 +\delta\Delta \pr_{X_S^\zeta} h^{\varepsilon,\zeta}_u(v_0)+ \pr_{X_S^\zeta}\psi( h^{\varepsilon,\zeta}_u(v_0), h^{\varepsilon,\zeta}_{X_F^\zeta}(v_0),v_0)\big]\\
  &-\big(\mathcal{A}\varepsilon\big)^{-1} F(h^{\varepsilon,\zeta}(v_0)).
\end{split}
 \end{align}
As these equations hold for arbitrary $v_0\in X_S^\zeta\cap X_1$ they also hold for $v_0=v_S^\varepsilon(t)$.
Moreover, by the differentiability of $h^{\varepsilon,\zeta}$ it holds that
\begin{align}\label{trajectory comp 2}
    \partial_t h^{\varepsilon,\zeta}(v_S^\varepsilon(t))&= \textnormal D h^{\varepsilon,\zeta}(v_S^\varepsilon(t)) \big[d\Delta v_S^\varepsilon(t) +\delta\Delta \pr_{X_S^\zeta} u^\varepsilon(t)+\pr_{X_S^\zeta}\psi(u^\varepsilon(t),v_F^\varepsilon(t),v_S^\varepsilon(t))\big].
\end{align}

\begin{prop}
    Let the above assumptions hold. Then, there  are constants $C,c>0$ such that
    \begin{align*}
        \left\| z^\varepsilon(t)-h^{\varepsilon,\zeta}(v_S^\varepsilon(t)) \right\|_{X_1}\leq C\textnormal{e}^{-ct}\big(\|v_0-h_{X_F^\zeta}^{\varepsilon,\zeta}(v_{S,0})\|_{X_1} +\|u_0-h_u^{\varepsilon,\zeta}(v_{S,0})\|_{X_1} \big)
    \end{align*}
    for all $t>0$.
\end{prop}

\begin{proof}
We compute
\begin{align*}
    z^\varepsilon(t)-h^{\varepsilon,\zeta}(v_S^\varepsilon(t))&= \textnormal{e}^{\mathcal{A}^\varepsilon t}z_0+\int_0^t \textnormal{e}^{\mathcal{A}^\varepsilon (t-s)} F^\varepsilon(z(s))\,\textnormal{d}s -h^{\varepsilon,\zeta}(v_S^\varepsilon(t))\\
    &-\textnormal{e}^{\mathcal{A}^\varepsilon t} h^{\varepsilon,\zeta}(v_{S,0}) +-\textnormal{e}^{\mathcal{A}^\varepsilon t} h^{\varepsilon,\zeta}(v_{S,0})\\
    &= \textnormal{e}^{\mathcal{A}^\varepsilon t} \big(z_0-h^{\varepsilon,\zeta}(v_{S,0})\big) +\int_0^t \textnormal{e}^{\mathcal{A}^\varepsilon (t-s)} F^\varepsilon(z(s))\,\textnormal{d}s\\
    &- \int_0^t \partial_s \textnormal{e}^{\mathcal{A}^\varepsilon (t-s)}h^{\varepsilon,\zeta}(v_S^\varepsilon(s))\,\textnormal{d}s\\
    &=  \textnormal{e}^{\mathcal{A}^\varepsilon t} \big(z_0-h^{\varepsilon,\zeta}(v_{S,0})\big) +\int_0^t \textnormal{e}^{\mathcal{A}^\varepsilon (t-s)} F^\varepsilon(z(s))\,\textnormal{d}s\\
    &+\int_0^t  \textnormal{e}^{\mathcal{A}^\varepsilon (t-s)} \big[ \mathcal{A}^\varepsilon h^{\varepsilon,\zeta}(v_S^\varepsilon(s))-\partial_s h^{\varepsilon,\zeta}(v_S^\varepsilon(s))\big]\,\textnormal{d}s.
\end{align*}
Using equations \eqref{trajectory comp 1} and \eqref{trajectory comp 2} yields
\begin{align*}
    z^\varepsilon(t)-h^{\varepsilon,\zeta}(v_S^\varepsilon(t)) &=  \textnormal{e}^{\mathcal{A}^\varepsilon t} \big(z_0-h^{\varepsilon,\zeta}(v_{S,0})\big) +\int_0^t \textnormal{e}^{\mathcal{A}^\varepsilon (t-s)} \big[ F^\varepsilon(z(s))- F^\varepsilon(h^{\varepsilon,\zeta}(v_S^\varepsilon(s)))\big] \,\textnormal{d}s\\
    &+\int_0^t\textnormal{e}^{\mathcal{A}^\varepsilon (t-s)} \textnormal{D}h^{\varepsilon,\zeta}(v_S^\varepsilon(s)) \delta A \pr_{X_S^\zeta} \big[u^\varepsilon(s)-h^{\varepsilon,\zeta}_u(v_S^\varepsilon(s)) \big]\,\textnormal{d}s\\
      &+\int_0^t\textnormal{e}^{\mathcal{A}^\varepsilon (t-s)} \textnormal{D}h^{\varepsilon,\zeta}(v_S^\varepsilon(s))\\
      &\qquad \qquad \times \pr_{X_S^\zeta} \big[ \psi(u^\varepsilon(s),v_F^\varepsilon(s),v_S^\varepsilon(s))-\psi( h^{\varepsilon,\zeta}_u(v_S^\varepsilon(s)), h^{\varepsilon,\zeta}_{X_F^\zeta}(v_S^\varepsilon(s)),v_S^\varepsilon(s)) \big]\,\textnormal{d}s.
\end{align*}
Then we estimate
\begin{align*}
   &\|  z^\varepsilon(t)-h^{\varepsilon,\zeta}(v_S^\varepsilon(t))\|_{X_1} \leq \\
   &\leq C_A \textnormal{e}^{((d+\delta)\omega_A -\varepsilon^{-1})t}\|u_0-h_u^{\varepsilon,\zeta}(v_{S,0})\|_{X_1}+ C_A \textnormal{e}^{(\zeta^{-1}(\varepsilon (d+\delta)\omega_A -1)+N_F^\zeta)t} \|v_0-h_{X_F^\zeta}^{\varepsilon,\zeta}(v_{S,0})\|_{X_1}\\
   &\quad +\bigg\| \delta A\textnormal{e}^{d A t}\int_0^t \textnormal{e}^{(\delta A -\varepsilon^{-1}Id)s}\,\textnormal{d} s \,\pr_{X_F^\zeta}\big[u_0-h_u^{\varepsilon,\zeta}(v_{S,0})\big] \bigg\|_{X_1}\\
   &\quad  + C_A (\varepsilon^{-1} L_f+L_\phi) \int_0^t  \textnormal{e}^{((d+\delta)\omega_A -\varepsilon^{-1})(t-s)} \|z^\varepsilon(s)-h^{\varepsilon,\zeta}(v_S^\varepsilon(s))\|_{X_1} \,\textnormal{d}s\\
   &\quad +\bigg \| \int_0^t \varepsilon^{-1}\delta A \textnormal{e}^{d A (t-s)}\int_0^{t-s} \textnormal{e}^{(\delta A -\varepsilon^{-1}Id)r}\,\textnormal{d}r\\
   &\qquad \qquad \times \,\pr_{X_F^\zeta} \big[f(u^\varepsilon(s), v_F^\varepsilon(s),v_S^\varepsilon(s))-f(h_u^{\varepsilon,\zeta}(v_S^\varepsilon(s)),h_{X_F^\zeta}^{\varepsilon,\zeta}(v_S^\varepsilon(s)),v_S^\varepsilon(s)) \big]\,\textnormal{d}s\bigg\|_{X_1}\\
    &\quad +\bigg \| \int_0^t \delta A \textnormal{e}^{d A (t-s)}\int_0^{t-s} \textnormal{e}^{(\delta A -\varepsilon^{-1}Id)r}\,\textnormal{d}r\\
   &\qquad \qquad \times \pr_{X_F^\zeta}  \big[\phi(u^\varepsilon(s), v_F^\varepsilon(s),v_S^\varepsilon(s))-\phi(h_u^{\varepsilon,\zeta}(v_S^\varepsilon(s)),h_{X_F^\zeta}^{\varepsilon,\zeta}(v_S^\varepsilon(s)),v_S^\varepsilon(s)) \big]\,\textnormal{d}s\bigg\|_{X_1}\\
   &\quad +C_A L_h \int_0^t \textnormal{e}^{((d+\delta)\omega_A -\varepsilon^{-1})(t-s)}\|\delta A\, \pr_{X_S^\zeta} \big[ u^\varepsilon(s)-h_u^{\varepsilon,\zeta}(v_S^\varepsilon(s))\big]\|_{X_1}\,\textnormal{d}s\\
   &\quad+C_A L_h \int_0^t \textnormal{e}^{(\zeta^{-1}(\varepsilon (d+\delta)\omega_A -1)+N_F^\zeta)(t-s)}\|\delta A\, \pr_{X_S^\zeta} \big[ u^\varepsilon(s)-h_u^{\varepsilon,\zeta}(v_S^\varepsilon(s))\big]\|_{X_1}\,\textnormal{d}s\\
   &\quad +\bigg\| \int_0^t \delta A \textnormal{e}^{d A (t-s)}\int_0^{t-s} \textnormal{e}^{(\delta A -\varepsilon^{-1}Id)r}\,\textnormal{d}r\, \pr_{X_F^\zeta} \textnormal{D}h^{\varepsilon,\zeta}_{X_F^\zeta}(v_S^\varepsilon(s))\delta A\, \pr_{X_S^\zeta} \big[ u^\varepsilon(s)-h_u^{\varepsilon,\zeta}(v_S^\varepsilon(s))\big]\,\textnormal{d}s \|_{X_1}\\
   &\quad +C_A L_h L_\psi \int_0^t \textnormal{e}^{((d+\delta)\omega_A -\varepsilon^{-1})(t-s)}\|z^\varepsilon(s)-h^{\varepsilon,\zeta}(v_S^\varepsilon(s)) \|_{X_1}\,\textnormal{d}s\\
   &\quad+C_A L_h L_\psi\int_0^t \textnormal{e}^{(\zeta^{-1}(\varepsilon (d+\delta)\omega_A -1)+N_F^\zeta)(t-s)}\|z^\varepsilon(s)-h^{\varepsilon,\zeta}(v_S^\varepsilon(s)) \|_{X_1}\,\textnormal{d}s\\
   &\quad +\bigg\| \int_0^t \delta A \textnormal{e}^{d A (t-s)}\int_0^{t-s} \textnormal{e}^{(\delta A -\varepsilon^{-1}Id)r}\,\textnormal{d}r\,\\
   &\qquad\qquad  \pr_{X_F^\zeta} \textnormal{D}h^{\varepsilon,\zeta}_{X_F^\zeta}(v_S^\varepsilon(s))\, \pr_{X_S^\zeta} \big[\psi(u^\varepsilon(s), v_F^\varepsilon(s),v_S^\varepsilon(s))-\psi(h_u^{\varepsilon,\zeta}(v_S^\varepsilon(s)),h_{X_F^\zeta}^{\varepsilon,\zeta}(v_S^\varepsilon(s)),v_S^\varepsilon(s)) \big]\,\textnormal{d}s \bigg\|_{X_1}. 
\end{align*}
This can be further reduced to
\begin{align*}
    &\|  z^\varepsilon(t)-h^{\varepsilon,\zeta}(v_S^\varepsilon(t))\|_{X_1}\\
    &\leq C_A \textnormal{e}^{((d+\delta)\omega_A -\varepsilon^{-1})t}\|u_0-h_u^{\varepsilon,\zeta}(v_{S,0})\|_{X_1}\\
    & \quad + C_A \textnormal{e}^{(\zeta^{-1}(\varepsilon (d+\delta)\omega_A -1)+N_F^\zeta)t} \big(\|v_0-h_{X_F^\zeta}^{\varepsilon,\zeta}(v_{S,0})\|_{X_1} +\delta d^{-1}\|u_0-h_u^{\varepsilon,\zeta}(v_{S,0})\|_{X_1} \big) \\
    &\quad  + C_A (\varepsilon^{-1} L_f +L_\phi)  \int_0^t \big( \textnormal{e}^{((d+\delta)\omega_A -\varepsilon^{-1})(t-s)}+ \delta d^{-1}\textnormal{e}^{(\zeta^{-1}(\varepsilon (d+\delta)\omega_A -1)+N_F^\zeta)(t-s)} \big)\|z^\varepsilon(s)-h^{\varepsilon,\zeta}(v_S^\varepsilon(s))\|_{X_1} \,\textnormal{d}s\\
    &\quad +C_A L_h C(\zeta^{-1}) \delta \int_0^t \big( \textnormal{e}^{((d+\delta)\omega_A -\varepsilon^{-1})(t-s)}+(1+\delta d^{-1})\textnormal{e}^{(\zeta^{-1}(\varepsilon (d+\delta)\omega_A -1)+N_F^\zeta)(t-s)}\big) \|z^\varepsilon(s)-h^{\varepsilon,\zeta}(v_S^\varepsilon(s))\|_{X_1} \,\textnormal{d}s\\ 
    &\quad +C_A L_h L_\psi \int_0^t \big( \textnormal{e}^{((d+\delta)\omega_A -\varepsilon^{-1})(t-s)}+(1+\delta d^{-1})\textnormal{e}^{(\zeta^{-1}(\varepsilon (d+\delta)\omega_A -1)+N_F^\zeta)(t-s)}\big) \|z^\varepsilon(s)-h^{\varepsilon,\zeta}(v_S^\varepsilon(s))\|_{X_1} \,\textnormal{d}s,
\end{align*}
where we applied Proposition \ref{prop:operator projection estimate}.
Hence, we have 
\begin{align*}
     &\|  z^\varepsilon(t)-h^{\varepsilon,\zeta}(v_S^\varepsilon(t))\|_{X_1} \\
     &\leq C_A \textnormal{e}^{((d+\delta)\omega_A -\varepsilon^{-1})t}\|u_0-h_u^{\varepsilon,\zeta}(v_{S,0})\|_{X_1}\\
    & \quad + C_A \textnormal{e}^{(\zeta^{-1}(\varepsilon (d+\delta)\omega_A -1)+N_F^\zeta)t} \big(\|v_0-h_{X_F^\zeta}^{\varepsilon,\zeta}(v_{S,0})\|_{X_1} +\delta d^{-1}\|u_0-h_u^{\varepsilon,\zeta}(v_{S,0})\|_{X_1} \big) \\
    &\quad + C_A( L_f+\varepsilon L_\phi +L_h C(\zeta^{-1})\delta +L_hL_\psi)\times \\
    &\qquad \times \int_0^t \bigg((1+\varepsilon^{-1} )\textnormal{e}^{((d+\delta)\omega_A -\varepsilon^{-1})(t-s)}\\
    &\qquad \qquad +(1+\varepsilon^{-1}\delta+\delta d^{-1})\textnormal{e}^{(\zeta^{-1}(\varepsilon (d+\delta)\omega_A -1)+N_F^\zeta)(t-s)}\bigg) \|z^\varepsilon(s)-h^{\varepsilon,\zeta}(v_S^\varepsilon(s))\|_{X_1} \,\textnormal{d}s.
\end{align*}
Then, by applying a generalized version of Gronwall's inequality we obtain
\begin{align*}
    \|  z^\varepsilon(t)-h^{\varepsilon,\zeta}(v_S^\varepsilon(t))\|_{X_1} \leq C\textnormal{e}^{-ct}\big(\|v_0-h_{X_F^\zeta}^{\varepsilon,\zeta}(v_{S,0})\|_{X_1} +\|u_0-h_u^{\varepsilon,\zeta}(v_{S,0})\|_{X_1} \big)
\end{align*}
for some constants $C,c>0$.
This proves the assertion.
   
\end{proof}
We first want to point out that this result implies the attraction to the slow manifold at an exponential rate. Since we have imposed a global Lipschitz assumption on the nonlinearity, the attractivity result is actually global. Once one reduces to local Lipschitz conditions, the attractivity will become local in phase space. This situation is well-known for classical finite-dimensional slow manifolds.

\begin{remark}
   We also note that the constant $c$ depends on the parameters $\varepsilon$ and $\zeta$ in such a way that $\lim_{\varepsilon,\zeta\to 0}c(\varepsilon,\zeta)=\infty$.
\end{remark}

\subsection{Approximation of the slow flow}
In this section we connect the previous results with the convergence of solutions of the fast-slow system to its limit system.

So far we have measured the distance of the slow manifold $S_{\varepsilon,\zeta}$ to a subset of the critical manifold $S_{0,\zeta}$, where
\begin{align}
    S_{0,\zeta}:=\{ (h^0(v^0),v^0) \in S_0:~ \pr_{X_F^\zeta}v^0=0\}.
\end{align}
However we do need to adapt the slow system to the splitting we introduced as otherwise $S_{0,\zeta}$ might not be invariant under the slow flow.
Therefore, we introduce the slow subsystem
\begin{align}
    \begin{split}\label{slow subsyst}
        0&=-u_\zeta^0(t)+ f(u_\zeta^0(t),v_\zeta^0(t))\\
        0&= \pr_{X_F^\zeta} v_\zeta^0(t)\\
        \partial_t v_\zeta(t)&= d A v_\zeta^0(t) +\pr_{X_S^\zeta} \psi(u_\zeta^0(t),v_\zeta^0(t))\\
        v_\zeta^0(0)&= \pr_{X_S^\zeta} v_0.
    \end{split}
\end{align}
We note that $S_{0,\zeta}$ is indeed invariant under the slow flow generated by \eqref{slow subsyst}.

\begin{remark}
    Here, again, the choice of the parameter $\delta$ is crucial. 
    For the case $\delta=\delta(\varepsilon)$ such that $\lim_{\varepsilon\to 0} \delta(\varepsilon)=0$ the limit system reduces to a semi-linear system.
    However, when $\delta$ is constant, the remaining system becomes quasi-linear, as the term $\delta A\, \pr_{X_S^\zeta} u^0$ remains, and therefore requires more advanced techniques.
\end{remark}

\begin{prop}\label{slow flow approx}
    Let the assumptions of Section \ref{Sec:Assumptions} hold.
    Then, there is a constant $C>0$ such that for all $t\geq 0$ and all $\zeta>0$ small enough it holds that
    \begin{align*}
        \|v^0(t)-v_\zeta^0(t)\|_{X_1}\leq C \textnormal{e}^{(-\zeta^{-1}+ N_F^\zeta)t}\|\pr_{X_F^\zeta}v_0\|_{X_1}.
    \end{align*}
\end{prop}
\begin{proof}
    We compute
    \begin{align*}
        v^0(t)-v_\zeta^0(t)&=\textnormal{e}^{d A t}\big(v_0-\pr_{X_S^\zeta}v_0\big) +\int_0^t \textnormal{e}^{d A(t-s)} \big(\psi(h^0(v^0(t)),v_\zeta^0(t))- \pr_{X_S^\zeta}\psi(h^0(v^0_\zeta(t)),v_\zeta^0(t))\big)\, \textnormal{d}s\\
        &=\textnormal{e}^{d A t}\pr_{X_F^\zeta}v_0+ \int_0^t \textnormal{e}^{d A(t-s)} \pr_{X_S^\zeta}\big(\psi(h^0(v^0(t)),v_\zeta^0(t))- \psi(h^0(v^0_\zeta(t)),v_\zeta^0(t))\big)\, \textnormal{d}s\\
        &\quad +  \int_0^t \textnormal{e}^{d A(t-s)} \pr_{X_F^\zeta} \psi(h^0(v^0(t)),v_\zeta^0(t))\, \textnormal{d}s.
    \end{align*}
    Thus,
    \begin{align*}
        &\|v^0(t)-v_\zeta^0(t)\|_{X_1}\leq \\
        &\leq C_A \textnormal{e}^{(-\zeta^{-1}+ N_F^\zeta)t}\|\pr_{X_F^\zeta}v_0\|_{X_1} + C_A L_\psi \int_0^t \textnormal{e}^{(-\zeta^{-1}+ N_F^\zeta)(t-s)}\big(\|h^0(v^0(s))\|_{X_1}+\|v^0(s)\|_{X_1}\big)\,\textnormal{d}s\\
        &\quad + C_A L_\psi \int_0^t \textnormal{e}^{d \omega_A (t-s)}\big(\|h^0(v^0(s))- h^0(v_\zeta^0(s))\|_{X_1}+\|v^0(s)-v_\zeta^0(s)\|_{X_1}\big)\,\textnormal{d}s.
    \end{align*}
    Using the assumption that $L_f<1$ and the estimate $\|v^0(t)\|_{X_1}\leq C \|v_0\| \textnormal{e}^{d \omega_A t}$, which follows from the variation of constants formula together with applying a Gronwall type inequality, we obtain
    \begin{align*}
         \|v^0(t)-v_\zeta^0(t)\|_{X_1}
          &\leq C_A \textnormal{e}^{(-\zeta^{-1}+ N_F^\zeta)t}\|\pr_{X_F^\zeta}v_0\|_{X_1} + \frac{C}{\zeta^{-1}-N_F^\zeta}\|v_0\|_{X_1}\\
           &\quad + C_A L_\psi (1+L_{h^0})\int_0^t \textnormal{e}^{d \omega_A (t-s)}\big(\|v^0(s)-v_\zeta^0(s)\|_{X_1}\big)\,\textnormal{d}s.
    \end{align*}
    Applying a Gronwall-type inequality this yields
    \begin{align*}
         \|v^0(t)-v_\zeta^0(t)\|_{X_1}&\leq \bigg( C_A \|\pr_{X_F^\zeta}v_0\|_{X_1} + \frac{C}{\zeta^{-1}-N_F^\zeta}\|v_0\|_{X_1} \bigg) \exp\bigg(\tilde C\int_0^t\textnormal{e}^{d \omega_A (t-s)}\,\textnormal{d}s\bigg)\\
         &\quad + C_A \|\pr_{X_F^\zeta}v_0\|_{X_1} \int_0^t (-\zeta^{-1}+ N_F^\zeta) \textnormal{e}^{(-\zeta^{-1}+ N_F^\zeta)s} \exp\bigg(\tilde C \int_s^t \textnormal{e}^{d \omega_A (t-r)}\,\textnormal{d}r\bigg)\,\textnormal{d}s\\
         &\leq \bigg( C_A \|\pr_{X_F^\zeta}v_0\|_{X_1} + \frac{C}{\zeta^{-1}-N_F^\zeta}\|v_0\|_{X_1} \bigg) \exp\bigg(\tilde C \frac{\textnormal{e}^{d \omega_A t}-1}{d\omega_A}\bigg)\\
         &\quad + C_A \|\pr_{X_F^\zeta}v_0\|_{X_1} \int_0^t (-\zeta^{-1}+ N_F^\zeta) \textnormal{e}^{(-\zeta^{-1}+ N_F^\zeta)s} \exp\bigg(\tilde C  \frac{\textnormal{e}^{d \omega_A (t-s)}-1}{d\omega_A}\bigg)\,\textnormal{d}s\\
         &\leq C\bigg( \|\pr_{X_F^\zeta}v_0\|_{X_1} + \frac{\|v_0\|_{X_1}}{\zeta^{-1}-N_F^\zeta} \bigg),
    \end{align*}
    which concludes the proof of the proposition.
\end{proof}

The last step in the proof of the existence of a slow manifold is to show the convergence of the semi-flows of the solution.
\begin{prop}
    Under the assumptions of Section \ref{Sec:Assumptions} there exists a constant $C>0$ such that for all $t\geq 0$ and all $\varepsilon,\zeta>0$ small enough it holds that
    \begin{align*}
        \left\|\begin{pmatrix}
            u^\varepsilon(t)-h^0(v_\zeta^0(t))\\ v^\varepsilon(t)-v^0_\zeta(t)
        \end{pmatrix} \right\|_{X_1\times X_1} \leq C\bigg( \|\pr_{X_F^\zeta}v_0\|_{X_1} + \frac{\|v_0\|_{X_1}}{\zeta^{-1}-N_F^\zeta} \bigg) + C(\varepsilon)\big( \|u_0-h^0(v_0)\|_{X_1}+\|v_0\|_{X_1}\big).
    \end{align*}
\end{prop}
\begin{proof}
    Combining the results of Proposition \ref{prop:distance to crit manifold} and Proposition \ref{slow flow approx} together with estimate \eqref{convergence of trajectories general} yields the desired result.
\end{proof}
This completes the last step in the proof of the existence of the slow manifold.\\

We state the following observations to conclude this section. 
\begin{remark}
   In applications one rarely has the global Lipschitz continuity of the nonlinearities. 
   Ways to overcome this would be to show a priori $L^\infty$-estimates for $u$ and $v$, or which is much more common, to use suitably chosen cut-off functions in a dissipative system. In the case of cutoff techniques, one can always aim to maximally extend a local slow manifold by patching together different local results. This does not add to our arguments here that focus on the structural problems of reducing reaction-diffusion systems to cross-diffusion systems.
\end{remark}

\begin{remark}
    The choice to consider both equations of system \eqref{general system}  in the Banach space $X$ is motivated by the fact that both the equations are derived from a two species fast-reaction system.
    However, it is possible to assume that the $u$-equation is posed in the Banach space $X$ whereas the $v$-equation is stated in the Banach space $Y$. Then one would have $D(\Delta)=:X_1$ densely contained in $X$ and additionally $D(\Delta)=:Y_1$ densely contained in $Y$.
    Yet, the structure of the proofs remains the same.
\end{remark}

\begin{remark}
    We want to remark that the method also works for more general nonlinear functions where $g: X_1\times X_1\to X_\gamma$ for $\gamma \in (0,1]$. 
    In this case one needs to linearize the function around the critical set $g(u,v)=0$ and consider a modified linear operator $\tilde A$. See \cite{kuehn2023fast} for more details. 
\end{remark}

\section{Slow Manifolds for Fast Chemical Reaction systems}\label{section:4}

Having established the abstract framework, we now return to the concrete class of PDEs given by~\eqref{Problem.Original} and proceed with the proof of Theorem~\ref{thm:fast_reaction}. 

\subsection{Sketch of the Proof}

We begin with the definition for the critical manifold $S_0$, where the notation $g$ has been employed similarly to the previous sections of the form 
\begin{align*}
    g(x,y):=- x + \kappa \tilde f(x,y), \quad 0 \le x\le y <\infty. 
\end{align*}
Let $B_R(0,0)$, with $R>0$, be the ball in $H^2(I)\times H^2(I)$ of center at $(0,0)$ and radius $R$, and $(\eta,\rho)\in B_R(0,0)$, where we take the sum norm for the product. Direct computations with Fr\'echet derivatives show 
$D_1 g(\eta,\rho)  = - I  - 2 \kappa (\rho-\eta) \in L(H^2(I))$. Let $C_{*}$ be the sharp constant such that 
\begin{align}
    \|w\|_{L^{\infty}(I)} \le C_* \|w\|_{H^1(I)}, \quad \forall w\in H^1(I). 
    \label{SharpSEconst}
\end{align}
Then, it is obvious   
\begin{align*}
\|-2 \kappa (\rho-\eta)\|_{L(H^2(I))} &= 2 \kappa \sup_{\|w\|_{H^2(I)}=1} \| (\rho-\eta)w\|_{H^2(I)}  \le 6 \kappa C_* \|\rho-\eta\|_{H^2(I)} . 
\end{align*}
Therefore, if  
\begin{align}
    12C_*R < \frac{1}{\kappa},
    \label{SmallKappa}
\end{align} 
we can apply the geometric series theorem to claim that the operator $D_1g(\eta,\rho)$ is an isomorphism (i.e., linear, bounded and invertible) on $H^2(I)$. By the implicit function theorem on Banach spaces, there exists a unique continuously differentiable function $h_\kappa:B_R(0)\subset H^2(I) \to H^2(I)$ such that 
$$ g(\widetilde \eta,\widetilde \rho)= -h(\widetilde \rho) + \kappa(\widetilde \rho-h(\widetilde \rho))^2 =0, \,\forall \widetilde \rho \in B_R(0). $$ 
In other words, under the condition \eqref{SmallKappa}, from the relation $g(\eta,\rho)=-\eta + \kappa (\rho-\eta)^2=0$  we can express $\eta=h_\kappa(\rho)$. This allows us to define the  following manifold 
\begin{align}
    S_0 := \{ (h_\kappa(\rho),\rho)|\, \rho \in B_R(0) \}  \subset H^2(I) \times H^2(I),
\end{align}
called the \textit{critical manifold}. 

\medskip

It is useful to note that if the unique classical solution $(u,v)$ to \eqref{Problem.Limiting}-\eqref{Condition.Initial.Limiting} exists globally, then it is not a-priori guaranteed that the trajectory $\{(u(t),v(t):t\in [0,\infty)\}$ stays on the critical manifold. The condition for this stay will be pointed out in the following theorem based on the analysis in Section \ref{Sec:FRL}. Let $C_{d,p}^{\mathrm{HS}}$ be the sharp constant such that 
\begin{align}
    \|    e^{td\partial_{xx}^2} \partial_x w\|_{L^p(I)} \hspace{0.23cm}  \le  C_{d,p}^{\mathrm{HS}} e^{-\lambda_1 t} t^{-\frac{1}{2}}  \|w\|_{L^p(I)},  \quad  t>0, 
    \label{HS}
\end{align}
 where $1\le p\le \infty$, and $\lambda_1$ stands for the first positive eigenvalue of the Neumann Laplacian $-\partial_{xx}^2$. We also denote  
\begin{gather}
K_M:= (2K_{0,M} + 3K_{1,M})|I|^{1/2} + 3K_{2,M} + 2 |I|^{1/2}K_{1,M}^2, \label{K:Def}
\end{gather}
with
\begin{gather*}
K_{0,M}:= C_*M+\frac{a}{c}, \quad  K_{1,M} := C_*M + C_{d,\infty}^{\mathrm{HS}} 
 \frac{\pi^{1/2}}{\lambda_1^{1/2}} \sup_{0\le x,y\le K_{0,M}} |\psi(x,y)|, \\
K_{2,M}:= M + 3C_{d,2}^{\mathrm{HS}} \frac{\pi^{1/2}}{\lambda_1^{1/2}}  \left(   \sum_{i=1}^2 \sup_{0\le x,y\le K_{0,M}} |\psi'_i(x,y)| \right) |I|^{1/2}K_{1,M}.
\end{gather*}
 
\begin{theorem}[Critical manifold for the limiting system]
\label{Theo:CritMani}
Assume that  $(u_{\mathrm{in}},v_{\mathrm{in}}) \in B_M(0,0)\subset H^2(I)\times H^2(I)$ with the radius $M>0$ satisfying   
    \begin{align}
       12C_*K_M < \frac{1}{\kappa},
    \end{align}
where $K_M$ is defined by \eqref{K:Def}. Then, all trajectories $\{(u(t),v(t):t\in [0,\infty)\}$, where $(u,v)$ is the unique classical solution to \eqref{Problem.Limiting}-\eqref{Condition.Initial.Limiting}, started from $B_M(0,0)$, stay in the critical manifold $S_0$.  
\end{theorem}

\begin{proof} Thanks to Proposition \ref{Theo.FastReactionLimit}, we have the global existence of a unique classical solution $(u,v)$ to \eqref{Problem.Limiting}-\eqref{Condition.Initial.Limiting}, and furthermore, $0\le u\le v\le C_*M+a/c=K_{0,M}$ since $\|v\|_{L^\infty(I)}\le C_*\|v\|_{H^1(I)}\le C_*M$ by   the embedding  \eqref{SharpSEconst}. The following computations directly come from the smoothing effect \eqref{HS}. Indeed,   
\begin{align*}
    \|\partial_x v(t)\|_{L^\infty(I)} & \le \|\partial_x v_0\|_{L^\infty(I)} + \int_0^t \|\partial_x   e^{(t-s)d\partial_{xx}^2} \psi(u(s),v(s))  \|_{L^\infty(I)} ds\\
    & \le  C_*M + C_{d,\infty}^{\mathrm{HS}} 
 \frac{\pi^{1/2}}{\lambda_1^{1/2}} \sup_{0\le x,y\le K_{0,M}} |\psi(x,y)|  = K_{1,M},
\end{align*}
where we have use $\| \partial_x v_{\mathrm{in}}\|_{L^\infty(I)} \le C_*\|v_{\mathrm{in}}\|_{H^2(I)} \le C_*M$. 
Moreover, since $0\le u \le v$, the algebraic equation $-u+(v-u)^2=0$ accordingly gives the pointwise relation $u=v-\sqrt{4v+1}/2-1/2$, and therefore,  $$   \|\partial_x u(t)\|_{L^\infty(I)} \le 2  \|\partial_x  v(t)\|_{L^\infty(I)} \le 2  K_{1,M},$$ for all $t>0$. 
Using the smoothing effect \eqref{HS} again, 
\begin{align*}
    \|\partial_{xx}^2 v(t) \|_{L^2(I)} & \le \|v_0\|_{H^2(I)} + C_{d,2}^{\mathrm{HS}} \left(\int_0^t e^{-\lambda_1(t-s)} (t-s)^{-1/2} ds\right) \|\partial_x\psi(u,v)\|_{L^\infty(0,T;L^2(I))} \\
    & \le M + C_{d,2}^{\mathrm{HS}} \frac{\pi^{1/2}}{\lambda_1^{1/2}}  \left(   \sum_{i=1}^2 \sup_{0\le x,y\le K_{0,M}} |\psi'_i(x,y)| \right) (3|I|^{1/2}\|\partial_x  v\|_{L^\infty(I_T)}) \\
    & \le M + C_{d,2}^{\mathrm{HS}} \frac{\pi^{1/2}}{\lambda_1^{1/2}}  \left(   \sum_{i=1}^2 \sup_{0\le x,y\le K_{0,M}} |\psi'_i(x,y)| \right)  3|I|^{1/2}K_{1,M} = K_{2,M}.
\end{align*}
Moreover, it follows from $u=v-\sqrt{4v+1}/2-1/2$ that 
\begin{align*}
    \partial_{xx}^2 u = \partial_{xx}^2 v - \frac{\partial_{xx}^2 v}{\sqrt{4v+1}} + \frac{2|\partial_{x}v|^2}{\sqrt{(4v+1)^3}},
\end{align*}
and therefore,
$$\|\partial_{xx}^2  u(t)\|_{L^2(I)} \le 2\|\partial_{xx}^2 v(t)\|_{L^2(I)} + 2 |I|^{1/2} \|\partial_x  v(t)\|_{L^\infty(I)}^2 \le 2K_{2,M} + 2 |I|^{1/2}K_{1,M}^2.$$ 
Altogether, we get
\begin{align*}
    \|(u(t),v(t))\|_{H^2(I)\times H^2(I)} \le  (2K_{0,M} + 3K_{1,M})|I|^{1/2} + 3K_{2,M} + 2 |I|^{1/2}K_{1,M}^2 = K_M,
\end{align*}
 for all $t>0$. Hence, the conclusion of this theorem comes from \eqref{SmallKappa}.  
\end{proof}

\begin{remark} 
\label{Remark.Theo.FastReactionLimit}
\hfill
\begin{itemize} 
    \item The global existence of a unique strong solution to the limiting system \eqref{Problem.Limiting}-\eqref{Condition.Initial.Limiting}  has been concluded by Proposition \ref{Theo.FastReactionLimit}. In fact, a direct argument can be performed from the equation for $v$ to show that $v\in L^\infty(I_\infty)$ by observing $\psi(u,v)\le (a-cv)v$, however, the uniqueness of the first component is not guaranteed from the algebraic equation $-u+(v-u)^2=0$ since we do not know $v\ge u$. Proposition \ref{Theo.FastReactionLimit} solved this point by showing that $(u,v)$ is the limit of $(u^\varepsilon,v^\varepsilon-u^\varepsilon)$ with $(u^\varepsilon,v^\varepsilon)=(u_1^\varepsilon,u_1^\varepsilon+u_2^\varepsilon)$, where $(u_1^\varepsilon,u_2^\varepsilon)$ is the unique, non-negative  global classical solution to \eqref{Problem.u1u2}-\eqref{Cond.Initial.Problem.u1u2}.   
    
    \item A quick glance at \eqref{Problem.Limiting}-\eqref{Condition.Initial.Limiting} directly shows that the initial value of the component $u$, i.e. $u(0)$, is defined pointwise via the algebraic equation $-u+(v-u)^2=0$, and therefore, may be different from $u^\varepsilon(0)=u_{\mathrm{in}}$. According to the proof of Proposition  \ref{Theo.FastReactionLimit}, we can see that   $v(0)=v_{\mathrm{in}} =v^\varepsilon(0)$.  However, although $(u_1^\varepsilon(0),u_2^\varepsilon(0))=(u_{\mathrm{in}},v_{\mathrm{in}}-u_{\mathrm{in}})$, we do not have $$(u_1(0),u_2(0))=(u_{\mathrm{in}},v_{\mathrm{in}}-u_{\mathrm{in}})$$ since $u_1,u_2$ are determined by $(u_1)^{1/2} + u_1 = v$ and $u_2 + u_2^2 = v$, see  \eqref{Theo.FastReacLimit.Proof1}. This shows the existence of the so-called \textit{initial layer}.    
    
\end{itemize}
\end{remark}

Next, we prove the existence of a slow manifold for the fast chemical reaction system \eqref{Problem.Original}-\eqref{Condition.Initial} with the main result presented in Theorem \ref{thm:fast_reaction}. By the analysis in Section \ref{section:3}, one of the most important and technically challenging points is to estimate the distance between the slow flow $\{(u^\varepsilon(t),v^\varepsilon(t)):[0,\infty) \}$ and the critical one $\{(u(t),v(t)):[0,\infty) \}$, or in other words, the convergence rate from the slow flow to the critical one. Moreover, due to the above remark, since the initial data may affect the convergence rate, computations on the initial layer will be carried out as follows. 

\medskip

As a direct consequence of Theorem \ref{thm:abstract slow manifold},   Lemma \ref{Lemma.InitialLayer},  Proposition \ref{Lemma.ConvRateULinfL2} and Lemma  \ref{Lemma.check assumptions}, we obtain the slow manifold for the fast chemical reaction system \eqref{Problem.Original}-\eqref{Condition.Initial} stated in Theorem \ref{thm:fast_reaction}. 

 \begin{lemma} 
\label{Lemma.InitialLayer}
Assume that $(u_{\mathrm{in}},v_{\mathrm{in}})$ satisfies the assumption of Theorem \ref{Theo:CritMani}. Then,
\begin{align}
    \|u_{\mathrm{in}}-u(0)\|_{H^2(I)}  &\le C_M  \varepsilon_{\mathrm{in}},
    \label{Lemma.InitialLayer.State1}
\end{align}
with $\eps_{\mathrm{in}}$ is in \eqref{Assumption.EpsilonIN}
\end{lemma}

The proof of the above lemma is presented in Section \ref{Sec:InitialLayer}. We will estimate the convergence rates in the following proposition, where the initial layer effect, Lemma \ref{Lemma.InitialLayer}, is expressed via the parameter $\varepsilon_{\mathrm{in}}$ and its proof is presented in Sections \ref{Sec:LinfL2} and \ref{Sec:LinfH2}. From there, the result of slow manifolds for the fast chemical reaction system in Theorem \ref{thm:fast_reaction} follows. Let us denote $$(U^\varepsilon,V^\varepsilon):=(u^\varepsilon-u,v^\varepsilon-v).$$ 
By subtracting side by side the equations for $(u^\varepsilon,v^\varepsilon)$, $(u,v)$, we obtain the rate system as follows
\begin{align}  \label{Problem.DistanceEquation} 
 \left\{
\begin{array}{llllllll}
\partial_t V^\varepsilon &=& d \partial_{xx}^2 V^\varepsilon + (\psi(u+U^\varepsilon,v+V^\varepsilon)-\psi(u,v)) + \delta \partial_{xx}^2 u^\varepsilon  ,  \vspace{0.15cm}  \\
\partial_t U^\varepsilon &=& (d+\delta)  \partial_{xx}^2 U^\varepsilon  + (\phi(u+U^\varepsilon,v+V^\varepsilon)-\phi(u,v)) + \xi^\varepsilon  \vspace*{0.15cm} \\
&& \hspace*{2.1cm} +\, \dfrac{1}{\varepsilon}  (g(u+U^\varepsilon,v+V^\varepsilon)-g(u,v)) ,  
\end{array}
\right.  
\end{align}
where 
\begin{align*}
\xi^\varepsilon : = - \partial_t u + (d+\delta(\varepsilon)) \partial_{xx}^2 u  + \phi(u,v) . 
\end{align*}
Studying this system will give the convergence rate in $L^\infty(0,T;L^2(I))$ of order $\mathcal O(\varepsilon + \varepsilon_{\mathrm{in}}  +  \delta)$, see \eqref{Lemma.ConvRateULinfL2.State1}. However, to obtain the rate in $L^\infty(0,T;H^2(I))$, considering higher derivatives of $U^\varepsilon,V^\varepsilon$ is needed. More clearly, differentiating the rate system \eqref{Problem.DistanceEquation} gives
\begin{align}  \label{Problem.Dx.DistanceEquation} 
 \left\{
\begin{array}{llllllll}
\partial_t (\partial_x V^\varepsilon) &=& d \partial_{xxx}^3 V^\varepsilon + \partial_x(\psi(u+U^\varepsilon,v+V^\varepsilon)-\psi(u,v)) + \delta \partial_{xxx}^3 u^\varepsilon  ,  \vspace{0.15cm}  \\
\partial_t (\partial_x U^\varepsilon) &=& (d+\delta)  \partial_{xxx}^3 U^\varepsilon  + \partial_x (\phi(u+U^\varepsilon,v+V^\varepsilon)-\phi(u,v)) + \partial_x \xi^\varepsilon  \vspace*{0.15cm} \\
& & \hspace*{2.25cm} +\, \dfrac{1}{\varepsilon}  \partial_x(g(u+U^\varepsilon,v+V^\varepsilon)-g(u,v)) .   
\end{array}
\right.  
\end{align}
Studying the system \eqref{Problem.Dx.DistanceEquation} will give the convergence rate in $L^\infty(0,T;H^2(I))$ of order $\mathcal O(\varepsilon + \varepsilon_{\mathrm{in}}  +  \delta)$, see \eqref{Lemma.ConvRateUVLinfH2.State1}. 
The key to obtaining these rate estimates is explained as follows. Recall that $(u,v)$ stays on the critical manifold, formed by the function $h$ that is defined via the function $g$. Therefore, differentiable properties of $g$, or in other words, geometrical properties of the critical manifold, play an important role. For the sake of convenience, we denote $$\displaystyle
(g_1',g_2'):=\left( \frac{\partial g}{\partial x}, \frac{\partial g}{\partial y} \right), \quad (g_{12}'',g_{21}''):=\left(\frac{\partial^2 g}{\partial x \partial y},\frac{\partial^2 g}{\partial y \partial x}\right).$$ 
According to the definition of $g$, we note here
\begin{equation}
 g_1':=\frac{\partial g}{\partial x}=-1 - 2\kappa (y-x) \le -1 ,     
\label{DerivativeOfg}
\end{equation}
which is significant in treating the singular terms 
\begin{gather*}
    \frac{1}{\varepsilon} \int_I  (g(u+U^\varepsilon,v+V^\varepsilon)-g(u,v)) \partial^{(k)} U^\varepsilon /\partial {x^k}, \, k\in \{0;2\}, \\
    \text{and} \quad \frac{1}{\varepsilon} \int_I \partial_x (g(u+U^\varepsilon,v+V^\varepsilon)-g(u,v)) \partial_{xxx}^3 U^\varepsilon,
\end{gather*}
respectively in Lemmas \ref{Lem.UEnergy}, \ref{Lem.NablaUEnergy} and \ref{Lemma.ConvRateULinfH2}.

\begin{prop} 
\label{Lemma.ConvRateULinfL2}    
Assume that $(u_{\mathrm{in}},v_{\mathrm{in}})$ satisfies the assumption of Theorem \ref{Theo:CritMani}. Then,
\begin{align}
\|(u^\varepsilon,v^\varepsilon)-(u,v)\|_{L^\infty(0,T;L^2(I))^2} 
 + \|(u^\varepsilon,v^\varepsilon)-(u,v)\|_{L^2(0,T;H^1(I))^2} 
\le C_{T} (   \varepsilon + \varepsilon_{\mathrm{in}}  +  \delta  ). 
\label{Lemma.ConvRateULinfL2.State1}
\end{align}
Assume furthermore that  
\begin{align}
    \lim_{\varepsilon \to 0} \left( \frac{\varepsilon_{\mathrm{in}}^2}{\varepsilon}  +  \frac{\delta^2}{\varepsilon} \right) < \infty.
    \label{Assumption.Compare.Eps.Del}
\end{align}
Then, for any $T>0$, 
    \begin{align}
    \|(u^\varepsilon,v^\varepsilon)-(u,v)\|_{L^\infty(0,T;H^2(I))^2}  \le C_{T} (\varepsilon +  \varepsilon_{\mathrm{in}} + \delta ).
    \label{Lemma.ConvRateUVLinfH2.State1}
\end{align} 
\end{prop}

Due to Theorem \ref{Theo:CritMani}, we can take $L_{\widetilde f}=12 C_*K_M$ and, therefore, the Lipschitz constant of $f$ in \eqref{N2} is $L_f = \kappa L_{\widetilde f} < 1$. Moreover, the Lipschitz constant $L_\phi,L_\psi$ can be directly calculated and are positive finite numbers. To end the proof of Theorem \ref{thm:fast_reaction}, we will check the spectral gap assumption \eqref{G} in the following lemma. 

\begin{lemma}\label{Lemma.check assumptions}
    Let $\varepsilon,\delta,\zeta>0$ be parameters and assume that they satisfy
    \begin{align}\label{Assumption.three.parameters}
        \varepsilon \zeta^{-1}< \frac{1}{4}(1-\kappa L_{\widetilde f})\quad \textnormal{and}\quad \lim_{\varepsilon\to 0}\frac{\delta}{\varepsilon \zeta^{-1/2}}=0.
    \end{align}
    Then, the fast chemical reaction \eqref{Problem.Original}-\eqref{Condition.Initial} satisfies the spectral gap condition \eqref{G}.
\end{lemma}
\begin{proof}
    The linear operator $A:H^2(I)\to L^2(I)$ is given by $\partial_{xx}^2$ satisfying Neumann boundary conditions.
    We note that the Laplacian generates a bounded analytic $C_0$-semigroup $(e^{t\Delta})_{t>0}$ on $H^{\alpha}(I)$ for any $\alpha\geq 0$ given by
    \begin{align*}
        e^{t\Delta}\varphi(x)= \sum_{k\in \mathbb{N}_0}e^{-k^2t} \widehat{\varphi}(k) e^{ikx},
    \end{align*}
    where $\hat{\varphi}(k)$ denotes the $k$-th Fourier coefficient.
    Moreover, we have that $\omega_A=0$ and $C_A=1$.
    Then, we can introduce the splitting for the space $X=L^2(I)= X_S^\zeta\oplus X_F^\zeta$ as follows.
    Let $-(k_0+1)^2<\zeta^{-1}<-k_0^2$ for some $k_0\in \mathbb{N}$.
    Then, we set 
    \begin{align}
        Y_S^\zeta=\textnormal{span}\{[x\mapsto e^{ikx}]: k\leq k_0-1, k\in \mathbb{N}_0\}
        \intertext{and similarly}
        Y_F^\zeta= \textnormal{cl}_{L^2(I)}\big(\textnormal{span}\{[x\mapsto e^{ikx}]: k\geq k_0, k\in \mathbb{N}_0\}\big).
    \end{align}
    The first set of assumptions can be easily checked using the properties of the Laplacian.
    Next, we have that for every $x_S\in X_S^\zeta$ and $t\geq 0$ the following estimate holds
    \begin{align*}
        \| e^{-t\Delta} x_S\|_{L^2(I)} \leq e^{(k_0-1)^2t}.
    \end{align*}
    Hence, we can take 
    \begin{align*}
        N_S^\zeta= -\zeta^{-1}-(k_0-1)^2
    \end{align*}
   and similarly we can take 
   \begin{align*}
      N_F^\zeta = -\zeta^{-1}- k_0^2+k_0+1,
   \end{align*}
   so that $N_S^\zeta-N_F^\zeta= k_0\geq c\zeta^{-1/2}$.
   It remains to check the spectral gap assumption \eqref{G}, which in the setting of this section reads
   \begin{align*}
       \frac{\kappa L_{\widetilde f}+\varepsilon L_\phi}{1- 2\varepsilon\zeta^{-1} -\frac{\varepsilon}{2}((k_0-1)^2+k_0^2 -k_0-1) } +2\frac{\delta(1-d^{-1})(\kappa L_{\widetilde f}+\varepsilon L_\phi)+2\varepsilon L_\psi}{\varepsilon k_0}<1.
   \end{align*}
   We observe that for $\delta$ satisfying $\lim_{\varepsilon \to 0} \frac{\delta}{\varepsilon \zeta^{-1/2}}=0$ the second term in the above inequality can be made arbitrarily small.
   For the first term we note that $\kappa L_{\widetilde f}<1$ and thus there exists an $\varepsilon_0$ such that for all $\varepsilon\leq \varepsilon_0$ the expression $\kappa L_{\widetilde f}/(1-c\varepsilon\zeta^{-1})$ can be controlled by a constant $0<C<1$.
   Thus, the system \eqref{Problem.Original}-\eqref{Condition.Initial} satisfies the splitting assumptions of Section \ref{Sec:Assumptions}.
\end{proof}

\subsection{Uniform-in-\texorpdfstring{$\eps$}~ Bounds and Fast Reaction Limit}
\label{Sec:FRL}

In this part, we will perform an analysis of fast reaction limits, cf. Proposition \ref{Theo.FastReactionLimit}, which will guarantee not only the convergence of the slow-semiflow $\{(u^\varepsilon,v^\varepsilon)\}$ to the critical flow $(u,v)$ in a suitable sense but also the global existence of a unique classical solution to the critical system \eqref{Problem.Limiting}. For this analysis, compactness of the $\varepsilon$-depending solution is necessary, cf. Lemma \ref{Lem.Energy}.  

\medskip
For simplicity, we set $\kappa=1$ since it does not affect the analysis in this subsection. Let us define $$(u_1^\varepsilon,u_2^\varepsilon):=(u^\varepsilon,v^\varepsilon-u^\varepsilon),$$  
which, after plugging into the system \eqref{Problem.Original}, satisfies 
\begin{align}  
\label{Problem.u1u2}
 \left\{
\begin{array}{llrllll}
\partial_t u_1^\varepsilon - (d+\delta)  \partial_{xx}^2 u_1^\varepsilon = \varphi_1(u_1^\varepsilon,u_2^\varepsilon)  + \dfrac{1}{\varepsilon}  (-u_1^\varepsilon+ (u_2^\varepsilon)^2)  & \text{in } I_\infty,  \vspace{0.15cm}  \\
\partial_t u_2^\varepsilon - d \partial_{xx}^2 u_2^\varepsilon = \varphi_2(u_1^\varepsilon,u_2^\varepsilon) - \dfrac{1}{\varepsilon}  (-u_1^\varepsilon+(u_2^\varepsilon)^2) & \text{in } I_\infty, 
\end{array}
\right.  
\end{align} 
subject to the Neumann boundary conditions
\begin{align}
    \partial_\nu u_1^\varepsilon = \partial_\nu u_2^\varepsilon =0 \quad \text{on } \Gamma \times (0,\infty),
\end{align}
and the initial condition 
\begin{align}
    (u_1^\varepsilon(0),u_2^\varepsilon(0)) = (u_{\mathrm{in}},v_{\mathrm{in}}-u_{\mathrm{in}}) \quad \text{on } I, 
    \label{Cond.Initial.Problem.u1u2}
\end{align}
in which $\varphi_1,\varphi_2$ are defined by
\begin{align*}
\varphi_1(x,y) := \phi(x,x+y), \quad 
\varphi_2(x,y) := \psi(x,x+y)-\phi(x,x+y).
\end{align*}
For the section's purpose, we first notice that the global existence of the above system for each $\varepsilon>0$ can be obtained from the existence of a non-negative local solution as a direct consequence of the following lemma, where 
a priori estimates for $u_1^\varepsilon,u_2^\varepsilon$ will be established by multiplying their equations by $(u_i^\varepsilon)^{ip}$, $p>0$.

\begin{lemma} 
\label{Lem.Energy}
It holds, for $j=1,2$,  
    \begin{align} 
    \|u_j^\varepsilon\|_{L^{\infty}(I_\infty)} & \le
    \sum_{i=1}^2  \|u_i^{\mathrm{in}}\|_{L^{\infty}(I)}^{i/j} + \left( \frac{a}{b+c}\right)^{1/j} + \left( \frac{2b}{c}\right)^{1/j} + \left( \frac{2a}{c}\right)^{2/j} 
    \label{Lem.Energy.State1}
    \end{align}
and, for any $T>0$, $2\le p < \infty$,
    \begin{align}
    \displaystyle \left( \sup_{\varepsilon>0} \sum_{i=1}^2 \|\partial_{x} u_i^\varepsilon\|_{L^2(I_T)} \right) + \varepsilon^{-1/p}  \|u_1^\varepsilon - (u_2^\varepsilon)^2 \|_{L^{p}(I_T)}   \le \displaystyle  C_{T},  
    \label{Lem.Energy.State2}
    \end{align} 
where $C_T$ depends on $a,b,c,u_1^{\mathrm{in}},u_2^{\mathrm{in}}$ and $p$.
\end{lemma}

\begin{proof} Let us consider $p>0$. For convenience we write $d_1 = d+\delta$ and $d_2 = d$. Multiplying the equations for $u_i^\varepsilon$ by $(u_i^\varepsilon)^{ip}$, $i=1,2$, gives 
\begin{equation}
\begin{gathered}
\sum_{i=1}^2 \frac{1}{ip+1} \frac{d}{dt} \int_I (u_i^\varepsilon)^{ip+1} + \sum_{i=1}^2 id_i p   \int_{I } (u_i^\varepsilon)^{ip-1} |\partial_{x} u_i^\varepsilon|^2 \\
+ \dfrac{1}{\varepsilon}  \int_{I } ( u_1^\varepsilon -(u_2^\varepsilon)^{2})((u_1^\varepsilon)^{p}-(u_2^\varepsilon)^{2p}) 
     = \sum_{i=1}^2  \int_I (u_i^\varepsilon)^{ip} \varphi_i^\varepsilon  ,
\end{gathered}
\label{Lem.Energy.Proof1}
\end{equation} 
for any $t>0$, where
\begin{align*}
(u_1^\varepsilon)^{p} \varphi_1^\varepsilon = (u_1^\varepsilon)^{p} \phi^\varepsilon(u_1^\varepsilon,u_1^\varepsilon+u_2^\varepsilon) \le  a (u_1^\varepsilon)^{p+1}-(b+c)(u_1^\varepsilon)^{p+2},
\end{align*} 
and 
\begin{align*}
(u_2^\varepsilon)^{2p} \varphi_2^\varepsilon & = (u_2^\varepsilon)^{2p} \Big(\psi(u_1^\varepsilon,u_1^\varepsilon+u_2^\varepsilon)-\phi(u_1^\varepsilon,u_1^\varepsilon+u_2^\varepsilon)\Big) \\
& = (u_2^\varepsilon)^{2p}  \Big(au_2^\varepsilon + bu_1^\varepsilon - (bu_1^\varepsilon+cu_2^\varepsilon)(u_1^\varepsilon+u_2^\varepsilon) \Big) \\
& \le  (u_2^\varepsilon)^{2p} (au_2^\varepsilon - c(u_2^\varepsilon)^2 + bu_1^\varepsilon -  b(u_1^\varepsilon)^2 ) \\
& \le  - c(u_2^\varepsilon)^{2p+2} + a (u_2^\varepsilon)^{2p+1}  + b(u_2^\varepsilon)^{2p} .   
\end{align*}
Moreover, by the Young inequality, 
\begin{gather*}
   a (u_1^\varepsilon)^{p+1} - (b+c) (u_1^\varepsilon)^{p+2} \le m_{1p} - \frac{1}{p+1} (u_1^\varepsilon)^{p+1}, \\
   - c(u_2^\varepsilon)^{2p+2} + a (u_2^\varepsilon)^{2p+1}  + b(u_2^\varepsilon)^{2p} \le m_{2p} - \frac{1}{2p+1} (u_2^\varepsilon)^{2p+1},
\end{gather*}
where 
\begin{gather*}
    m_{1p}:= \frac{a+1/(p+1)}{p+2} \left( \frac{p+1}{p+2} \frac{a+1/(p+1)}{b+c} \right)^{p+1},
\end{gather*}
and
\begin{gather*}
   m_{2p}:= \frac{b}{p+1} \left( \frac{p}{p+1} \frac{2b}{c} \right)^{p} + \frac{a+1/(2p+1)}{p+1} \left( \frac{2p+1}{2p+2} \frac{2a+2/(2p+1)}{c} \right)^{2p+1}.
\end{gather*} 
Since $( u_1^\varepsilon -(u_2^\varepsilon)^{2})((u_1^\varepsilon)^{p}-(u_2^\varepsilon)^{2p})$ is non-negative, it follows from \eqref{Lem.Energy.Proof1} that 
\begin{equation*}
 \frac{d}{dt} \left( \sum_{i=1}^2 \frac{1}{ip+1}  \int_I (u_i^\varepsilon)^{ip+1} \right) + \left( \sum_{i=1}^2 \frac{1}{ip+1}  \int_I (u_i^\varepsilon)^{ip+1} \right) \le (m_{1p}+m_{2p})|I|,
\end{equation*}
which consequently shows, for $j=1,2,$
\begin{align*}
\|u_j^\varepsilon\|_{L^{p+1}(I)} & \le (jp+1)^{\frac{1}{jp+1}} \left( \sum_{i=1}^2 \frac{1}{ip+1} \|u_i^{\mathrm{in}}\|_{L^{ip+1}(I)}^{ip+1} + (m_{1p}+m_{2p})|I| \right)^{\frac{1}{jp+1}} \\
& \le (jp+1)^{\frac{1}{jp+1}} \left( \sum_{i=1}^2  \|u_i^{\mathrm{in}}\|_{L^{ip+1}(I)}^{\frac{ip+1}{jp+1}} + (m_{1p}^{\frac{1}{jp+1}}+m_{2p}^{\frac{1}{jp+1}})|I|^{\frac{1}{jp+1}} \right) .
\end{align*}
Letting $p\to \infty$, we obtain  
\begin{equation}
\begin{aligned}
\|u_j^\varepsilon\|_{L^{\infty}(I_\infty)} & \le
  \sum_{i=1}^2  \|u_i^{\mathrm{in}}\|_{L^{\infty}(I)}^{\frac{i}{j}} + \left( \frac{a}{b+c}\right)^{\frac{1}{j}} + \left( \frac{2b}{c}\right)^{\frac{1}{j}} + \left( \frac{2a}{c}\right)^{\frac{2}{j}} . 
\end{aligned}
\label{Lem.Energy.Proof2}
\end{equation}

Since \eqref{Lem.Energy.Proof1} holds for all $p>0$, taking $p=1$ and then $p=1/2$ respectively gives uniform estimates for the gradient $\partial_{x}u_i^\varepsilon$ in $L^2(I_T)$. Finally, taking into account the basic inequality $|x^p-y^p|\ge |x-y|^p$ for $p\ge 1$ and $x,y\ge 0$, we get from \eqref{Lem.Energy.Proof1} that    
\begin{align*} 
 \dfrac{1}{\varepsilon} \iint_{I_T} | u_1^\varepsilon -(u_2^\varepsilon)^{2}|^{p+1} \le\,&\, C_p(T) . 
\end{align*}
These arguments show \eqref{Lem.Energy.State2}. 
\end{proof} 

In the proof of Proposition \ref{Theo.FastReactionLimit}, we will use the definition below.

\begin{definition} \label{Definition.Weak}  A couple $(u,v)\in L^\infty(I_T)^2$ with $v\in L^2(0,T;H^1(I))$, for any $T>0$, is called a global weak solution to \eqref{Problem.Limiting}-\eqref{Condition.Initial.Limiting}  if the following conditions are satisfied
\begin{itemize}
\item $-u + (v-u)^2 = 0$ a.e. on $I_T$; and
\item for all $\chi \in C_c^\infty([0,T)\times \Omega)$, 
\begin{align*}
- \iint_{I_T}  v \partial_t \chi + d \iint_{I_T} \partial_{x} v \partial_x \chi = \int_I  v_{\mathrm{in}} \chi(0) + \iint_{I_T} \psi(u,v) \chi.
\end{align*}
\end{itemize}
\end{definition}

\begin{prop}[Fast reaction limit]
\label{Theo.FastReactionLimit} 
Let $T>0$  and $(u_1^\varepsilon,u_2^\varepsilon)$ be the classical solution to \eqref{Problem.Original}-\eqref{Condition.Initial} for each $\epsilon>0$.   Then, as $\varepsilon\to 0$,   
	\begin{equation}
		(u^\varepsilon, v^\varepsilon) \longrightarrow (u,v) \quad \text{ in } \quad L^p(I_T)^3,
		\label{Theo.FastReacLimit.State1}  
	\end{equation}
for any $1\le p<\infty$, where $(u,v)$ is the strong solution to \eqref{Problem.Limiting}-\eqref{Condition.Initial.Limiting}. Moreover, if we consider the initial data $(u_{\mathrm{in}},v_{\mathrm{in}})$ in the ball $B_M(0,0) \subset H^2(I)\times H^2(I)$ then 
\begin{align}
     0\le u(x,t) \le v(x,t) \le \|v_{\mathrm{in}}\|_{L^\infty(I)}+\frac{a}{c} ,
\end{align}
for all $(x,t)\in [0,T]\times I$.
\end{prop}

\begin{proof}
    For any $T>0$, Lemma \ref{Lem.Energy} yields that $u_1^\varepsilon$ is uniformly bounded in $L^2((0,T);H^1(I))$. Moreover, since 
$$\partial_t v^\varepsilon = \partial_t u_1^\varepsilon + \partial_t u_2^\varepsilon = (d+\delta) \partial_{xx}^2 u_1^\varepsilon + d \partial_{xx}^2 u_2^\varepsilon + \varphi_1(u_1^\eps,u_2^\eps) + \varphi_2(u_1^\eps, u_2^\eps),$$
$\partial_t v^\varepsilon$ is uniformly bounded in $L^2(0,T;(H^1(I))')$. Therefore $\{v^\varepsilon\}_{\varepsilon>0}$ is relative compact in $L^2(I_T)$ due to the Aubin–Lions lemma. Then, there is a subsequence (not relabeled) $\{v^\varepsilon\}$ such that $v^\varepsilon \to v\ge 0$ in $L^2(I_T)$.  With the function $v$ obtained above, and the couple $(u_1,u_2)$ defined by the non-negative solution of 
\begin{align}
(u_1)^{1/2} + u_1 = v, \quad \text{and} \quad u_2 + u_2^2 = v, 
\label{Theo.FastReacLimit.Proof1}
\end{align}
we will show \eqref{Theo.FastReacLimit.State1} for the $u$-components.  For this purpose, thanks to Lemma \ref{Lem.Energy}, we have $u_1^\varepsilon-(u_2^\varepsilon)^2 \to 0$ in $L^2(I_T)$.
Moreover, we can observe that
\begin{align*}
|u_2^\varepsilon - u_2| &= \frac{|(u_2^\varepsilon+(u_2^\varepsilon)^2)-(u_2+u_2^2)|}{1+u_2^\varepsilon+u_2} \\
&\le \frac{|(u_2^\varepsilon+(u_2^\varepsilon)^2)-v^\varepsilon|+|v^\varepsilon-v|}{1+u_2^\varepsilon+u_2} \\
&\le |(u_2^\varepsilon)^2-u_1^\varepsilon|+|v^\varepsilon-v|,
\end{align*}
which yields that $u_2^\varepsilon \to u_2$ in $L^2(I_T)$ up to extraction of a subsequence. On the other hand, we can see from the equalities in \eqref{Theo.FastReacLimit.Proof1} that $u_1=u_2^2$ (and then $u_1+u_2=v$), or equivalently 
\begin{align}
- u + (v-u)^2 = - u_1 + u_2^2 = 0. 
\label{Theo.FastReacLimit.Proof2}
\end{align}
Therefore, $v- u_2=u_2^2 = u_1$, and 
\begin{align*}
u_1^\varepsilon = v^\varepsilon - u_2^\varepsilon \to v - u_2 = u_1 \quad \text{in } L^2(I_T)
\end{align*} 
up to a subsequence. Because of its uniform boundedness (cf. Lemma \ref{Lem.Energy}), we can extract a subsequence of $(u^\varepsilon,u_1^\varepsilon,u_2^\varepsilon)$ such that the above convergences hold in  $L^p(I)$ for any $1\le p <\infty$.  

\medskip 

Next, we will show that $v$ is the weak solution to the limiting problem $(P_0)$. The weak formulation for the equation of $v^\varepsilon$ in \eqref{Problem.Original} is given by  
\begin{align*}
- \iint_{I_T}  v^\varepsilon \partial_t \chi + d \iint_{I_T} \partial_{x} v^\varepsilon \partial_x \chi = \int_I  v_{\mathrm{in}} \chi(0) + \iint_{I_T} \psi(u^\varepsilon,v^\varepsilon) \chi - \delta(\varepsilon) \iint_{I_T}  \partial_{x} u^\varepsilon \partial_{x} \chi, 
\end{align*}  
for all $\chi \in C_c^\infty([0,T)\times \Omega)$. Thanks to Lemma \ref{Lem.Energy}, Assumption \eqref{Assumption.Delta} on $\delta=\delta(\varepsilon)$, and the above arguments, we obtain 
\begin{align}
- \iint_{I_T}  v \partial_t \chi + d \iint_{I_T} \partial_{x} v \partial_x \chi = \int_I v_{\mathrm{in}}  \chi(0) + \iint_{I_T} \psi(u,v) \chi   
\label{Theo.FastReacLimit.Proof3}
\end{align}
after sending $\varepsilon \to 0$. Thus, by \eqref{Theo.FastReacLimit.Proof2}-\eqref{Theo.FastReacLimit.Proof3}, $(u,v)$ is a weak solution to \eqref{Problem.Limiting}-\eqref{Condition.Initial.Limiting} according to Definition \ref{Definition.Weak}, where we note that $(u,v)\in L^\infty(I_T)$ since $\psi(u,v)\le a v$ and $0\le u\le v$. A standard regularisation yields that $v$ is also a strong (and a classical) solution to \eqref{Problem.Limiting}-\eqref{Condition.Initial.Limiting}. Moreover,  the uniqueness of this solution is straightforward. Therefore, all the above convergences hold for the whole sequence. 

\medskip

Since $\psi(u,v) \le (a-cv)v$, by the same techniques as Lemma \ref{Lem.Energy} we have  
\begin{align*}
    \frac{1}{p} \frac{d}{dt} \int_I v^p \le \int_I (a-cv)v^p \le \frac{a+1/p}{p+1} \left( \frac{a+1/p}{c} \frac{p}{p+1} \right)^p |I|  - \frac{1}{p} \int_I v^p,
\end{align*}
and therefore,  $0\le u(x,t) \le v(x,t) \le \|v_{\mathrm{in}}\|_{L^\infty(I)}+a/c$ for all $(x,t)\in I_\infty.$  
\end{proof}

\subsection{Initial Layer}
\label{Sec:InitialLayer}

Useful computations on the initial layer are presented in this part, constituting the proof of Lemma \ref{Lemma.InitialLayer}. 

\begin{proof}[Proof of Lemma \ref{Lemma.InitialLayer}] We firstly note from \eqref{Theo.FastReacLimit.Proof1} and $u(0)\ge 0$ that 
\begin{align*}
\sqrt{u(0)} + u(0)=v_{\mathrm{in}}, \quad \text{or} \quad  u(0)=\frac{2 v_{\mathrm{in}} + 1 - \sqrt{4v_{\mathrm{in}}+1}}{2}.    
\end{align*}
Therefore, direct computations show
\begin{align}
    u_{\mathrm{in}}-u(0)= u_{\mathrm{in}} - \frac{2 v_{\mathrm{in}} + 1 - \sqrt{4v_{\mathrm{in}}+1}}{2} = \frac{u_{\mathrm{in}} - (v_{\mathrm{in}}-u_{\mathrm{in}})^2}{\Phi(u_{\mathrm{in}},v_{\mathrm{in}})}, 
\label{Lemma.InitialLayer.Proof1}
\end{align}
where we denote 
\begin{align*}
    \Phi(u_{\mathrm{in}},v_{\mathrm{in}}) := \frac{2(v_{\mathrm{in}}-u_{\mathrm{in}})+\sqrt{4v_{\mathrm{in}}+1}+1}{2}. 
\end{align*}
Thanks to the assumption $v_{\mathrm{in}} \ge u_{\mathrm{in}}$, we have $\Phi(u_{\mathrm{in}},v_{\mathrm{in}}) \ge 1$. This consequently gives  
    \begin{align*}
    \int_I |u_{\mathrm{in}}-u(0)|^2  &\le \int_I |u_{\mathrm{in}} - (v_{\mathrm{in}}-u_{\mathrm{in}})^2|^2 \le \varepsilon_{\mathrm{in}}^2 .  
\end{align*}
 
Next, let us estimate the derivatives $\partial_x(1/\Phi(u_{\mathrm{in}},v_{\mathrm{in}})$ and $\partial_{xx}^2 (1/ \Phi(u_{\mathrm{in}},v_{\mathrm{in}})) $. With straightforward computations,  one can check that 
\begin{align}
    \begin{gathered}
    \partial_x \Phi(u_{\mathrm{in}},v_{\mathrm{in}}) = \partial_x(v_{\mathrm{in}}-u_{\mathrm{in}}) + \frac{\partial_x v_{\mathrm{in}}}{ \sqrt{4v_{\mathrm{in}}+1}}, \\
    \partial_{xx}^2 \Phi(u_{\mathrm{in}},v_{\mathrm{in}}) = \partial_{xx}^2(v_{\mathrm{in}}-u_{\mathrm{in}}) + \frac{(4v_{\mathrm{in}}+1)\partial_{xx}^2 v_{\mathrm{in}} - 2(\partial_x v_{\mathrm{in}})^2 }{ (4v_{\mathrm{in}}+1)\sqrt{4v_{\mathrm{in}}+1}}.
\end{gathered}
\label{Expression.DerivativeOfPhi}
\end{align}
Therefore, by using $\Phi(u_{\mathrm{in}},v_{\mathrm{in}}) \ge 1/2$ again, we get from 
\begin{gather*}
    \partial_x \left( \frac{1}{\Phi(u_{\mathrm{in}},v_{\mathrm{in}})} \right) = - \frac{\partial_x \Phi(u_{\mathrm{in}},v_{\mathrm{in}})}{[\Phi(u_{\mathrm{in}},v_{\mathrm{in}})]^2}, \\  
    \partial_{xx}^2 \left( \frac{1}{\Phi(u_{\mathrm{in}},v_{\mathrm{in}})} \right) =  2 \frac{(\partial_x \Phi(u_{\mathrm{in}},v_{\mathrm{in}}))^2}{[\Phi(u_{\mathrm{in}},v_{\mathrm{in}})]^3} - \frac{\partial_{xx}^2 \Phi(u_{\mathrm{in}},v_{\mathrm{in}})}{[\Phi(u_{\mathrm{in}},v_{\mathrm{in}})]^2} , 
\end{gather*}
and the computations in \eqref{Expression.DerivativeOfPhi}show that  
\begin{align*}
    \left|\partial_x \left( \frac{1}{\Phi(u_{\mathrm{in}},v_{\mathrm{in}})} \right)\right|  & \le  \left|\partial_x \Phi(u_{\mathrm{in}},v_{\mathrm{in}}) \right| \le 2 (|\partial_x u_{\mathrm{in}}| + |\partial_x v_{\mathrm{in}}|), \\
    \left| \partial_{xx}^2 \left( \frac{1}{\Phi(u_{\mathrm{in}},v_{\mathrm{in}})} \right)\right|  & \le 2 |\partial_x \Phi(u_{\mathrm{in}},v_{\mathrm{in}})|^2  +  \left| \partial_{xx}^2 \Phi(u_{\mathrm{in}},v_{\mathrm{in}}) \right| \\
     & \le C ( |\partial_x u_{\mathrm{in}}| + |\partial_x v_{\mathrm{in}}|  + |\partial_x v_{\mathrm{in}}|^2     )  + C (|\partial_{xx}^2 u_{\mathrm{in}}| +  |\partial_{xx}^2 v_{\mathrm{in}}|).
\end{align*}
Now, differentiating the expression \eqref{Lemma.InitialLayer.Proof1} in $x$ gives 
\begin{align}
    \begin{aligned}
    |\partial_{x}(u_{\mathrm{in}}-u(0))| & \le  \frac{\left| \partial_x ( u_{\mathrm{in}} - (v_{\mathrm{in}}-u_{\mathrm{in}})^2 ) \right| }{\Phi(u_{\mathrm{in}},v_{\mathrm{in}})} +  \left| u_{\mathrm{in}} - (v_{\mathrm{in}}-u_{\mathrm{in}})^2 \right| \left|\partial_x \left( \frac{1}{\Phi(u_{\mathrm{in}},v_{\mathrm{in}})} \right)\right|  \\
    & \le C\Big( \left| \partial_x ( u_{\mathrm{in}} - (v_{\mathrm{in}}-u_{\mathrm{in}})^2 ) \right| + \left| u_{\mathrm{in}} - (v_{\mathrm{in}}-u_{\mathrm{in}})^2 \right|  \left(|\partial_x u_{\mathrm{in}}| + |\partial_x v_{\mathrm{in}}|\right) \Big),
    \end{aligned}
    \label{Lemma.InitialLayer.Proof2}
\end{align}
and 
\begin{align}
    \begin{aligned}
    |\partial_{xx}^2 (u_{\mathrm{in}}-u(0))| & \le  \frac{\left| \partial_{xx}^2 ( u_{\mathrm{in}} - (v_{\mathrm{in}}-u_{\mathrm{in}})^2 ) \right| }{\Phi(u_{\mathrm{in}},v_{\mathrm{in}})} +  \left| u_{\mathrm{in}} - (v_{\mathrm{in}}-u_{\mathrm{in}})^2 \right|  \left| \partial_{xx}^2 \left( \frac{1}{\Phi(u_{\mathrm{in}},v_{\mathrm{in}})} \right)\right| \\
    & +  \left| \partial_x( u_{\mathrm{in}} - (v_{\mathrm{in}}-u_{\mathrm{in}})^2 ) \right| \left|\partial_x \left( \frac{1}{\Phi(u_{\mathrm{in}},v_{\mathrm{in}})} \right)\right|  \\
    & \le C\Big( \left| \partial_{xx}^2 ( u_{\mathrm{in}} - (v_{\mathrm{in}}-u_{\mathrm{in}})^2 ) \right| + \left| \partial_x ( u_{\mathrm{in}} - (v_{\mathrm{in}}-u_{\mathrm{in}})^2 ) \right| \left( |\partial_x u_{\mathrm{in}}| + |\partial_x v_{\mathrm{in}}| \right) \Big) \\
    & + C \left| u_{\mathrm{in}} - (v_{\mathrm{in}}-u_{\mathrm{in}})^2 \right|  \Big(   |\partial_x u_{\mathrm{in}}| + |\partial_x v_{\mathrm{in}}|  + |\partial_x v_{\mathrm{in}}|^2  \Big) \\
    & + C \left| u_{\mathrm{in}} - (v_{\mathrm{in}}-u_{\mathrm{in}})^2 \right|  \Big(   |\partial_{xx}^2 u_{\mathrm{in}}| + (2v_{\mathrm{in}}+1) |\partial_{xx}^2 v_{\mathrm{in}}| \Big).
    \end{aligned}
    \label{Lemma.InitialLayer.Proof3}
\end{align}

By the embedding $H^1(I) \hookrightarrow L^\infty(I)$, we have $(u_{\mathrm{in}},v_{\mathrm{in}}) \in L^\infty(I)^2$. Then, we can get from the estimate \eqref{Lemma.InitialLayer.Proof2} that 
\begin{align*}
    \int_I |\partial_{x}(u_{\mathrm{in}}-u(0))|^2 & \le C  \int_I \left| \partial_x ( u_{\mathrm{in}} - (v_{\mathrm{in}}-u_{\mathrm{in}})^2 ) \right|^2 \\
    & + C \| u_{\mathrm{in}} - (v_{\mathrm{in}}-u_{\mathrm{in}})^2 \|_{L^\infty(I)}^2 \left( \|\partial_x u_{\mathrm{in}}\|_{L^2(I)} + \|\partial_x v_{\mathrm{in}}\|_{L^2(I)} \right) \\
    & \le C \| u_{\mathrm{in}} - (v_{\mathrm{in}}-u_{\mathrm{in}})^2 \|_{H^1(I)}^2 \\
    & + C \| u_{\mathrm{in}} - (v_{\mathrm{in}}-u_{\mathrm{in}})^2 \|_{H^1(I)}^2 \left( \|u_{\mathrm{in}}\|_{H^1(I)}^2+\|v_{\mathrm{in}}\|_{H^1(I)}^2 \right) \\
    & \le C \left( 1 + \|(u_{\mathrm{in}},v_{\mathrm{in}})\|_{H^1(I)^2}^2 \right) \varepsilon_{\mathrm{in}}^2  . 
\end{align*}

Due to the embedding $H^2(I) \hookrightarrow W^{1,\infty}(I)$, we have $(\partial_x u_{\mathrm{in}},\partial_x v_{\mathrm{in}}) \in L^\infty(I)^2$. Therefore, it follows from the estimate \eqref{Lemma.InitialLayer.Proof3} that  
\begin{align*}
    \int_I |\partial_{xx}^2 (u_{\mathrm{in}}-u(0))|^2   
    & \le C  \| \partial_{xx}^2 ( u_{\mathrm{in}} - (v_{\mathrm{in}}-u_{\mathrm{in}})^2 ) \|_{L^2(I)}^2 \\
    & + C \left( \|\partial_x u_{\mathrm{in}}\|_{L^\infty(I)}^2 + \|\partial_x v_{\mathrm{in}}\|_{L^\infty(I)}^2 \right) \| \partial_x ( u_{\mathrm{in}} - (v_{\mathrm{in}}-u_{\mathrm{in}})^2 ) \|_{L^2(I)}^2  \\
    & + C \| u_{\mathrm{in}} - (v_{\mathrm{in}}-u_{\mathrm{in}})^2 \|_{L^\infty(I)}^2  \left( \|\partial_x u_{\mathrm{in}}\|_{L^2(I)}^2 + \|\partial_x v_{\mathrm{in}}\|_{L^2(I)}^2  + \|\partial_x v_{\mathrm{in}}\|_{L^4(I)}^4  \right) \\
    & + C \| u_{\mathrm{in}} - (v_{\mathrm{in}}-u_{\mathrm{in}})^2 \|_{L^\infty(I)}^2  \Big( \|\partial_{xx}^2 u_{\mathrm{in}}\|_{L^2(I)}^2 + \|2v_{\mathrm{in}}+1\|_{L^\infty(I)}^2 \|\partial_{xx}^2 v_{\mathrm{in}}\|_{L^2(I)}^2 \Big) \\
    & \le C \left( 1 + \|(u_{\mathrm{in}},v_{\mathrm{in}})\|_{H^2(I)}^4 \right) \varepsilon_{\mathrm{in}}^2 . 
\end{align*}
The estimate \eqref{Lemma.InitialLayer.State1} is completely proved.
\end{proof}

\subsection{Convergence in $L^{\infty}(0,T;L^2(I))$}\label{Sec:LinfL2}

For the computations below, we recall from Proposition \ref{Theo.FastReactionLimit} that   
\begin{align}
\|(u,v)\|_{L^\infty(0,T;H^2(I))^2} +  \|(u,v)\|_{W^{2,1}_2(I_T)^2} + \sup_{\varepsilon>0} \|(u^\varepsilon,v^\varepsilon)\|_{L^\infty(I_T)^2} 
 \le C_T. 
\label{SolutionSmoothness}
\end{align}

\begin{lemma}
\label{Lem.VEnergy} For $t\in(0,T)$, it holds 
\begin{align}
 \frac{1}{2} \frac{d}{dt} \int_I (V^\varepsilon)^2  
& \le - \frac{d}{2} \int_I |\partial_{x} V^\varepsilon|^2 + C_T \left( \int_I (U^\varepsilon)^2 + \int_I (V^\varepsilon)^2 \right) + C_T\delta^2 \int_I |\partial_x u^\varepsilon|^2.
\label{Lem.VEnergy.State1}
\end{align}
\end{lemma}

\begin{proof}
 Multiplying the equation for $V^\varepsilon$ by itself gives
\begin{align*}
 \frac{1}{2} \frac{d}{dt} \int_I (V^\varepsilon)^2    
& =  -d \int_I |\partial_{x} V^\varepsilon|^2 + \int_I \left( \psi(u+U^\varepsilon,v+V^\varepsilon)-\psi(u,v+V^\varepsilon) \right) V^\varepsilon \\
& + \int_I \left( \psi(u,v+V^\varepsilon)-\psi(u,v) \right) V^\varepsilon +  \delta \int_I  \partial_{xx}^2 u^\varepsilon V^\varepsilon  \\
& \le  - \frac{d}{2} \int_I |\partial_{x} V^\varepsilon|^2 + \int_0^1 \hspace*{-0.15cm} \int_I \psi_1'(u+hU^\varepsilon,v+V^\varepsilon) U^\varepsilon V^\varepsilon \\
& + \int_0^1 \hspace*{-0.15cm} \int_I \psi_2'(u,v+hV^\varepsilon)  (V^\varepsilon)^2 + \frac{\delta^2}{2d} \int_I (\partial_x u^\varepsilon)^2,
\end{align*}
where $\psi_1'$, $\psi_2'$ denote the derivatives of first order of $\psi$ with respect to the first and the second arguments. Thanks to \eqref{SolutionSmoothness},    
\begin{align}
\sup_{\varepsilon>0} |\psi_1'(u+hU^\varepsilon,v+V^\varepsilon)| = \sup_{\varepsilon>0} | -b (v+V^\varepsilon) | \le C, 
\label{Lem.VEnergy.Proof1}
\end{align}
and, for all $h\in[0,1]$, 
\begin{align}
\sup_{\varepsilon>0} |\psi_2'(u,v+hV^\varepsilon)| = \sup_{\varepsilon>0} |a-bu-2c(v+hV^\varepsilon)| \le C_T.
\label{Lem.VEnergy.Proof2}
\end{align}
The estimate \eqref{Lem.VEnergy.State1} can be derived directly by taking all the above estimates. 
\end{proof}

\begin{lemma} 
\label{Lem.UEnergy} For $t\in(0,T)$, it holds
\begin{equation}
\label{Lem.UEnergy.State1}
\begin{aligned}
\frac{1}{2} \frac{d}{dt} \int_I (U^\varepsilon)^2   \le - d \int_I |\partial_{x} U^\varepsilon|^2 - \frac{1}{2\varepsilon} \int_I (U^\varepsilon)^2 +  \frac{C_T}{\varepsilon} \int_I (V^\varepsilon)^2 + C_T \varepsilon   .
\end{aligned}
\end{equation}

\end{lemma}

\begin{proof} Multiplying the equation for $U^\varepsilon$ by itself gives
\begin{equation}
\begin{aligned}
 \frac{1}{2} \frac{d}{dt} \int_I (U^\varepsilon)^2 &= - (d+\delta) \int_I |\partial_{x} U^\varepsilon|^2 + \int_I ( \phi(u+U^\varepsilon,v+V^\varepsilon)-\phi(u,v)) U^\varepsilon  \\
& + \int_I \xi^\varepsilon U^\varepsilon  +  \dfrac{1}{\varepsilon} \int_I ( g(u+U^\varepsilon,v+V^\varepsilon)-g(u,v)) U^\varepsilon .
\end{aligned}
\label{Lem.UEnergy.Proof1}
\end{equation}
  Therefore, with notations $\phi_1'$, $\phi_2'$ for the derivatives of first order of $\phi$ according to the first and the second arguments, it is similar to  \eqref{Lem.VEnergy.Proof1}-\eqref{Lem.VEnergy.Proof2} to see that 
\begin{equation}
\begin{aligned}
& \int_I ( \phi(u+U^\varepsilon,v+V^\varepsilon)-\phi(u,v)) U^\varepsilon \\
& \le \int_0^1 \hspace*{-0.15cm} \int_I \Big( |\phi_1'(u+hU^\varepsilon,v+V^\varepsilon)U^\varepsilon| +  |\phi_2'(u,v+hV^\varepsilon)V^\varepsilon| \Big)|U^\varepsilon| \\
& \le C_T \int_I (|U^\varepsilon| +  |V^\varepsilon| )|U^\varepsilon| \le  \frac{1}{8\varepsilon} \int_I (U^\varepsilon)^2 + C_T \varepsilon \int_I (V^\varepsilon)^2 , 
\end{aligned}
\label{Lem.UEnergy.Proof2}
\end{equation}
where, without loss of generality, we can consider $\varepsilon$ small enough. 
On the other hand, due to Assumption \eqref{Assumption.Delta} and the regularity \eqref{SolutionSmoothness},  
\begin{align}
\int_I \xi^\varepsilon U^\varepsilon \le \frac{1}{8\varepsilon}  \int_I  (U^\varepsilon)^2 + 2 \varepsilon \left( \sup_{\varepsilon>0} \int_I (\xi^\varepsilon)^2 \right) \le \frac{1}{8\varepsilon}  \int_I  (U^\varepsilon)^2 + C_T  \varepsilon .
\label{Lem.UEnergy.Proof3}
\end{align}
For the last term in \eqref{Lem.UEnergy.Proof1}, thanks to \eqref{DerivativeOfg} and \eqref{SolutionSmoothness}, one has
\begin{equation}
\begin{aligned}
& \dfrac{1}{\varepsilon} \int_I ( g(u+U^\varepsilon,v+V^\varepsilon)-g(u,v)) U^\varepsilon \\
& = \dfrac{1}{\varepsilon} \int_0^1 \hspace*{-0.15cm} \int_I g_1'(u+hU^\varepsilon,v+V^\varepsilon) (U^\varepsilon)^2 + \dfrac{1}{\varepsilon} \int_0^1 \hspace*{-0.15cm} \int_I g_2'(u,v+hV^\varepsilon) U^\varepsilon V^\varepsilon \\
& \le - \dfrac{1}{\varepsilon} \int_I  (U^\varepsilon)^2 + \dfrac{2}{\varepsilon} \int_0^1 \hspace*{-0.15cm} \int_I |(v+hV^\varepsilon-u)U^\varepsilon V^\varepsilon | \\
& \le - \dfrac{1}{\varepsilon} \int_I  (U^\varepsilon)^2 + \dfrac{C_T}{\varepsilon} \int_I |U^\varepsilon V^\varepsilon | \le - \dfrac{3}{4\varepsilon} \int_I  (U^\varepsilon)^2 + \dfrac{C_T}{\varepsilon} \int_I (V^\varepsilon)^2 .   
\end{aligned}
\label{Lem.UEnergy.Proof4}
\end{equation}
We obtain the estimate \eqref{Lem.UEnergy.State1} by combining \eqref{Lem.UEnergy.Proof1}-\eqref{Lem.UEnergy.Proof4}. 
\end{proof}

\begin{lemma}
\label{Lem.NablaUEnergy} For $t\in(0,T)$, it holds
\begin{equation}
\begin{aligned}
\frac{1}{2} \frac{d}{dt} \int_I | \partial_{x} U^\varepsilon|^2 &\le - \frac{d}{2} \int_I |\partial_{xx}^2 U^\varepsilon|^2 + \frac{C_T}{\varepsilon} \left( \int_I (U^\varepsilon)^2 + \int_I (V^\varepsilon)^2 \right) \\
&-\dfrac{1}{2\varepsilon} \int_I |\partial_{x} U^\varepsilon|^2 + \dfrac{C_T}{\varepsilon} \int_I | \partial_{x} V^\varepsilon|^2 + C_T \varepsilon .
\end{aligned}
\label{Lem.NablaUEnergy.State1}
\end{equation}
\end{lemma}

\begin{proof} Multiplying the equation for $U^\varepsilon$ by $- \partial_{xx}^2 U^\varepsilon$ itself gives
\begin{equation}
\begin{aligned}
\frac{1}{2} \frac{d}{dt} \int_I | \partial_{x} U^\varepsilon|^2  
& = - (d+\delta) \int_I |\partial_{xx}^2 U^\varepsilon|^2 - \int_I ( \phi(u+U^\varepsilon,v+V^\varepsilon)-\phi(u,v)) \partial_{xx}^2 U^\varepsilon  \\
& - \int_I \xi^\varepsilon \partial_{xx}^2  U^\varepsilon  - \dfrac{1}{\varepsilon} \int_I ( g(u+U^\varepsilon,v+V^\varepsilon)-g(u,v)) \partial_{xx}^2 U^\varepsilon ,
\end{aligned}
\label{Lem.NablaUEnergy.Proof1}
\end{equation}
Firstly, by a similar argument as \eqref{Lem.UEnergy.Proof2}, 
\begin{equation*}
\begin{aligned}
& \int_I | \phi(u+U^\varepsilon,v+V^\varepsilon)-\phi(u,v)|^2   \\
& \le \int_0^1 \hspace*{-0.15cm} \int_I \Big( 2|\phi_1'(u+hU^\varepsilon,v+V^\varepsilon)U^\varepsilon|^2 + 2|\phi_2'(u,v+hV^\varepsilon)V^\varepsilon|^2 \Big) \\
& \le C_T \int_I (|U^\varepsilon|^2 +  |V^\varepsilon|^2 ) ,
\end{aligned}
\end{equation*}
while, with the uniform boundedness of $\partial_x \xi^\varepsilon$, 
\begin{align*}
- \int_I \xi^\varepsilon \partial_{xx}^2 U^\varepsilon  \le \frac{1}{8\varepsilon} \int_I |\partial_{x} U^\varepsilon|^2 + C_T \varepsilon \int_I (\partial_x \xi^\varepsilon)^2  \le \frac{1}{8\varepsilon} \int_I |\partial_{x} U^\varepsilon|^2 + C_T \varepsilon .
\end{align*}
Let us estimate the first term on the right hand side of \eqref{Lem.NablaUEnergy.Proof1}. Due to integration by parts 
\begin{align*}
 \int_I ( g(u+U^\varepsilon,v+V^\varepsilon)-g(u,v)) \partial_{xx}^2 U^\varepsilon  = - \int_I \partial_{x} ( g(u+U^\varepsilon,v+V^\varepsilon)-g(u,v)) \partial_{x} U^\varepsilon,
\end{align*}
and direct computations, we have  
\begin{equation}
\begin{aligned}
& - \dfrac{1}{\varepsilon} \int_I ( g(u+U^\varepsilon,v+V^\varepsilon)-g(u,v)) \partial_{xx}^2 U^\varepsilon \\
& =  \dfrac{1}{\varepsilon} \int_I  g_1'(u+U^\varepsilon,v+V^\varepsilon) |\partial_{x} U^\varepsilon|^2 + \dfrac{1}{\varepsilon} \int_I g_2'(u+U^\varepsilon,v+V^\varepsilon) \partial_{x} U^\varepsilon \partial_{x} V^\varepsilon \\
&   + \dfrac{1}{\varepsilon} \int_I ( g_1'(u+U^\varepsilon,v+V^\varepsilon)-g_1'(u,v)) \partial_{x} u  \partial_{x} U^\varepsilon \\
&  + \dfrac{1}{\varepsilon} \int_I ( g_2'(u+U^\varepsilon,v+V^\varepsilon)-g_2'(u,v)) \partial_{x} v  \partial_{x} U^\varepsilon   =: K_1+K_2+K_3+K_4.
\end{aligned}
\label{Lem.NablaUEnergy.Proof2}
\end{equation}
Taking into account \eqref{DerivativeOfg}, 
\begin{align*}
K_1+K_2 &\le -\dfrac{1}{\varepsilon} \int_I |\partial_{x} U^\varepsilon|^2 + \dfrac{2}{\varepsilon} \int_I |((v+V^\varepsilon)-(u+U^\varepsilon)) \partial_{x} U^\varepsilon \partial_{x} V^\varepsilon| \\
& \le -\dfrac{1}{\varepsilon} \int_I |\partial_{x} U^\varepsilon|^2 + \dfrac{C_T}{\varepsilon} \int_I |\partial_{x} U^\varepsilon \partial_{x} V^\varepsilon| \\
&\le -\dfrac{3}{4\varepsilon} \int_I |\partial_{x} U^\varepsilon|^2 + \dfrac{C_T}{\varepsilon} \int_I | \partial_{x} V^\varepsilon|^2,
\end{align*}
which the second estimate can be obtained by using the regularity \eqref{SolutionSmoothness}. Moreover, since 
\begin{align*}
g_1'(u+U^\varepsilon,v+V^\varepsilon)-g_1'(u,v) = \int_0^1 \Big( g_{11}'' (u+hU^\varepsilon,v+V^\varepsilon)U^\varepsilon + g_{12}'' (u,v+hV^\varepsilon,v)V^\varepsilon \Big),\\
g_2'(u+U^\varepsilon,v+V^\varepsilon)-g_2'(u,v) = \int_0^1 \Big( g_{21}'' (u+hU^\varepsilon,v+V^\varepsilon)U^\varepsilon + g_{22}'' (u,v+hV^\varepsilon,v)V^\varepsilon \Big), 
\end{align*}
we infer from \eqref{DerivativeOfg} that   
\begin{align}
\left\{
\begin{aligned}
g_1'(u+U^\varepsilon,v+V^\varepsilon)-g_1'(u,v) = 2U^\varepsilon -2V^\varepsilon ,\\
g_2'(u+U^\varepsilon,v+V^\varepsilon)-g_2'(u,v) = 2V^\varepsilon -2U^\varepsilon .
\end{aligned} \right.
\label{Lem.NablaUEnergy.Proof3}
\end{align}
Hence, the last two terms can be estimated as follows
\begin{align*}
K_3+K_4 & = \dfrac{1}{\varepsilon} \int_I ( 2U^\varepsilon -2V^\varepsilon) (\partial_{x} u - \partial_{x} v)  \partial_{x} U^\varepsilon  \\
& \le \dfrac{2}{\varepsilon} \|\partial_{x} u - \partial_{x} v\|_{L^\infty(I_T)} \int_I |(U^\varepsilon - V^\varepsilon) \partial_{x} U^\varepsilon| \\
& \le \dfrac{1}{8\varepsilon} \int_I |\partial_{x} U^\varepsilon|^2 + \frac{C_T}{\varepsilon} \int_I \left(  (U^\varepsilon)^2 + (V^\varepsilon)^2 \right).
\end{align*}
We obtain \eqref{Lem.NablaUEnergy.State1} by combining the above estimates.
\end{proof}
 
We are now ready to prove the estimate \eqref{Lemma.ConvRateULinfL2.State1}, Proposition \ref{Lemma.ConvRateULinfL2}.

\begin{proof}[Proof of the estimate \eqref{Lemma.ConvRateULinfL2.State1}, Proposition \ref{Lemma.ConvRateULinfL2}] We firstly gather all estimates from Lemmas \ref{Lem.VEnergy}-\ref{Lem.UEnergy} as follows
\begin{align*}
\frac{1}{2} \frac{d}{dt} \int_I (V^\varepsilon)^2 \hspace*{0.35cm} 
\le \,& - \frac{d}{2} \int_I |\partial_{x} V^\varepsilon|^2 + C_{1,T} \int_I (U^\varepsilon)^2 + C_{1,T}  \int_I (V^\varepsilon)^2 + C_{1,T}\delta^2 \int_I |\partial_x u|^2   ,\\
\frac{\varepsilon}{2} \frac{d}{dt} \int_I (U^\varepsilon)^2  \hspace*{0.35cm} \le \,& - d \varepsilon \int_I |\partial_{x} U^\varepsilon|^2 - \frac{1}{2} \int_I (U^\varepsilon)^2 + C_{2,T} \int_I (V^\varepsilon)^2 + C_{2,T} \varepsilon^2   . 
\end{align*}
Then, we can take a linear combination of them to have 
\begin{align*}
& \frac{d}{dt} \left[ \frac{1}{2} \int_I (V^\varepsilon)^2 + C_{1,T} \, \varepsilon  \int_I (U^\varepsilon)^2 \right]  +  \frac{d}{2} \int_I |\partial_{x} V^\varepsilon|^2 + 2dC_{1,T} \, \varepsilon \int_I |\partial_{x} U^\varepsilon|^2 \\
& \le 2C_{1,T}C_{2,T} \varepsilon^2 + C_{1,T} \delta^2  \int_I |\partial_x u|^2 +  \left( C_{1,T} + 2C_{1,T}C_{2,T} \right) \int_I (V^\varepsilon)^2 . 
\end{align*}
By integrating over time and noting from Lemma \ref{Lem.Energy} that $\partial_x u^\varepsilon$ is uniformly bounded in $L^2(I_T)$,
\begin{align*}
 \int_I (V^\varepsilon(t))^2 +  \iint_{I_t} |\partial_{x} V^\varepsilon|^2  &\le C_T \left( (\delta^2 +\varepsilon^2) +  \int_I (V^\varepsilon(0))^2  + \varepsilon \int_I (U^\varepsilon(0))^2  +    \int_0^t \hspace{-0.15cm} \int_{I} (V^\varepsilon)^2 \right), 
\end{align*}
where $V^\varepsilon(0)=v^\varepsilon(0)-v_{\mathrm{in}}=0$. 
Therefore, by using the Gr\"onwall inequality, we obtain 
\begin{equation}
\begin{aligned}
\int_I (v^\varepsilon(t)-v(t))^2 + \iint_{I_t} |\partial_x (v^\varepsilon(t)-v(t))|^2  \le C_T \left( (\delta^2 +\varepsilon^2) + \varepsilon \int_I (u_{\mathrm{in}}-u(0))^2  \right).
\end{aligned} 
\label{Lemma.ConvRateULinfL2.Proof1}
\end{equation}
 Thanks to Lemma \ref{Lemma.InitialLayer},     
\begin{align}
    \esssup_{t\in[0,T]} \int_I (v^\varepsilon(t)-v(t))^2 \le  C_{T} (  \varepsilon^2 +  \varepsilon_{\mathrm{in}}^2 \varepsilon +  \delta^2 ) . \label{Lemma.ConvRateULinfL2.Proof2}
\end{align}
Next, we will show  that 
\begin{align}
    \esssup_{t\in[0,T]} \int_I (u^\varepsilon(t)-u(t))^2  \le C_{T} (  \varepsilon^2 +  \varepsilon_{\mathrm{in}}^2  +  \delta^2 ) .
    \label{Lemma.ConvRateULinfL2.Proof3}
\end{align}    
By using Lemma \ref{Lem.UEnergy} and the estimate \eqref{Lemma.ConvRateULinfL2.Proof2}, 
\begin{equation} 
\begin{aligned}
 \varepsilon \frac{d}{dt} \int_I (u^\varepsilon(t)-u(t))^2 + \int_I (u^\varepsilon(t)-u(t))^2   \le C_T \left( \varepsilon^2 + \varepsilon_{\mathrm{in}}^2 \varepsilon +  \delta^2 \right).
\end{aligned}
\label{Lemma.ConvRateULinfL2.Proof4}
\end{equation}
Using the comparison principle, we get 
\begin{align}
\begin{aligned}
    \int_I (u^\varepsilon(t)-u(t))^2 &  \le e^{-t/\varepsilon} \int_I (u_{\mathrm{in}}-u(0))^2 +  C_T \left( \varepsilon^2 + \varepsilon_{\mathrm{in}}^2 \varepsilon +  \delta^2 \right) (1-e^{-t/\varepsilon}) \\
    & \le C e^{-t/\varepsilon} \varepsilon_{\mathrm{in}}^2  +  C_T \left( \varepsilon^2 + \varepsilon_{\mathrm{in}}^2 \varepsilon +  \delta^2 \right) (1-e^{-t/\varepsilon}) ,
\end{aligned}
\label{Lemma.ConvRateULinfL2.Proof5}
\end{align}
which gives \eqref{Lemma.ConvRateULinfL2.Proof3}.  

\medskip

From the above  proof, we can estimate $|(u^\varepsilon,v^\varepsilon)-(u,v)|$ in $L^2(0,T;H^1(I))$. Indeed, due to the estimates \eqref{Lemma.ConvRateULinfL2.Proof2}-\eqref{Lemma.ConvRateULinfL2.Proof3}, we can infer from Lemma \ref{Lem.VEnergy} that
\begin{align*}
\iint_{I_T} |\partial_{x} (v^\varepsilon(t)-v(t))|^2
& \le   C_T \left( \iint_{I_T} (u^\varepsilon(t)-u(t))^2 + \iint_{I_T} (v^\varepsilon(t)-v(t))^2 \right) + C_T\delta^2 , 
\end{align*}
and so,
\begin{align*}
\iint_{I_T} |\partial_{x} (v^\varepsilon(t)-v(t))|^2
  \le    C_T \left( \varepsilon^2 + \varepsilon_{\mathrm{in}}^2  +  \delta^2 \right).
\end{align*}
Now, by using Lemma \ref{Lem.NablaUEnergy}, 
\begin{align*}
\iint_{I_T} |\partial_{x} (u^\varepsilon(t)-u(t))|^2
& \le  C_T \left( \varepsilon \int_I |\partial_x(u^\varepsilon(0)-u(0))|^2 + \iint_{I_T} |\partial_{x} (v^\varepsilon(t)-v(t))|^2 + \varepsilon^2  \right)  \\
& + C_T \left( \iint_{I_T} (u^\varepsilon(t)-u(t))^2 + \iint_{I_T} (v^\varepsilon(t)-v(t))^2 \right)   .
\end{align*}
Thanks to Lemma \ref{Lemma.InitialLayer}, 
\begin{align*}
\iint_{I_T} |\partial_{x} (u^\varepsilon(t)-u(t))|^2
& \le  C_T \left( \varepsilon^2 + \varepsilon_{\mathrm{in}}^2  +  \delta^2 \right).
\end{align*}
Hence, we obtain the estimate
\begin{align}
\|(u^\varepsilon,v^\varepsilon)-(u,v)\|_{L^2(0,T;H^1(\Omega))^2} 
\le C_{T} (   \varepsilon + \varepsilon_{\mathrm{in}}  +  \delta  ). 
\label{Remark.L2H1}
\end{align}
\end{proof}

\subsection{Convergence in $L^\infty(0,T;H^2(I))$}
\label{Sec:LinfH2}

In order to obtain a uniform-in-time convergence rate in $H^2(I)$, we will first establish estimates for the rate in the intersection of $L^\infty(0,T;H^1(I))$ and $ L^2(0,T;H^2(I))$,  presented in Lemma \ref{Lemma.ConvRateULinfH1}, where we note that the assumption \eqref{Assumption.Compare.Eps.Del} is necessary.

\begin{lemma} 
\label{Lemma.ConvRateULinfH1}
Assume that the initial data $(u_{\mathrm{in}},v_{\mathrm{in}})$ is defined by Lemma \ref{Lemma.InitialLayer}. Then
\begin{align*}
    \|(u^\varepsilon,v^\varepsilon) -(u,v)\|_{L^\infty(0,T;H^1(I))^2} + \|v^\varepsilon -v\|_{L^2(0,T;H^2(I))} \le C_{T} (\varepsilon +  \varepsilon_{\mathrm{in}}   +  \delta).
\end{align*}
\end{lemma}

\begin{proof}  By Lemma \ref{Lem.NablaUEnergy} and the estimate \eqref{Lemma.ConvRateULinfL2.State1} in Proposition \ref{Lemma.ConvRateULinfL2},
\begin{equation}
\begin{aligned}
\iint_{I_T} |\partial_{xx}^2 (u^\varepsilon(t)-u(t))|^2 & \le \frac{C_T}{\varepsilon} \left( \iint_{I_T} (u^\varepsilon(t)-u(t))^2 + \iint_{I_T} (v^\varepsilon(t)-v(t))^2 \right) \\
& +  \dfrac{C_T}{\varepsilon} \iint_{I_T} | \partial_{x} (v^\varepsilon(t)-v(t))|^2 + C_T \varepsilon \\
& \le C_T \left( \varepsilon  + \frac{\varepsilon_{\mathrm{in}}^2}{\varepsilon}  +  \frac{\delta^2}{\varepsilon} \right) < \infty,
\end{aligned}
\label{Lem.NablaUEnergy.State1_1}
\end{equation}
where we have used the assumption \eqref{Assumption.Compare.Eps.Del}. This implies that $\partial_{xx}^2 u^\varepsilon$ is uniformly bounded in $L^2(I_T)$ due to the boundedness of $\partial_{xx}^2u$ in $L^2(I_T)$. 
Next, multiplying the equation for $V^\varepsilon$ (in the system \eqref{Problem.DistanceEquation}) by $-\partial_{xx}^2 V^\varepsilon$,  
    \begin{align*}
        & \frac{1}{2} \frac{d}{dt} \int_I |\partial_x (v^\varepsilon(t)-v(t))|^2 + d \int_I |\partial_{xx}^2 (v^\varepsilon(t)-v(t))|^2 \\
        & = - \int_I (\psi(u+U^\varepsilon,v+V^\varepsilon)-\psi(u,v)) \partial_{xx}^2 V^\varepsilon - \delta \int_I \partial_{xx}^2 u^\varepsilon \partial_{xx}^2 (v^\varepsilon(t)-v(t)) \\
        & \le \frac{d}{2} \int_I |\partial_{xx}^2 (v^\varepsilon(t)-v(t))|^2 + C\int_I |\psi(u+U^\varepsilon,v+V^\varepsilon)-\psi(u,v)|^2 + C \delta^2 \int_I |\partial_{xx}^2 u^\varepsilon|^2.  
    \end{align*}
We obtain
\begin{align}
    \begin{aligned}
    & \int_I |\partial_x (v^\varepsilon(t)-v(t))|^2 + \iint_{I_T} |\partial_{xx}^2 (v^\varepsilon(t)-v(t))|^2 \\
    & \le  C_T \left( \delta^2 +  \iint_{I_T} |\psi(u+U^\varepsilon,v+V^\varepsilon)-\psi(u,v)|^2 \right) .  
    \end{aligned}
\label{Lemma.ConvRateULinfH1.Proof3}
\end{align}
Note that, similarly to Lemma \ref{Lem.VEnergy},
\begin{align*}
    &\iint_{I_T} |\psi(u+U^\varepsilon,v+V^\varepsilon)-\psi(u,v)|^2 \\
    & \le \int_0^1 \hspace{-0.15cm} \iint_{I_T} \left( |\psi'_1(u+hU^\varepsilon,v+V^\varepsilon)|^2 (U^\varepsilon)^2 + |\psi'_2(u,v+hV^\varepsilon)|^2 (V^\varepsilon)^2 \right) \\
    & \le C_T  \left( \iint_{I_T} (u^\varepsilon(t)-u(t))^2 + \iint_{I_T} (v^\varepsilon(t)-v(t))^2 \right) 
    \le C_{T} (  \varepsilon^2 +  \varepsilon_{\mathrm{in}}^2 +  \delta^2 )
\end{align*}
by using Proposition \ref{Lemma.ConvRateULinfL2}. Therefore, since $v^\varepsilon(0)-v(0)=0$, we derive 
\begin{align*}
    & \|v^\varepsilon-v\|_{L^\infty(0,T;H^1(I))}  \le C_{T} (\varepsilon + \varepsilon_{\mathrm{in}}  +  \delta ).
\end{align*}
Due to Lemma \ref{Lem.NablaUEnergy},  
\begin{align}
\begin{aligned}
    & \varepsilon \frac{d}{dt} \int_I | \partial_{x} (u^\varepsilon(t)-u(t))|^2 + d \varepsilon \int_I |\partial_{xx}^2 (u^\varepsilon(t)-u(t))|^2 +  \int_I |\partial_{x} (u^\varepsilon(t)-u(t))|^2 \\
    &\le  C_T \left( \int_{I} (u^\varepsilon(t)-u(t))^2 + \int_{I} (v^\varepsilon(t)-v(t))^2 \right) + C_T \int_I | \partial_{x} (v^\varepsilon(t)-v(t))|^2 + C_T \varepsilon^2 \\
    & \le C_T \left( \varepsilon^2 + \varepsilon_{\mathrm{in}}^2  +  \delta^2 \right).
\end{aligned}
\label{Lemma.ConvRateULinfH1.Proof1}
\end{align}
Thus, by making use of the comparison principle similarly to  \eqref{Lemma.ConvRateULinfL2.Proof5}, 
\begin{align}
    \begin{aligned}
    \int_I | \partial_{x} (u^\varepsilon(t)-u(t))|^2  
    & \le C e^{-t/\varepsilon} \varepsilon_{\mathrm{in}}^2  +  C_T \left( \varepsilon^2 + \varepsilon_{\mathrm{in}}^2 +  \delta^2 \right) (1-e^{-t/\varepsilon}) \\
    & \le C_T \left( \varepsilon^2 + \varepsilon_{\mathrm{in}}^2  +  \delta^2 \right),
    \end{aligned}
\label{Lemma.ConvRateULinfH1.Proof2}
\end{align}
which completes the proof.
\end{proof}

Now, we estimate the convergence rate in $L^\infty(0,T;H^2(I))$ by studying the energy functionals for the system  \eqref{Problem.Dx.DistanceEquation}. We first consider the equation for $\partial_x V^\varepsilon$. 

\begin{lemma}
\label{Lemma.ConvRateVLinfH2} There holds for some $C_3>0$
\begin{align}
    \frac{1}{2}\frac{d}{dt} \int_I |\partial_{xx}^2 V^\varepsilon|^2 + \frac{d}{2} \int_I |\partial_{xxx}^3 V^\varepsilon|^2 \le  C_{T} \left( \varepsilon^2 + \varepsilon_{\mathrm{in}}^2  +  \delta^2 \right) + C_3\delta^2 \int_I |\partial_{xxx}^3 U^\varepsilon|^2 . 
    \label{Lemma.ConvRateVLinfH2.State1}
\end{align}
\end{lemma}

\begin{proof} Firstly, multiplying the equation for $\partial_x V^\varepsilon$ in \eqref{Problem.Dx.DistanceEquation} by $-\partial_{xxx}^3 V^\varepsilon$ and using the Neumman boundary conditions \eqref{Condition.Initial.Limiting}, we get
\begin{equation}
\begin{aligned}
    \frac{1}{2} \frac{d}{dt} \int_I |\partial_{xx}^2 V^\varepsilon|^2  
        & = - d \int_I |\partial_{xxx}^3 V^\varepsilon|^2  - \delta  \int_I  \partial_{xxx}^3 u^\varepsilon \partial_{xxx}^3 V^\varepsilon \\
        & - \int_I \partial_x (\psi(u+U^\varepsilon,v+V^\varepsilon)-\psi(u,v)) \partial_{xxx}^3 V^\varepsilon \\
        & \le - \frac{d}{2} \int_I  | \partial_{xxx}^3 V^\varepsilon|^2   + \frac{\delta^2}{d} \int_I  |\partial_{xxx}^3 u^\varepsilon|^2  \\
        & + \frac{1}{d} \int_I |\partial_x (\psi(u+U^\varepsilon,v+V^\varepsilon)-\psi(u,v))|^2 .
\end{aligned}  
\label{Lemma.ConvRateVLinfH2.Proof1}
\end{equation}
Due to the following computations 
\begin{align*}
    & \partial_x (\psi(u+U^\varepsilon,v+V^\varepsilon)-\psi(u,v)) \\
    & = \int_0^1 \partial_x \Big( \psi_1'(u+hU^\varepsilon,v+V^\varepsilon)U^\varepsilon + \psi_2'(u,v+hV^\varepsilon)V^\varepsilon \Big) \\
    & = \int_0^1  \Big( \psi_1'(u+hU^\varepsilon,v+V^\varepsilon) \partial_x U^\varepsilon + \psi_2'(u,v+hV^\varepsilon) \partial_x V^\varepsilon \Big) \\
    & + \int_0^1  \Big( U^\varepsilon \partial_x \psi_1'(u+hU^\varepsilon,v+V^\varepsilon)   + V^\varepsilon \partial_x \psi_2'(u,v+hV^\varepsilon)   \Big),
\end{align*}
and $\psi_{11}'' = 0$, $\psi_{12}''=\psi_{21}''=-b$,   $\psi_{22}''=-2c$, we have  
\begin{align*}
    & \partial_x (\psi(u+U^\varepsilon,v+V^\varepsilon)-\psi(u,v)) \\
    & = \int_0^1  \Big( \psi_1'(u+hU^\varepsilon,v+V^\varepsilon) \partial_x U^\varepsilon + \psi_2'(u,v+hV^\varepsilon) \partial_x V^\varepsilon \Big) \\
    & -   \int_0^1 \Big( b U^\varepsilon (\partial_x v+ \partial_x V^\varepsilon) + b  V^\varepsilon \partial_x u  + 2c  V^\varepsilon (\partial_x v+ h\partial_x V^\varepsilon) \Big).
\end{align*}
Therefore, we can conclude from \eqref{Lem.VEnergy.Proof1}-\eqref{Lem.VEnergy.Proof2} to have  
\begin{align}
     \int_I |\partial_x (\psi(u+U^\varepsilon,v+V^\varepsilon)-\psi(u,v))|^2  
    & \le C_T   \int_I \left( (U^\varepsilon)^2  + (V^\varepsilon)^2 + |\partial_x U^\varepsilon|^2 + |\partial_x V^\varepsilon|^2  \right)  ,
    \label{Lemma.ConvRateVLinfH2.Proof2}
\end{align}
where we employ \eqref{SolutionSmoothness} to   bound the term with $\psi_1',\psi_2'$. 
By Proposition \ref{Lemma.ConvRateULinfL2} and Lemma \ref{Lemma.ConvRateULinfH1},
\begin{align}
     \int_I |\partial_x (\psi(u+U^\varepsilon,v+V^\varepsilon)-\psi(u,v))|^2  
     \le C_{T} \left( \varepsilon^2 + \varepsilon_{\mathrm{in}}^2  +  \delta^2 \right). 
     \label{Lemma.ConvRateVLinfH2.Proof3}
\end{align}
Now, taking in account the boundedness of $\partial_{xxx}^3 u$ in $L^2(I_T)$, we have 
\begin{align}
    \frac{\delta^2}{d} \int_I  |\partial_{xxx}^3 u^\varepsilon|^2 \le \frac{2\delta^2}{d} \int_I  (|\partial_{xxx}^3 U^\varepsilon|^2 + |\partial_{xxx}^3 u|^2)  \le  C\delta^2 \left( 1 + \int_I  |\partial_{xxx}^3 U^\varepsilon|^2  \right).
    \label{Lemma.ConvRateVLinfH2.Proof4}
\end{align}
Finally, plugging \eqref{Lemma.ConvRateVLinfH2.Proof3}-\eqref{Lemma.ConvRateVLinfH2.Proof4} into \eqref{Lemma.ConvRateVLinfH2.Proof1}  gives the estimate \eqref{Lemma.ConvRateVLinfH2.State1}.
\end{proof}

\begin{lemma}
\label{Lemma.ConvRateULinfH2}  There holds 
\begin{align}
    \varepsilon \frac{d}{dt} \int_I |\partial_{xx}^2 U^\varepsilon|^2 + d \varepsilon \int_I |\partial_{xxx}^2 U^\varepsilon|^2 +  \int_I |\partial_{xx}^2 U^\varepsilon|^2    \le C_{T} \left( \varepsilon^2 + \varepsilon_{\mathrm{in}}^2 + \delta^2 \right) + C_T \int_I |\partial_{xx}^2 V^\varepsilon|^2. 
    \label{Lemma.ConvRateULinfH2.State1}
\end{align}
\end{lemma}

\begin{proof} Thanks to the boundary and initial conditions \eqref{Condition.Initial.Limiting}, 
 multiplying the equation for $\partial_x U^\varepsilon$ in \eqref{Problem.Dx.DistanceEquation} by $-\partial_{xxx}^3 U^\varepsilon$ gives 
\begin{align}
\begin{aligned}
\frac{1}{2} \frac{d}{dt} \int_I | \partial_{xx}^2 U^\varepsilon|^2 
 = \,&  - (d+\delta) \int_I |\partial_{xxx}^3 U^\varepsilon|^2 - \int_I  \partial_x ( \phi(u+U^\varepsilon,v+V^\varepsilon)-\phi(u,v)) \partial_{xxx}^3 U^\varepsilon  \\
& - \int_I \partial_x \xi^\varepsilon \partial_{xxx}^3 U^\varepsilon - \dfrac{1}{\varepsilon} \int_I \partial_{x} ( g(u+U^\varepsilon,v+V^\varepsilon)-g(u,v)) \partial_{xxx}^3 U^\varepsilon . 
\end{aligned}
\label{Lemma.ConvRateULinfH2.Proof1}
\end{align}
For the term related to $\phi$, the same techniques as \eqref{Lemma.ConvRateVLinfH2.Proof3} yield  
\begin{align}
     \int_I |\partial_x (\phi(u+U^\varepsilon,v+V^\varepsilon)-\phi(u,v))|^2  
    & \le C_{T} \left( \varepsilon^2 + \varepsilon_{\mathrm{in}}^2  +  \delta^2 \right)  .
    \label{Lemma.ConvRateULinfH2.Proof2}
\end{align}    
To estimate the term including $\partial_x \xi^\varepsilon$, or more explicitly
$$\partial_x \xi^\varepsilon = - \partial_{xt}^2 u + (d+\delta(\varepsilon)) \partial_{xxx}^3 u + \partial_x \phi (u,v),$$
it is useful to note that we do not have the homogeneous Neumann boundary condition for the component $u$ in the system \eqref{Problem.Limiting}. Since the component $u$ is defined by the algebraic equation $-u+(v-u)^2=0$,  we have    
\begin{align}
    u = \frac{2v+1-\sqrt{4v+1}}{2} , 
    \label{Expression.uInTermOfv}
\end{align}
where  $u\le v$ follows from $u^\varepsilon \le v^\varepsilon (= u^\varepsilon+u_2^\varepsilon)$ after passing to the limit as $\varepsilon \to 0$.  We can differentiate  \eqref{Expression.uInTermOfv} to see that   
\begin{align*}
    \partial_t (\partial_x u) = \frac{\sqrt{4v+1}-1}{\sqrt{4v+1}} 
    \partial_t (\partial_x v) + \frac{2}{\sqrt{(4v+1)^3}} \partial_t v \partial_x v,  
\end{align*}
and 
\begin{align*}
    \partial_{xxx}^3 u & = \frac{\sqrt{4v+1}-1}{\sqrt{4v+1}} \partial_{xxx}^3 v + \left( \frac{6}{\sqrt{(4v+1)^3}} \partial_{xx}^2 v  - \frac{12}{\sqrt{(4v+1)^5}} (\partial_x v)^2 \right) \partial_x v \\ 
    & =  \frac{1}{d} \frac{\sqrt{4v+1}-1}{\sqrt{4v+1}} \partial_t (\partial_x v)  - \frac{1}{d} \frac{\sqrt{4v+1}-1}{\sqrt{4v+1}} \left(  \psi_1'(u,v) \frac{\sqrt{4v+1}-1}{\sqrt{4v+1}} + \psi_2'(u,v) \right) \partial_x v   \\
    & \hspace{3.8cm} + \left( \frac{6}{\sqrt{(4v+1)^3}} \partial_{xx}^2 v  - \frac{12}{\sqrt{(4v+1)^5}} (\partial_x v)^2 \right) \partial_x v,
\end{align*}
where we note from differentiating the equation for $v$ that 
\begin{align*}
    \partial_{xxx}^3 v & = \frac{1}{d} \partial_t (\partial_x v)  - \frac{1}{d} \psi_1'(u,v) \partial_x u - \frac{1}{d} \psi_2'(u,v) \partial_x v  \\
    & = \frac{1}{d} \partial_t (\partial_x v)  - \frac{1}{d} \left(  \psi_1'(u,v) \frac{\sqrt{4v+1}-1}{\sqrt{4v+1}} + \psi_2'(u,v) \right) \partial_x v. 
\end{align*}
Using the boundary condition \eqref{Condition.Initial.Limiting} and all the above computation, we see that $\partial_x \xi^\varepsilon$ satisfies the homogeneous Neumann boundary condition, i.e. 
\begin{align*}
    \partial_\nu (\partial_x \xi^\varepsilon) = 0 \quad \text{on } \Gamma \times (0,\infty). 
\end{align*}
Note that, with the regularity \eqref{SolutionSmoothness}, a standard argument can be performed to see that $\partial_{xx}^2 \xi^\varepsilon$ is smooth enough and bounded in $I_T$. Therefore, the term including $\phi$ can be dealt with as follows
\begin{align}
\begin{aligned}
    & - \int_I \partial_x \xi^\varepsilon \partial_{xxx}^3 U^\varepsilon 
    = \int_I \partial_{xx}^2 \xi^\varepsilon \partial_{xx}^2 U^\varepsilon \\
    & \le \frac{1}{4\varepsilon} \int_I |\partial_{xx}^2 U^\varepsilon|^2 + C \varepsilon \int_I |\partial_{xx}^2 \xi^\varepsilon|^2 \le \frac{1}{4\varepsilon} \int_I |\partial_{xx}^2 U^\varepsilon|^2 + C_T \varepsilon. 
\end{aligned}
    \label{Lemma.ConvRateULinfH2.Proof3}
\end{align}

We now proceed to estimate the last term in \eqref{Lemma.ConvRateULinfH2.Proof1}. We observe that  $\partial_{x} ( g(u+U^\varepsilon,v+V^\varepsilon)-g(u,v))$ satisfies the homogeneous Neumann boundary condition. Thus, integration by parts yields
\begin{equation*}
\begin{aligned}
& - \dfrac{1}{\varepsilon} \int_I \partial_{x} ( g(u+U^\varepsilon,v+V^\varepsilon)-g(u,v)) \partial_{xxx}^3 U^\varepsilon \\
& = \frac{1}{\varepsilon} \int_I \partial_{xx}^2 ( g(u+U^\varepsilon,v+V^\varepsilon)-g(u,v)) \partial_{xx}^2 U^\varepsilon .  
\end{aligned} 
\end{equation*}
Since $g_{11}''=-g_{12}''=g_{22}''=2$, we have 
\begin{align*}
& \partial_{xx}^2 ( g(u+U^\varepsilon,v+V^\varepsilon)-g(u,v)) \\
& =  g_1'(u+U^\varepsilon,v+V^\varepsilon) \partial_{xx}^2(u+ U^\varepsilon) + g_2'(u+U^\varepsilon,v+V^\varepsilon) \partial_{xx}^2 (v+V^\varepsilon) \\
& + 2 (\partial_{x}(u+ U^\varepsilon))^2 - 4\partial_{x}(u+ U^\varepsilon) \partial_{x}(v+V^\varepsilon) + 2 (\partial_{x}(v+V^\varepsilon))^2 \\
& - \left( g_1'(u,v) \partial_{xx}^2u + g_2'(u,v) \partial_{xx}^2v + 2 (\partial_{x}u)^2 - 4\partial_{x}u \partial_{x}v + 2 (\partial_{x}v)^2 \right).
\end{align*}
The negative sign suggests to write the above computation as follows
\begin{align*}
    & \partial_{xx}^2 ( g(u+U^\varepsilon,v+V^\varepsilon)-g(u,v)) \\
    & =  g_1'(u+U^\varepsilon,v+V^\varepsilon) \partial_{xx}^2 U^\varepsilon + (g_1'(u+U^\varepsilon,v+V^\varepsilon)-g_1'(u,v)) \partial_{xx}^2 u \\
    & + g_2'(u+U^\varepsilon,v+V^\varepsilon) \partial_{xx}^2 V^\varepsilon + (g_2'(u+U^\varepsilon,v+V^\varepsilon) -g_2'(u,v)) \partial_{xx}^2 v\\
    &+ 2(2\partial_xu +\partial_xU^\varepsilon -2\partial_xv  - \partial_xV^\varepsilon)(\partial_xU^\varepsilon-\partial_xV^\varepsilon).
\end{align*}
Now, with \eqref{Lem.NablaUEnergy.Proof3}, we get 
\begin{align*}
    & \partial_{xx}^2 ( g(u+U^\varepsilon,v+V^\varepsilon)-g(u,v)) \\
    & =  g_1'(u+U^\varepsilon,v+V^\varepsilon) \partial_{xx}^2 U^\varepsilon + g_2'(u+U^\varepsilon,v+V^\varepsilon) \partial_{xx}^2 V^\varepsilon \\
    & +   2 (U^\varepsilon - V^\varepsilon)  \partial_{xx}^2 (u- v)  + 2 (\partial_xU^\varepsilon-\partial_xV^\varepsilon) \partial_x(2u-2v+U^\varepsilon-V^\varepsilon) .
\end{align*}
With the regularity \eqref{SolutionSmoothness} and the definition of the function $g$,  the terms $g_2'(u+U^\varepsilon,v+V^\varepsilon)$ and  $\partial_{xx}^2 (u-v)$ are bounded on $I_T$ (due to Lemma \ref{Lem.Energy}) by a constant depending on $T$. 
Moreover, by the interpolation inequality with the harmonic mean, we have 
\begin{align*}
    \|\partial_{x} U^\varepsilon\|_{L^4(I)}^4 &\le C \|\partial_{xx} U^\varepsilon\|_{L^2(I)}^2 \| U^\varepsilon\|_{L^\infty(I)}^2 \le C \|\partial_{xx} U^\varepsilon\|_{L^2(I)}^2 \| U^\varepsilon\|_{H^1(I)}^2.
\end{align*}
Consequently, we infer from Lemma \ref{Lemma.ConvRateULinfH1} that 
\begin{align*}
    \int_I |\partial_{x} U^\varepsilon|^4 &\le C \| U^\varepsilon\|_{L^\infty(0,T;H^1(I))}^2  \int_I |\partial_{xx} U^\varepsilon|^2 \le C_T (\varepsilon+\varepsilon_{\mathrm{in}} + \delta) \int_I |\partial_{xx} U^\varepsilon|^2 , 
\end{align*}
and similarly, 
\begin{align*}
    \int_I |\partial_{x} V^\varepsilon|^4 \le C_T (\varepsilon+\varepsilon_{\mathrm{in}} + \delta) \int_I |\partial_{xx} V^\varepsilon|^2 . 
\end{align*}
Therefore, by exploiting \eqref{DerivativeOfg},  
\begin{align}
\begin{aligned}
    &- \dfrac{1}{\varepsilon} \int_I \partial_{x} ( g(u+U^\varepsilon,v+V^\varepsilon)-g(u,v)) \partial_{xxx}^3 U^\varepsilon \\
    & \le - \frac{1}{\varepsilon} \int_I |\partial_{xx}^2 U^\varepsilon|^2 + \frac{C_T}{\varepsilon} \int_I \Big( |\partial_{xx}^2 V^\varepsilon| + | U^\varepsilon-V^\varepsilon |   + |\partial_xU^\varepsilon-\partial_xV^\varepsilon| + |\partial_xU^\varepsilon-\partial_xV^\varepsilon|^2 \Big) |\partial_{xx}^2 U^\varepsilon| \\
    &\le - \frac{7}{8\varepsilon} \int_I |\partial_{xx}^2 U^\varepsilon|^2 + \frac{C_T}{\varepsilon} \int_I \left( (U^\varepsilon)^2  + (V^\varepsilon)^2 + |\partial_x U^\varepsilon|^2 + |\partial_x V^\varepsilon|^2 + |\partial_x U^\varepsilon|^4 + |\partial_x V^\varepsilon|^4 + |\partial_{xx}^2 V^\varepsilon|^2  \right) \\
    &  \le - \frac{7}{8\varepsilon} \int_I |\partial_{xx}^2 U^\varepsilon|^2 + C_{T} \left( \varepsilon + \frac{\varepsilon_{\mathrm{in}}^2}{\varepsilon} + \frac{\delta^2}{\varepsilon} \right) + \frac{C_T}{\varepsilon} \left( (\varepsilon+\varepsilon_{\mathrm{in}} + \delta) \int_I |\partial_{xx} U^\varepsilon|^2 \right) + \frac{C_T}{\varepsilon} \int_I |\partial_{xx}^2 V^\varepsilon|^2  \\
    &  \le - \frac{3}{4\varepsilon} \int_I |\partial_{xx}^2 U^\varepsilon|^2 + C_{T} \left( \varepsilon + \frac{\varepsilon_{\mathrm{in}}^2}{\varepsilon} + \frac{\delta^2}{\varepsilon} \right) + \frac{C_T}{\varepsilon} \int_I |\partial_{xx}^2 V^\varepsilon|^2,
\end{aligned}
\label{Lemma.ConvRateULinfH2.Proof4}
\end{align}
where Proposition \ref{Lemma.ConvRateULinfL2} and Lemma \ref{Lemma.ConvRateULinfH1} were used at the last estimate. Finally, \eqref{Lemma.ConvRateULinfH2.State1} is obtained by combining  \eqref{Lemma.ConvRateULinfH2.Proof1}-\eqref{Lemma.ConvRateULinfH2.Proof4}.
\end{proof}

We are ready to prove the estimate \eqref{Lemma.ConvRateUVLinfH2.State1}, Proposition \ref{Lemma.ConvRateULinfL2}, where the assumption  \eqref{Assumption.Compare.Eps.Del} is necessary.

\begin{proof}[Proof of the estimate \eqref{Lemma.ConvRateUVLinfH2.State1}, Proposition \ref{Lemma.ConvRateULinfL2}] Thanks to the assumption \eqref{Assumption.Compare.Eps.Del}, there exists a finitely positive constant $\alpha$ such that $\delta^2 \le \alpha \varepsilon$. Therefore, by Lemmas \ref{Lemma.ConvRateVLinfH2}-\ref{Lemma.ConvRateULinfH2},
\begin{align}
\begin{gathered} 
    \frac{d}{2\alpha C_3}\frac{d}{dt} \int_I |\partial_{xx}^2 V^\varepsilon|^2 + \frac{d^2}{2\alpha C_3} \int_I |\partial_{xxx}^3 V^\varepsilon|^2 \le  C_{T} \left( \varepsilon^2 + \varepsilon_{\mathrm{in}}^2  +  \delta^2 \right) + d\varepsilon \int_I |\partial_{xxx}^3 U^\varepsilon|^2, \\
    \varepsilon \frac{d}{dt} \int_I |\partial_{xx}^2 U^\varepsilon|^2 + d \varepsilon \int_I |\partial_{xxx}^2 U^\varepsilon|^2 +  \int_I |\partial_{xx}^2 U^\varepsilon|^2    \le C_{T} \left( \varepsilon^2 + \varepsilon_{\mathrm{in}}^2 + \delta^2 \right) + C_T \int_I |\partial_{xx}^2 V^\varepsilon|^2.  
\end{gathered}
\label{TheLastEstimate}
\end{align}
Combining these estimates gives 
\begin{gather*}
    \frac{d}{dt} \left( \frac{d}{2\alpha C_3} \int_I |\partial_{xx}^2 V^\varepsilon|^2 + \varepsilon \int_I |\partial_{xx}^2 U^\varepsilon|^2 \right) \le  C_{T} \left( \varepsilon^2 + \varepsilon_{\mathrm{in}}^2  +  \delta^2 \right) +  C_T \int_I |\partial_{xx}^2 V^\varepsilon|^2,  
\end{gather*}
and then by integrating over time, for all $t\in (0,T)$,
\begin{gather*}
    \int_I |\partial_{xx}^2 V^\varepsilon(t)|^2 \le  C  \varepsilon \int_I |\partial_{xx}^2 U^\varepsilon(0)|^2 +  C_{T} \left( \varepsilon^2 + \varepsilon_{\mathrm{in}}^2  +  \delta^2 \right) +  C_T \int_0^t\int_{I} |\partial_{xx}^2 V^\varepsilon|^2 . 
\end{gather*}
This yields \eqref{Lemma.ConvRateUVLinfH2.State1} for the second component by using Lemmas \ref{Lemma.InitialLayer} and \ref{Lemma.ConvRateULinfH1}. The proof for the first component is directly obtained by using the second estimate in \eqref{TheLastEstimate} and the comparison principle.
\end{proof}

\section{Approximation of Slow Manifolds in the Linear Case}\label{section:5}
\subsection{Analysis and Geometry of the Linear System}

The goal of this section is twofold. On the one hand, we study a linear cross diffusion system and construct an (approximate) slow manifold explicitly in Fourier space similar to \cite{kuehn2024infinite}. On the other hand we present a result that relates the abstract slow manifolds obtained by the Lyapunov-Perron method to the Galerkin manifold obtained via the Fourier coefficients as done in \cite{engel2021connecting}.

The linear system expressed in $u_1,u_2$-variables has the form
\begin{align}\label{linear system u1u2}
    \begin{split}
        \partial_t u_1&= (d+\delta)\Delta u_1+\frac{1}{\varepsilon}(u_2-u_1),\\
        \partial_t u_2&= d\Delta u_2-\frac{1}{\varepsilon}(u_2-u_1),\\
        u_1(0)&=u_{1,0},\quad u_2(0)=u_{2,0},
    \end{split}
\end{align}
where the domain is $[0,2\pi]$ and the initial data is in some Hilbert space $H$ with suitable boundary conditions.
Then, setting $u_1=u$ and $u_1+u_2=v$ yields
\begin{align}
    \label{linear ystem uv}\begin{split}
          \partial_t u&= (d+\delta)\Delta u+\frac{1}{\varepsilon}(v-2u),\\
        \partial_t v&= d\Delta v +\delta \Delta u,\\
        u(0)&=u_{0},\quad v(0)=v_{0}.
    \end{split}
\end{align}
In this case, the limit system when $\varepsilon\to 0$ is given by
\begin{align}
    \label{linear system limit uv}\begin{split}
        2u^0&=v^0,\\
        \partial_t v^0&= d\Delta v^0 +\delta \Delta u^0,\\
        v^0(0)&=v_{0}^0.
    \end{split}
\end{align}
This system then can be reduced to a system evolving on the critical manifold 
\begin{align}\label{critical manifold linear}
C_0:=\{ (u,v)\in H\times H:~ u=\frac{1}{2}v\}.
\end{align}
and $v^0$ solves
\begin{align}
    \partial_t v^0&= (d+\frac{1}{2}\delta)\Delta v^0,\quad  v^0(0)=v_{0}^0.
\end{align}

To show the existence of solutions and to construct the slow manifold, we work on the space of Fourier coefficients.
To this end, let $u= \sum_k u_k(t)e_k(x)$ and $v= \sum_k v_k(t)e_k(x)$, where $e_k(x)$ are eigenfunctions of the Laplacian $\Delta$ in $[0,2\pi]$ and $\textnormal{span}_{k}(e_k(x))=H$.

Then, the linear PDE system \eqref{linear ystem uv} reduces to an ODE system for the coefficients of the form
\begin{align}
   \label{linear Fourier system uv}\begin{split}
       u_k'(t)&=-[(d+\delta)k^2-2\varepsilon^{-1}]u_k(t) +\varepsilon^{-1}v_k(t),\\
       v_k'(t) &=-\delta k^2 u_k(t) -d k^2 v_k(t),\\
       u_k(0)&=u_{k,0},\quad v_k(0)= v_{k,0},
   \end{split} 
\end{align}
where $'$ denotes the time derivative.
In the same way we obtain an ODE for the limiting system \eqref{linear system limit uv} 
\begin{align}\label{linear Fourier limit uv}
    \begin{split}
        2u_k^0(t)&=v_k^0(t),\\
        (v_k^0)'(t) &=-\delta k^2 u_k^0(t) -d k^2 v_k^0(t),\\
        v_k^0(0)&= v_{k,0}^0.
    \end{split}
\end{align}
and the critical manifold for the $k$-th Fourier coefficients reads
\begin{align}\label{critical manifold linear Fourier}
    C_0^{\mathcal{F},k}:=\{(x,y)\in \mathbb{R}^2:~x =\frac{1}{2}y\}.
\end{align}

We set $\Omega_\varepsilon= (\varepsilon^2\delta^2 k^4+4)^{1/2}$ and $w_\varepsilon^\pm= \pm \Omega_\varepsilon-\varepsilon(2d+\delta)k^2-2$.
Then, the solutions to the $\varepsilon$-dependent and limit system are given by
\begin{align}
    \label{solution Fourier linear}\begin{split}
      u_k(t) &= \frac{1}{2}(\Omega_\varepsilon)^{-1}\bigg((\Omega_\varepsilon-\varepsilon\delta k^2-2) \textnormal{e}^{\frac{t w_\varepsilon^+}{2 \varepsilon}} +(\Omega_\varepsilon+\varepsilon\delta k^2+2)\textnormal{e}^{\frac{t w_\varepsilon^-}{2\varepsilon}} \bigg) u_{k,0}\\
      &\quad + (\Omega_\varepsilon)^{-1} \bigg(\textnormal{e}^{\frac{t w_\varepsilon^+}{2 \varepsilon}}-\textnormal{e}^{\frac{t w_\varepsilon^-}{2 \varepsilon}}\bigg)v_{k,0},\\
      v_k(t)&= (\Omega_\varepsilon)^{-1} \varepsilon\delta k^2 \bigg(\textnormal{e}^{\frac{t w_\varepsilon^-}{2 \varepsilon}}-\textnormal{e}^{\frac{t w_\varepsilon^+}{2 \varepsilon}}\bigg)u_{0,k}\\
      &\quad + \frac{1}{2}(\Omega_\varepsilon)^{-1}\bigg( (\Omega_\varepsilon+\varepsilon\delta k^2+2) \textnormal{e}^{\frac{t w_\varepsilon^+}{2 \varepsilon}} - (-\Omega_\varepsilon +\varepsilon\delta k^2 +2)\textnormal{e}^{\frac{t w_\varepsilon^-}{2\varepsilon}} \bigg) v_{k,0}
    \end{split}
\end{align}
and 
\begin{align}
    v_k^0(t)= \textnormal{e}^{-\frac{t}{2}(2d+\delta)k^2} v_{k,0}^0,
\end{align}
respectively.

To construct the slow manifold we now only consider the Fourier modes satisfying the condition $\varepsilon \delta k^2\ll 1$, for fixed $\varepsilon$.
We observe that in this regime 
\begin{align*}
    \Omega_\varepsilon= 2 +\frac{\varepsilon^2\delta^2 k^4}{4}+\mathcal{O}(\varepsilon^4).
\end{align*}
Thus, 
\begin{align*}
    \frac{t w_\varepsilon^+}{2\varepsilon} = \frac{t}{2}\big((2d+\delta) k^2+\frac{1}{4}\varepsilon\delta^2k^4 + \mathcal{O}(\varepsilon^3)\big)
\end{align*}
and similarly 
\begin{align*}
    \frac{t w_\varepsilon^-}{2\varepsilon} = -\frac{t}{2\varepsilon}\big(-4 -\varepsilon(2d+\delta)k^2 -\varepsilon^2\delta^2k^4 +\mathcal{O}(\varepsilon^4)\big).
\end{align*}
We note that the solutions, in this regime, have a slow component and a fast component, i.e. we refer to the term $\exp\big( \frac{t w_\varepsilon^+}{2\varepsilon}\big)$ as slow and to the term $\exp\big( \frac{t w_\varepsilon^-}{2\varepsilon} \big)$ as fast.
Assuming that
\begin{align}
    u_{k,0}=\frac{2}{\Omega_\varepsilon+\varepsilon\delta k^2 +2}v_{k,0},
\end{align}
then, the solution of \eqref{solution Fourier linear} evolves only on the slow time scale.
In particular, we observe that
\begin{align}
    u_k(t)= \frac{2}{\Omega_\varepsilon+\varepsilon\delta k^2 +2}v_k(t) 
\end{align}
for all $t\geq 0$ and all Fourier modes $k=1,\dots,K$ satisfying $\varepsilon \delta k^2\ll 1$.
Hence, the slow manifold for the first $k\in [0,K]$ Fourier coefficient is given by
\begin{align}
    S_\varepsilon^{\mathcal{F},k}=\{(x,y)\in \mathbb{R}^2:~x= \frac{2}{\Omega_\varepsilon+\varepsilon\delta k^2 +2}y\}.
\end{align}
We observe that it is indeed an invariant manifold and that it is indeed close to the critical manifold for the first $K$ Fourier coefficients.    

\begin{remark}
 We can obtain the slow manifold in the Hilbert space $H$ by introducing the space of the slow variable $H_s$
 \begin{align*}
     H_s=\{v\in H: x=\sum_{k=1}^K v_k e_k(x)\}
 \end{align*}
 and then 
 \begin{align}
     S_\varepsilon^K=\{ \big( h^\varepsilon(v_s),v_s\big),~ v_s\in H_s\}, 
 \end{align}
 where $h^\varepsilon(v_s)= 2\big( (\varepsilon^2\delta^2\Delta^2+4)^{1/2}-\varepsilon\delta\Delta +2\big)^{-1}$.

When we compare this result to the one obtained in Section 2, we note that in this description the part of the slow manifold for the fast component of $v$ is missing. Thus, it cannot be the exact slow manifold obtained in the previous sections.
\end{remark}

\begin{remark}
    Note, that in the linear case we do not require the cross-diffusion parameter $\delta$ to depend on $\varepsilon$ as we can find an explicit solution to systems\eqref{linear ystem uv} and \eqref{linear system limit uv} via the Fourier transform.
\end{remark}

\subsection{Approximate Slow Manifolds}

Here, we present an approach to construct the abstract slow manifold obtained in Section \ref{section:3}.
The idea is to derive an approximate slow manifold via a Galerkin approach and show that these Galerkin manifolds are close to the exact slow manifold.
This idea follows from adapting the method for the linear system of the previous section to a fast-reaction system of the form
\begin{align}\label{fast reaction system }
   \begin{split}
       \partial_t u& = (d+\delta) \Delta u-\frac{1}{\varepsilon}(u-f(u,v))+ \phi(u,v),\\
       \partial_t v &= d \Delta v+\delta\Delta u+\psi(u,v),\\
       u(0)&=u_0,\quad v(0)=v_0.
   \end{split}
\end{align}
Assuming that the underlying Banach space $X$ allows for a Fourier expansion of 
\begin{align*}
    u(x,t)= \sum_{k\in \mathbb{Z}} u_k(t)e_k(x),\quad  v(x,t)= \sum_{k\in \mathbb{Z}} v_k(t)e_k(x),
\end{align*}
where $\{e_k(x)\}_{k\in\mathbb{Z}}$ is a basis of eigenfunctions of $X$ solving the eigenvalue problem 
$$\Delta e_k(x)= \lambda_k e_k(x).$$
Taking the inner product of \eqref{fast reaction system } with each $e_k$ yields a system of Galerkin ODEs
\begin{align}\label{fast reaction Galerkin system }
   \begin{split}
       \frac{\textnormal{d}}{\textnormal{d}t} u& = (d+\delta) \lambda_k u_k-\frac{1}{\varepsilon}(u_k-\langle f(u,v),e_k\rangle) +\langle\phi(u,v),e_k\rangle,\\
       \frac{\textnormal{d}}{\textnormal{d}t} v &= d  \lambda_k v_k +\delta\lambda_k u_k +\langle\psi(u,v),e_k\rangle,\\
       u_k(0)&=u_{k,0},\quad v_k(0)=v_{k,0}.
   \end{split}
\end{align}
Truncating at a $k\in \mathbb{Z}$ such that  $|k|<k_0(\zeta)= N_G^\zeta$ reduces the equations to a finite ODE system for which the classical Fenichel theory can be applied to.
In addition, we can relate this truncation to the abstract framework by observing that it corresponds to a splitting of both variables into a fast and slow part, where we, in the end, only consider the evolution of the slow parts.
Hence, we obtain
\begin{align}\label{Slow variables Galerkin system}
    \begin{split}
        \partial_t u^G &= (d+\delta) \Delta u^G-\frac{1}{\varepsilon}(u^G-\pr_{X_S^\zeta} f(u^G,v^G)) + \pr_{X_S^\zeta}\phi(u^G,v^G),\\
        \partial_t v^G &= d \Delta v^G+\delta\Delta  u^G+\pr_{X_S^\zeta}\psi(u^G,v^G),\\
       u^G(0)&=\pr_{X_S^\zeta}u_0,\quad v^G(0)=\pr_{X_S^\zeta}v_0,
    \end{split}
\end{align}
where the space $X_S^\zeta$ is given as the linear span of the eigenfunctions associated with the truncation $N_G^\zeta$.
By the finite dimensional fast-slow theory we obtain that \eqref{Slow variables Galerkin system} has a family of slow manifolds $G_{\varepsilon,\zeta}$ given by 
\begin{align*}
    G_{\varepsilon,\zeta}:=\{(h^{\varepsilon,\zeta}_G(v),v): v\in X_S^\zeta\}.
\end{align*}
The key question now is to understand the relation between these approximate Galerkin manifolds and the slow manifolds $S_{\varepsilon,\zeta}$ of Section \ref{section:3}.

\begin{lemma}
    Let the assumptions of Section \ref{Sec:Assumptions} hold and suppose that slow variable space $X_S^\zeta$ has a finite dimensional basis, where the dimension depends inversely on the parameter $\zeta$.
    Then, for all $v_0\in X_S^\zeta$ we have
    \begin{equation}
        \| h^{\varepsilon,\zeta}_u (v_{0,S})-h_G^{\varepsilon,\zeta}(v_{0,S})\|_{X_1} +\|h_{Y_F^\zeta}^{\varepsilon,\zeta}(v_{0,S})\|_{X_1} \leq C\bigg(\varepsilon+\frac{\delta +\varepsilon}{\varepsilon(N_S^\zeta-N_F^\zeta)}\bigg) \|v_{0,S}\|_{X_1}.
    \end{equation}
\end{lemma}

\begin{proof}
    The idea of this proof follows along the lines of Proposition \ref{prop:distance to crit manifold}, where we estimated the difference between the critical manifold and the slow manifold.
    Hence, replacing the graph of the critical manifold $h^0$ with the graph of the Galerkin manifold $h^{\varepsilon,\zeta}_G$ and using the variation of constants formula to express the solution of \eqref{Slow variables Galerkin system} as an integral identity yields the desired estimate. 
    
    More details can be found in \cite{engel2021connecting}, where the approximate Galerkin manifolds were studied for a fast-slow system, where the system has the form $\partial_t u=\frac{1}{\varepsilon}(A u+f(u,v))$ and $\partial_t v= Bv +g(u,v)$.
\end{proof}

\begin{remark}
    We want to emphasize that the parameter $\zeta$ in the Galerkin approximation still depends on $\varepsilon$ through the relation $\varepsilon \zeta^{-1}<1$ that is crucial in the existence of a slow manifold and therefore we cannot take the limit $\zeta$ to zero independently form $\varepsilon$.
    
\end{remark}





\section{Technical Results}
In this section we present the proofs of several technical results that were skipped in the previous sections.
\subsection{Proof of Proposition \ref{prop:distance to crit manifold}}\label{sec.proof of prop distance}

Before we can prove Proposition \ref{prop:distance to crit manifold} we need to additional results.
\begin{prop}\label{prop:operator estimate}
  We have the following two estimates
  \begin{align*}
      \big\|\big(\varepsilon(d+\delta)\Delta -Id\big)^{-1}\big\|_{\mathcal{B}(X_\beta,X_\alpha)}&\leq \begin{cases}
          1 &\textnormal{if}~ \alpha\leq \beta,\\
             \varepsilon^{2(\beta-\alpha)} &\textnormal{if}~ \alpha > \beta
      \end{cases}\\
       \big\|\big(\varepsilon(d+\delta)\Delta -Id\big)^{-1}+Id\big\|_{\mathcal{B}(X_\beta,X_\alpha)}&\leq \varepsilon^{2(\beta-\alpha)} .
  \end{align*}
\end{prop}
\begin{proof}
    Assume for simplicity that $X$ is a Hilbert space and let $\lambda_n$ denote the n-th eigenvalue of $\Delta$ with corresponding eigenfunction $e_n$, with $\lambda_n\leq 0$ for all $n\in \mathbf{N}$.
    Then, we can write $x=\sum_n \lambda_n \langle e_n,x\rangle$ and hence
    \begin{align*}
         \big\|\big(\varepsilon(d+\delta)\Delta -Id\big)^{-1} x\big\|_{X_\alpha}^2& \leq \sum_n \frac{\lambda_n^{2\alpha}}{(\varepsilon(d+\delta)\lambda_n-1)^2}  \langle e_n,x\rangle^2\\
         &\leq  \sum_n \frac{\lambda_n^{2(\alpha-\beta)}}{(\varepsilon(d+\delta)\lambda_n-1)^2} \lambda_n^{2\beta} \langle e_n,x\rangle^2\\
         &\leq \begin{cases}
             \|x\|_{X_\beta}^2 &\textnormal{if}~ \alpha\leq \beta,\\
             \varepsilon^{2(\beta-\alpha)} \|x\|_{X_\beta}^2 &\textnormal{if}~ \alpha > \beta.
         \end{cases}
    \end{align*}
    Similarly for $\alpha\leq \beta$ we estimate
    \begin{align*}
         \big\|\big(\big(\varepsilon(d+\delta)\Delta -Id\big)^{-1}+Id\big) x\big\|_{X_\alpha}^2& \leq \sum_n \frac{(\varepsilon(d+\delta)\lambda_n)^2}{(\varepsilon(d+\delta)\lambda_n-1)^2}  \lambda_n^{2\alpha}\langle e_n,x\rangle^2\\
         &\leq \sum_n \frac{(\varepsilon(d+\delta)\lambda_n)^2\lambda_n^{2(\alpha-\beta)}}{(\varepsilon(d+\delta)\lambda_n-1)^2}  \lambda_n^{2\beta}\langle e_n,x\rangle^2\\
         &\leq \varepsilon^{2(\beta-\alpha)} \|x\|_{X_\beta}^2 .       
    \end{align*}
    
\end{proof}

\begin{prop}\label{prop:operator projection estimate}
    It holds that 
    \begin{align*}
        \big\|\pr_{X_S^\zeta} \delta\Delta x\big \|_{X_1}&\leq \delta C(\zeta^{-1}) \|\pr_{X_S^\zeta} x\|_{X_1}\\
        \big \|\pr_{X_S^\zeta} \Delta \big(\varepsilon(d+\delta)\Delta -Id\big)^{-1} x\big\|_{X_1}&\leq   \|\pr_{X_S^\zeta} x\|_{X_1}.
    \end{align*}
\end{prop}
\begin{proof}
    As in the previous proof we assume for simplicity that $X$ is a Hilbert space.
    Since we assume that the Laplacian $\Delta$ has a spectral gap of size $N_S^\zeta-N_F^\zeta$ the splitting in the fast and slow component can be done in such a way that $\pr_{X_S^\zeta}$ projects onto the space $X_S^\zeta$ spanned by the eigenvalues of $\Delta$ satisfying $|\lambda_n|\leq \zeta^{-1}$.
    Then, we estimate
    \begin{align*}
         \big\|\pr_{X_S^\zeta} \delta\Delta x\big \|_{X_1}^2 &\leq \pr_{X_S^\zeta} \delta^2 \sum_n \lambda_n^2 \lambda_n^{2}\langle e_n,x\rangle^2\\
         &\leq \delta^2 \sum_n^{|\lambda_n|\leq \zeta^{-1}} \lambda_n^2 \lambda_n^{2}\langle e_n,x\rangle^2\\
         &\leq \delta^2 \zeta^{-2} \sum_n^{|\lambda_n|\leq \zeta^{-1}} \lambda_n^{2}\langle e_n,x\rangle^2\\
         &= \delta^2 \zeta^{-2} \|\pr_{X_S^\zeta} x\|_{X_1}^2
    \end{align*}
    and
    \begin{align*}
        \big \|\pr_{X_S^\zeta} \Delta \big(\varepsilon(d+\delta)\Delta -Id\big)^{-1} x\big\|_{X_1}^2 &\leq \pr_{X_S^\zeta} \sum_n \frac{\lambda_n^2}{(\varepsilon(d+\delta)\lambda_n-1)^2} \lambda_n^{2}\langle e_n,x\rangle^2\\
        &\leq  \sum_n^{|\lambda_n|\leq \zeta^{-1}} \lambda_n^{2}\langle e_n,x\rangle^2.
    \end{align*}
\end{proof}

\begin{proof}[Proof of Proposition \ref{prop:distance to crit manifold}]
     Let $(u^*,v_F^*,v_S^*)\in C_\eta$ be the unique fixed point of $ \mathcal{L}_{v_0,\varepsilon,\zeta}$, i.e.
    \begin{align*}
        (u^*,v_F^*,v_S^*)=(h^{\varepsilon,\zeta}_u(v_S^*),h^{\varepsilon,\zeta}_{X_F^\zeta}(v_S^*),v_S^*).
    \end{align*}
By the previous result we have shown that $(u^*,v_F^*,v_S^*)$ solves\eqref{general system split} on $(-\infty,0]$. Furthermore, we have that $v_S^*\in C^1((-\infty,0],\textnormal{e}^{\eta t}; X)$ and the estimate
\begin{align*}
    \sup_{t\leq 0} \textnormal{e}^{-\eta t}\big(\|v_S^*(t)\|_{X_1}+\|\partial_t v_S^*(t)\|_X\big)\leq C \|v_0\|_{X_1}
\end{align*}
holds.
Then we estimate 
\begin{align*}
  &  \|h^{\varepsilon,\zeta}_{X_F^\zeta}(v_S^*(t))\|_{X_1}= \\
  &=\bigg\|    \int_{-\infty}^t \delta A \textnormal{e}^{d A (t-s)}\bigg( \int_0^{(t-s)} \textnormal{e}^{(\delta A-\frac{1}{\varepsilon}Id) r} \, \textnormal{d}r  \bigg)  \varepsilon^{-1}   f(h^{\varepsilon,\zeta}_u(v_S^*(s)), h^{\varepsilon,\zeta}_{X_F^\zeta}(v_S^*(s)),v_S^*(s))\, \textnormal{d}s \\
  &\qquad +\int_{-\infty}^t \delta A \textnormal{e}^{d A (t-s)}\bigg( \int_0^{(t-s)} \textnormal{e}^{(\delta A-\frac{1}{\varepsilon}Id) r} \, \textnormal{d}r  \bigg)  \phi(h^{\varepsilon,\zeta}_u(v_S^*(s)), h^{\varepsilon,\zeta}_{X_F^\zeta}(v_S^*(s)),v_S^*(s))\, \textnormal{d}s \\
  &\qquad + \int_{-\infty}^t \textnormal{e}^{d A (t-s)} \pr_{X_F^\zeta}\psi(h^{\varepsilon,\zeta}_u(v_S^*(s)), h^{\varepsilon,\zeta}_{X_F^\zeta}(v_S^*(s)),v_S^*(s))\big)\, \textnormal{d}s \bigg\|_{X_1}\\
  &\leq \big(C_A^2 \delta d^{-1}(L_f+\varepsilon L_\phi)+ C_A \varepsilon L_\psi \big) \int_{-\infty}^t  \varepsilon^{-1} \textnormal{e}^{(N_F^\zeta +\zeta^{-1}(d \omega_A-1)-\eta) (t-s) }   \, \textnormal{d}s \textnormal{e}^{\eta t}\|(h^{\varepsilon,\zeta}_{u}(v_S^*),h^{\varepsilon,\zeta}_{X_F^\zeta}(v_S^*),v_S^*)\|_{C_\eta}\\
  &\leq \frac{\big(C_A^2\delta  d^{-1}(L_f+\varepsilon L_\phi)+ C_A \varepsilon L_\psi \big)}{\varepsilon|N_F^\zeta +\zeta^{-1}(d \omega_A-1)-\eta|}\textnormal{e}^{\eta t} \|v_0\|_{X_1}.
\end{align*}
Setting $t=0$ we obtain 
\begin{align*}
     \|h^{\varepsilon,\zeta}_{X_F^\zeta}(v_0)\|_{X_1} \leq C\frac{\delta+\varepsilon}{\varepsilon(N_S^\zeta-N_F^\zeta)} \|v_0\|_{X_1}.
\end{align*}
For the remaining component we compute
\begin{align*}
   &h^{\varepsilon,\zeta}_{u}(v_S^*(t))-h^0(v_S^*(t)) =\\
    &=\int_{-\infty}^t  \textnormal{e}^{\varepsilon^{-1}(\varepsilon(d+\delta) A -Id) (t-s)}\big(\varepsilon^{-1}f(h_u^{\varepsilon,\zeta  }(v_S^*(s)),h^{\varepsilon,\zeta}_{X_F^\zeta}(v_S^*(s)),v_S^*(s))+ \phi(h_u^{\varepsilon,\zeta  }(v_S^*(s)),h^{\varepsilon,\zeta}_{X_F^\zeta}(v_S^*(s)),v_S^*(s)) \big)\, \textnormal{d}s\\
    &\quad-h^0(v_S^*(t))\\
    &=\int_{-\infty}^t  \textnormal{e}^{\varepsilon^{-1}(\varepsilon(d+\delta) A -Id) (t-s)}\big(\varepsilon^{-1}f(h_u^{\varepsilon,\zeta  }(v_S^*(s)),h^{\varepsilon,\zeta}_{X_F^\zeta}(v_S^*(s)),v_S^*(s))+ \phi(h_u^{\varepsilon,\zeta  }(v_S^*(s)),h^{\varepsilon,\zeta}_{X_F^\zeta}(v_S^*(s)),v_S^*(s)) \big)\, \textnormal{d}s \\
    &\quad -f(h^0(v_S^*(t),0,v_S^*(t)) - \big(\varepsilon(d+\delta) A -Id\big)^{-1}f(h^0(v_S^*(t),0,v_S^*(t))+\big(\varepsilon(d+\delta) A -Id\big)^{-1}f(h^0(v_S^*(t),0,v_S^*(t))
\end{align*}
Splitting the integral into two parts, where $-\infty<t_0< t\leq 0$ and using integration by parts on the last term we obtain
\begin{align*}
   &=\int_{-\infty}^t  \textnormal{e}^{\varepsilon^{-1}(\varepsilon(d+\delta) A -Id) (t-s)}\phi(h_u^{\varepsilon,\zeta  }(v_S^*(s)),h^{\varepsilon,\zeta}_{X_F^\zeta}(v_S^*(s)),v_S^*(s)) \, \textnormal{d}s\\
   &\quad +\int_{-\infty}^{t_0} \varepsilon^{-1} \textnormal{e}^{\varepsilon^{-1}(\varepsilon(d+\delta) A -Id) (t-s)}f(h_u^{\varepsilon,\zeta  }(v_S^*(s)),h^{\varepsilon,\zeta}_{X_F^\zeta}(v_S^*(s)),v_S^*(s))\, \textnormal{d}s\\
   &\quad + \int_{t_0}^{t} \varepsilon^{-1} \textnormal{e}^{\varepsilon^{-1}(\varepsilon(d+\delta) A -Id) (t-s)}\big(f(h_u^{\varepsilon,\zeta  }(v_S^*(s)),h^{\varepsilon,\zeta}_{X_F^\zeta}(v_S^*(s)),v_S^*(s))-f(h^0(v_S^*(s)),0,v_S^*(s))\big) \, \textnormal{d}s\\
   &\quad + \int_{t_0}^{t} \textnormal{e}^{\varepsilon^{-1}(\varepsilon(d+\delta) A -Id) (t-s)}\big(\varepsilon(d+\delta) A -Id\big)^{-1} \partial_s f(h^0(v_S^*(s)),0,v_S^*(s)) \, \textnormal{d}s\\
   &\quad+ \textnormal{e}^{\varepsilon^{-1}(\varepsilon(d+\delta) A -Id) (t-t_0)} \big(\varepsilon(d+\delta) A -Id\big)^{-1}f(h^0(v_S^*(t_0)),0,v_S^*(t_0))  \\
   &\quad  - \big(\big(\varepsilon(d+\delta) A -Id\big)^{-1}+Id\big)f(h^0(v_S^*(t)),0,v_S^*(t)).
\end{align*}
Then, we estimate
\begin{align*}
  & \| h^{\varepsilon,\zeta}_{u}(v_S^*(t))-h^0(v_S^*(t)) \|_{X_\alpha}\leq \\
  &\leq C_A  L_\phi  \textnormal{e}^{\eta t} \int_{-\infty}^{t} \textnormal{e}^{((d+\delta) \omega_A -\varepsilon^{-1}-\eta) (t-s)}\, \textnormal{d}s \|(h^{\varepsilon,\zeta}_{u}(v_S^*),h^{\varepsilon,\zeta}_{X_F^\zeta}(v_S^*),v_S^*)\|_{C_\eta}\\
  &\quad + C_A  L_f \varepsilon^{-1}  \textnormal{e}^{\varepsilon^{-1}(\varepsilon(d+\delta) \omega_A -1) (t-t_0) +\eta t_0} \int_{-\infty}^{t_0} \textnormal{e}^{((d+\delta) \omega_A -\varepsilon^{-1}-\eta) (t_0-s)}\, \textnormal{d}s \|(h^{\varepsilon,\zeta}_{u}(v_S^*),h^{\varepsilon,\zeta}_{X_F^\zeta}(v_S^*),v_S^*)\|_{C_\eta}\\
  &\quad + C_A L_f \varepsilon^{-1}  \int_{t_0}^{t}  \textnormal{e}^{((d+\delta) \omega_A -\varepsilon^{-1}) (t-s)} \big(\| h^{\varepsilon,\zeta}_{u}(v_S^*(s))-h^0(v_S^*(s)) \|_{X_\alpha}+\|h^{\varepsilon,\zeta}_{X_F^\zeta}(v_S^*(s))\|_{X_\alpha}\big)    \, \textnormal{d}s          
        \\
  &\quad + C_A \int_{t_0}^{t}  \textnormal{e}^{((d+\delta) \omega_A -\varepsilon^{-1}) (t-s)} \big \|\big(\varepsilon(d+\delta) A -Id\big)^{-1}\big \|_{\mathcal{B}(X,X_\alpha)}\|\partial_s f(h^0(v_S^*(s)),0,v_S^*(s)) \|_X  \, \textnormal{d}s\\
  &\quad + C_A L_f \textnormal{e}^{((d+\delta) \omega_A -\varepsilon^{-1}) (t-t_0)+\eta t_0} \big \|\big(\varepsilon(d+\delta) A -Id\big)^{-1}\big \|_{\mathcal{B}(X_\alpha,X_1)} \|(h^0(v_S^*),0,v_S^*)\|_{C_\eta}\\
  &\quad + L_f \big\|\big(\varepsilon(d+\delta) A -Id\big)^{-1}+Id\big \|_{\mathcal{B}(X_\alpha,X_1)}\textnormal{e}^{\eta t}\|(h^0(v_S^*),0,v_S^*)\|_{C_\eta}.
\end{align*}
Then by Proposition \ref{Lipschitz cont} and Propositions \ref{prop:operator estimate} and \ref{prop:operator projection estimate} the above estimate can further be reduced to
\begin{align*}
    & \| h^{\varepsilon,\zeta}_{u}(v_S^*(t))-h^0(v_S^*(t)) \|_{X_\alpha}\leq \\
     &\leq \frac{C_A L_\phi L}{|(d+\delta) \omega_A -\varepsilon^{-1}-\eta|} \textnormal{e}^{\eta t}\|v_0\|_{X_1}\\
     &\quad + \frac{C_A L_fL}{\varepsilon|(d+\delta) \omega_A -\varepsilon^{-1}-\eta| }  \textnormal{e}^{\varepsilon^{-1}(\varepsilon(d+\delta) \omega_A -1) (t-t_0) +\eta t_0} \|v_0\|_{X_1}\\
     &\quad + C_A L_f \varepsilon^{-1}  \int_{t_0}^{t}  \textnormal{e}^{((d+\delta) \omega_A -\varepsilon^{-1}) (t-s)} \| h^{\varepsilon,\zeta}_{u}(v_S^*(s))-h^0(v_S^*(s)) \|_{X_\alpha}   \, \textnormal{d}s  \\
     &\quad + C_A L_f \varepsilon^{-1}  \int_{t_0}^{t}  \textnormal{e}^{((d+\delta) \omega_A -\varepsilon^{-1}) (t-s)}   \frac{\big(C_A^2\delta  d^{-1}(L_f+\varepsilon L_\phi)+ C_A \varepsilon L_\psi \big)}{\varepsilon|N_F^\zeta +\zeta^{-1}(d \omega_A-1)-\eta|}\textnormal{e}^{\eta s}   \, \textnormal{d}s  \|v_0\|_{X_1} \\
     &\quad +  C_A L_f \textnormal{e}^{\eta t} \varepsilon^{1-\alpha}\frac{1-\textnormal{e}^{((d+\delta) \omega_A -\varepsilon^{-1}) (t-t_0)}}{\varepsilon|(d+\delta) \omega_A -\varepsilon^{-1}|} \sup_{t\leq 0} \textnormal{e}^{-\eta t} \big(\|v_S^* \|_X +\|\partial_t v^*_S\|_X\big)\\
     &\quad + \big( C_A L_f L (1+L_h) \textnormal{e}^{((d+\delta) \omega_A -\varepsilon^{-1}) (t-t_0)+\eta t_0} \varepsilon^{-1}  + L_f L (1+L_h) \varepsilon^{1-\alpha}\textnormal{e}^{\eta t}\big)\|v_0\|_{X_1}.
\end{align*}
We recall that $\eta -(d+\delta)\omega_A +\varepsilon^{-1} >0$ and hence we can let $t_0 \to -\infty$ in the second, fourth and sixth term. 
Hence, we obtain
\begin{align*}
     & \| h^{\varepsilon,\zeta}_{u}(v_S^*(t))-h^0(v_S^*(t)) \|_{X_\alpha}\leq \\
     &\leq \bigg( \frac{C_A L_\phi L}{|(d+\delta) \omega_A -\varepsilon^{-1}-\eta|} +2 \frac{ C_A L_f}{\varepsilon | (d+\delta) \omega_A -\varepsilon^{-1}|}   \frac{\big(C_A^2\delta  d^{-1}(L_f+\varepsilon L_\phi)+ C_A \varepsilon L_\psi \big)}{\varepsilon|N_S^\zeta -N_F^\zeta|} \\
     &\qquad + C_A L_f \varepsilon^{1-\alpha}+L_f L (1+L_h) \varepsilon^{1-\alpha} \bigg)\textnormal{e}^{\eta t}\|v_0\|_{X_1}\\
     &\quad + C_A L_f \varepsilon^{-1}  \int_{t_0}^{t}  \textnormal{e}^{((d+\delta) \omega_A -\varepsilon^{-1}) (t-s)} \| h^{\varepsilon,\zeta}_{u}(v_S^*(s))-h^0(v_S^*(s)) \|_{X_\alpha}   \, \textnormal{d}s. 
\end{align*}
Applying Gronwall's inequality and setting $t=0$ yields
\begin{align*}
    \| h^{\varepsilon,\zeta}_{u}(v_0)-h^0(v_0)) \|_{X_\alpha} &\leq C\bigg( \frac{1}{(1-\varepsilon\zeta^{-1})((d+\delta)\omega_A-\varepsilon^{-1})-\frac{1}{2}(N_S^\zeta+N_F^\zeta)} \\
    &\qquad+ \frac{\delta+\varepsilon}{\varepsilon(N_S^\zeta-N_F^\zeta )} + \varepsilon^{1-\alpha}\bigg)\|v_0\|_{X_1}\\
    &\leq C\bigg( \varepsilon + \frac{\delta+\varepsilon}{\varepsilon(N_S^\zeta-N_F^\zeta )} + \varepsilon^{1-\alpha}\bigg)\|v_0\|_{X_1}.
\end{align*}
\end{proof}

\subsection{Proof of Lemma \ref{lemma.differentiability of slow manifold}}\label{sec.proof of lemma}
\begin{proof}
Let $u_0\in X_1$ and $v_0\in X_S^\zeta\cap X_1$ be given. 
Then, we write 
$$U(\cdot,u_0,v_0)= \big(u(\cdot,u_0,v_0),v_F(\cdot,u_0,v_0),v_S(\cdot,u_0,v_0)\big)\in C_\eta$$
for the fixed point of $\mathcal{L}_{\varepsilon,\zeta}$.
Now for fixed initial data $(u_0,v_0),\, (\tilde u_0,\tilde v_0)\in X_1\times X_S^\zeta\cap X_1$ the idea is to show the existence of the derivative as the best local linear approximation of the graph of the manifold.
We write
\begin{align*}
& U(\tilde u_0,\tilde v_0)- U(\cdot,u_0,v_0)  -\textnormal{T}\big[ U(\tilde u_0,\tilde v_0)- U(\cdot,u_0,v_0) \big]=\\
 &=\begin{pmatrix}
     0\\0\\ \delta \Delta \big(\delta\Delta-\frac{1}{\varepsilon}\big)^{-1}\bigg( \textnormal{e}^{d\Delta t}- \textnormal{e}^{[(d+\delta)\Delta -\frac{1}{\varepsilon}Id]t}\bigg) \pr_{X_S^\zeta} (\tilde u_0-u_0)+ \textnormal{e}^{d\Delta t}(\tilde v_0-v_0) 
 \end{pmatrix}+ I(\tilde u_0,u_0,\tilde v_0,v_0),
\end{align*}
where $\textnormal{T}: C_\eta \to C_\eta,$ and the components of the linear operator are given by
\begin{align*}
    z_1 &\mapsto \left[ t\mapsto \left[\int_{-\infty}^t \textnormal{e}^{[(d+\delta)\Delta-\frac{1}{\varepsilon}\textnormal{Id}](t-s)} \big( \varepsilon^{-1}\textnormal{D} f(U(s,u_0,v_0))+\textnormal{D} \phi(U(s,u_0,v_0))\big)z(s)\, \textnormal{d}s \right]\right],\\
    z_2  &\mapsto \left[ t\mapsto \left[ \int_{-\infty}^t  \delta \Delta \textnormal{e}^{d \Delta (t-s)}\bigg( \int_0^{(t-s)} \textnormal{e}^{(\delta \Delta-\frac{1}{\varepsilon}Id) r} \, \textnormal{d}r  \bigg)  \pr_{X^\zeta_F} \big(\varepsilon^{-1}\textnormal{D} f(U(s,u_0,v_0))\right.\right.\\
    &\qquad\qquad \left.\left. + \textnormal{D} \phi(U(s,u_0,v_0))\big)z(s) \, \textnormal{d}s+ \int_{-\infty}^t \textnormal{e}^{d \Delta (t-s)}  \pr_{X^\zeta_F}\textnormal{D} \psi(U(s,u_0,v_0))z(s) \, \textnormal{d}s \right]\right],\\
    z_3&\mapsto \left[ t\mapsto 0 \right],
\end{align*}
and where $z=(z_1,z_2,z_3)^T$
and the remainder term $I(\tilde u_0,u_0,\tilde v_0,v_0)= \big(I_1, I_2, 0\big)^T$ has the components
\begin{align*}
    I_1(\tilde u_0,u_0,\tilde v_0,v_0)&=\bigg[t\mapsto  \varepsilon^{-1}\int_{-\infty}^t \textnormal{e}^{((d+\delta)\Delta-\frac{1}{\varepsilon}\textnormal{Id})(t-s)}\big(f(U(s,\tilde u_0,\tilde v_0))-f(U(s,u_0,v_0))\\
    &\qquad\qquad\qquad-\textnormal{D} f(U(s,u_0,v_0))[U(s,\tilde u_0,\tilde v_0)-U(s,u_0,v_0)]\big)\, \textnormal{d}s\\
    &\qquad \qquad+\int_{-\infty}^t \textnormal{e}^{((d+\delta)\Delta-\frac{1}{\varepsilon}\textnormal{Id})(t-s)}\big(\phi(U(s,\tilde u_0,\tilde v_0))-\phi(U(s,u_0,v_0))\\
    &\qquad\qquad\qquad-\textnormal{D} \phi(U(s,u_0,v_0))[U(s,\tilde u_0,\tilde v_0)-U(s,u_0,v_0)]\big)\, \textnormal{d}s \bigg]
\end{align*}
and similar for $I_2(\tilde u_0,u_0,\tilde v_0,v_0)$.
The idea is to show that $\|\textnormal{T}\|_{\mathcal{B}(C_\eta)}<1$ and that 
$$\|I(\tilde u_0,u_0,\tilde v_0,v_0)\|_{X_1}= o\big(\|\tilde u_0-u_0\|_{X_1}+\|\tilde v_0-v_0\|_{X_1}\big)\text{  as } (\tilde u_0,\tilde v_0)\to (u_0,v_0).$$
Then, evaluating the expression at $t=0$ we have
\begin{align*}
    U(0,\tilde u_0,\tilde v_0)-U(0,u_0,v_0)=U(0,\tilde v_0)-U(0,v_0)    &= (\textnormal{I}-\textnormal{T})^{-1}\begin{pmatrix}
     0\\0\\ \tilde v_0-v_0
 \end{pmatrix} +o\big(\|\tilde v_0-v_0\|_{X_1}\big)
\end{align*}
as $\tilde v_0\to v_0$, so that
\begin{align*}
    U(0,\cdot)=\big(h_u^{\varepsilon,\zeta},h_{X_F}^{\varepsilon,\zeta},\textnormal{id}_{X_S^\zeta}\big)
\end{align*}
is differentiable.

The first part, showing that $\|\textnormal{T}\|_{\mathcal{B}(C_\eta)}<1$, follows from the fact that the operator $\mathcal{L}_{\varepsilon,\zeta}$ is a contraction in Proposition \ref{existence slow manifold}.

For the second part we estimate each of the components of $I(\tilde u_0,u_0,\tilde v_0,v_0)$ as follows.
By the assumptions on $f$ we have that for all $\sigma>0$ there exists a $N>0$ such that 
\begin{align*}
   \textnormal{e}^{-\eta t} &\bigg\| \varepsilon^{-1}\int_{-\infty}^{\min(-N,t)} \textnormal{e}^{((d+\delta)\Delta-\frac{1}{\varepsilon}\textnormal{Id})(t-s)}\big(f(U(s,\tilde u_0,\tilde v_0))-f(U(s,u_0,v_0))\\
  &\qquad\qquad\qquad -\textnormal{D} f(U(s,u_0,v_0))[U(s,\tilde u_0,\tilde v_0)-U(s,u_0,v_0)]\big)\, \textnormal{d}s \bigg\|_{X_1}\\
  &\leq 2L_f C_\Delta \|U(\cdot,\tilde u_0,\tilde v_0)-U(\cdot,u_0,v_0)\|_{C_\eta}  \varepsilon^{-1}\int_{-\infty}^{\min(-N,t)}\textnormal{e}^{((d+\delta)\omega_\Delta-\frac{1}{\varepsilon})(t-s)}\, \textnormal{d}s\\
  &\leq \frac{\sigma}{2}\big(\|\tilde u_0-u_0\|_{X_1}+\|\tilde v_0-v_0\|_{X_1}\big)
\end{align*}
for all $t\leq 0$ holds.
Now, having fixed such a $N>0$ we obtain
\begin{align*}
     \textnormal{e}^{-\eta t} &\bigg\| \varepsilon^{-1}\int_{\min(-N,t)}^t \textnormal{e}^{((d+\delta)\Delta-\frac{1}{\varepsilon}\textnormal{Id})(t-s)}\big(f(U(s,\tilde u_0,\tilde v_0))-f(U(s,u_0,v_0))\\
  &\qquad\qquad\qquad -\textnormal{D} f(U(s,u_0,v_0))[U(s,\tilde u_0,\tilde v_0)-U(s,u_0,v_0)]\big)\, \textnormal{d}s \bigg\|_{X_1}\\
  &\leq C_\Delta \|U(\cdot,\tilde u_0,\tilde v_0)-U(\cdot,u_0,v_0)\|_{C_\eta} \varepsilon^{-1} \int_{\min(-N,t)}^t \textnormal{e}^{((d+\delta)\omega_\Delta-\frac{1}{\varepsilon})(t-s)}\\
  &\qquad\qquad\qquad \int_0^1 \big\| \textnormal{D}f\big(r U(s,\tilde u_0,\tilde v_0)-(1-r)U(s, u_0, v_0) \big)-\textnormal{D} f(U(s,u_0,v_0))\|_{\mathcal{B}(X_1)}\, \textnormal{d}r\, \textnormal{d}s\\
  &\leq C \big(\|\tilde u_0-u_0\|_{X_1}+\|\tilde v_0-v_0\|_{X_1}\big) \varepsilon^{-1} \int_{\min(-N,t)}^t \textnormal{e}^{((d+\delta)\omega_\Delta-\frac{1}{\varepsilon})(t-s)}\\
  &\qquad\qquad\qquad \int_0^1 \big\| \textnormal{D}f\big(r U(s,\tilde u_0,\tilde v_0)-(1-r)U(s, u_0, v_0) \big)-\textnormal{D} f(U(s,u_0,v_0))\|_{\mathcal{B}(X_1)}\, \textnormal{d}r\, \textnormal{d}s.
\end{align*}
Then, applying the dominated convergence theorem and using the linearity of the integrand it follows that the integral is smaller than $\frac{\sigma}{2C}$ if $(\tilde u_0,\tilde v_0)$ is close enough to $(u_0,v_0)$.
Thus, for all $\sigma>0$ there exists $\tilde \sigma>0$ such that for all initial data $(\tilde u_0,\tilde v_0)\in X_1\times X_1\cap X_S^\zeta$ with $ \big(\|\tilde u_0-u_0\|_{X_1}+\|\tilde v_0-v_0\|_{X_1}\big)<\tilde \sigma$ and all $t\geq 0$ it holds that
\begin{align*}
      \textnormal{e}^{-\eta t} &\bigg\| \varepsilon^{-1}\int_{-\infty}^t \textnormal{e}^{((d+\delta)\Delta-\frac{1}{\varepsilon}\textnormal{Id})(t-s)}\big(f(U(s,\tilde u_0,\tilde v_0))-f(U(s,u_0,v_0))\\
  &\qquad\qquad\qquad -\textnormal{D} f(U(s,u_0,v_0))[U(s,\tilde u_0,\tilde v_0)-U(s,u_0,v_0)]\big)\, \textnormal{d}s \bigg\|_{X_1}\\
  &< \sigma \big(\|\tilde u_0-u_0\|_{X_1}+\|\tilde v_0-v_0\|_{X_1}\big).
\end{align*}
We note that a similar computation can be carried out for the second component of $I_1$ and the expression for $I_2$.
Therefore, we conclude that 
\begin{align*}
  \|I(\tilde u_0,u_0,\tilde v_0,v_0)\|_{C_\eta}=   o\big(\|\tilde u_0-u_0\|_{X_1}+\|\tilde v_0-v_0\|_{X_1}\big)\text{  as } (\tilde u_0,\tilde v_0)\to (u_0,v_0),
\end{align*}
which shows the differentiability of the slow manifold.
\end{proof}

\section*{Acknowledgments}
B.Q. Tang and B.-N. Tran are supported by the FWF project
“Quasi-steady-state approximation for PDE”, number I-5213.
C.K. and J.-E.-S. are supported by DFG grant 456754695.
C.K. would like to thank the VolkswagenStiftung for support via a Lichtenberg Professorship. 
\bibliographystyle{alpha}
\bibliography{lit}







\end{document}